\documentclass[11pt,oneside,english]{amsart}
\usepackage[T1]{fontenc}
\usepackage[latin9]{inputenc}
\usepackage{xcolor}
\usepackage{babel}
\usepackage{mathrsfs}
\usepackage{mathtools}
\usepackage{amstext}
\usepackage{amsthm}
\usepackage{amssymb}
\usepackage[pdfusetitle,
bookmarks=true,bookmarksnumbered=false,bookmarksopen=false,
breaklinks=false,pdfborder={0 0 0},backref=false,colorlinks=false]
 {hyperref}

\makeatletter
\numberwithin{equation}{section}
\numberwithin{figure}{section}
\theoremstyle{plain}
\newtheorem{thm}{\protect\theoremname}[section]
\theoremstyle{definition}
\newtheorem{defn}[thm]{\protect\definitionname}
\theoremstyle{remark}
\newtheorem{rem}[thm]{\protect\remarkname}
\theoremstyle{plain}
\newtheorem{prop}[thm]{\protect\propositionname}
\theoremstyle{plain}
\newtheorem{assumption}[thm]{\protect\assumptionname}
\theoremstyle{plain}
\newtheorem{lem}[thm]{\protect\lemmaname}
\theoremstyle{plain}
\newtheorem{cor}[thm]{\protect\corollaryname}
\usepackage[a4paper]{geometry}
\usepackage{amsfonts}

\usepackage{times}

\allowdisplaybreaks[3]

\newcommand{\dd}{\mathrm d}
\newcommand{\one}{\mathbf 1}
\makeatother

\providecommand{\assumptionname}{Assumption}
\providecommand{\definitionname}{Definition}
\providecommand{\lemmaname}{Lemma}
\providecommand{\remarkname}{Remark}
\providecommand{\theoremname}{Theorem}
\providecommand{\corollaryname}{Corollary}
\providecommand{\propositionname}{Proposition}
\begin{document}
\title[Entropy approach to doubly nonlinear obstacle problem]{An entropy approach to doubly nonlinear parabolic obstacle problems}
\author{Ruoyang Liu}
\thanks{The author was supported by the Natural Science Foundation of Shanghai (Grant No.25ZR1402408).}
\address{Department of Mathematics, Shanghai Normal University, Shanghai, China}
\email{ryliu@shnu.edu.cn}
\begin{abstract}
This paper studies doubly nonlinear parabolic obstacle problems. 
We introduce a renormalized entropy formulation in which the reaction measure is encoded by evaluating the entropy multiplier at the obstacle.
For diffusion depending only on the gradient and continuous time-independent obstacles, we prove a global $L^1$ comparison estimate, uniqueness of the solution, and a minimality principle among a measure-free entropy supersolution class.

For diffusion depending on both the solution and its gradient, we prove existence for continuous time-dependent obstacles with a boundary-controlled decomposition.
Using two ordered penalization procedures, we obtain convergence of the approximate solutions and of the corresponding reaction terms.
The limiting reaction consists of an absolutely continuous part with bounded density and a finite Radon measure concentrated on a prescribed compact spatial region.
A one-sided time regularization yields strong convergence of the gradients and identifies the nonlinear flux.
Together, the results yield existence, uniqueness and $L^1$ stability under their common assumptions.
\end{abstract}

\maketitle
\section{Introduction}

Obstacle problems for doubly nonlinear parabolic equations combine a nonlinear time derivative and nonlinear diffusion with a unilateral state constraint.
Both nonlinearities may be degenerate or singular, while the constraint is enforced by a reaction that may be only a finite Radon measure.
The interaction of these features makes comparison and uniqueness particularly delicate.

Let $D\subset\mathbb R^d$ be a bounded Lipschitz domain, and let $T>0$.
We study the lower obstacle problem
\begin{equation}
\partial_t\beta(u)=\operatorname{div}A(u,\nabla u)+f(x,t,u)+\nu,\qquad u\geq\psi\quad\text{in }D\times(0,T),\label{eq:doubly nonlinear eqn}
\end{equation}
with homogeneous Dirichlet boundary condition and initial value $\xi$.
Here $\beta$ is nondecreasing, and reaction $\nu$ is a nonnegative finite Radon measure enforcing the constraint.
In sufficiently regular obstacle problems, $\nu$ is concentrated on the contact set $\{u=\psi\}$ \cite{pierre1979problems,pierre1980representant,denis2014obstacle}.
In the present setting, the degeneracy or singularity of the equation leads to low regularity of the solution. 
Since $\nu$ may also have a singular part, the value of $u$ on the support of this singular part need not be determined by its Lebesgue representative.
Therefore, the classical contact condition may not in general be imposed directly on the singular part of the reaction measure.

A basic prototype is
\begin{equation}
\partial_t\bigl(|u|^{m-1}u\bigr)=\operatorname{div}\bigl(|\nabla u|^{p-2}\nabla u\bigr)+f(x,t,u)+\nu,\qquad m>0,\quad p>1.\label{eq:p-laplace porous media}
\end{equation}
For $m=1$ this contains the evolutionary $p$-Laplace equation.
When $p=2$, the change of dependent variable through $\beta(u)=|u|^{m-1}u$ leads to a porous-medium/fast-diffusion type equation.
Thus degeneracy or singularity may occur in both the time nonlinearity and the diffusion.

A direct application of the doubly nonlinear obstacle problem arises in shallow ice-sheet dynamics: Calvo, Diaz, Durany, Schiavi and Vazquez \cite{calvo2003ice} derived and analyzed a doubly nonlinear parabolic obstacle problem for non-Newtonian ice flow, with the associated free boundary describing the evolving ice margin.
Closely related unilateral constraints also occur in reflected porous-medium type equations and stochastic nonlinear diffusion obstacle problems \cite{rockner2013stochastic,liu2024obstacle,du2024well}.
More broadly, the underlying doubly nonlinear diffusion class includes porous-medium and filtration models \cite{vazquez2007porous} and Stefan-type phase-transition models \cite{visintin2007stefan}.

The classical theory of parabolic obstacle problems is rooted in variational inequalities and evolution equations with unilateral constraints; see \cite{lions1967variational,brezis1972problemes,mignot1977inequations,alt1983quasilinear}.
Variational and potential-theoretic approaches have also been developed for porous-medium and evolutionary $p$-Laplace equations with irregular obstacles; see, for example, \cite{bogelein2015obstacle,lindqvist2012irregular,bogelein2017parabolic,moring2024two}.
For doubly nonlinear obstacle problems, minimizing-movement and variational methods provide broad existence theories \cite{bogelein2018doubly,schatzler2019existence,schatzler2020obstacle}.
Sch\"atzler \cite{schatzler2020obstacle} identifies uniqueness as an open problem for the variational solution class considered there.

Entropy and renormalized methods yield comparison and uniqueness for obstacle-free doubly nonlinear equations \cite{blanchard1998renormalized,carrillo1999uniqueness,
bogelein2024comparison}.
The obstacle adds a new difficulty: in the doubly entropy inequalities, each solution carries its own reaction measure, and the two measures are neither ordered nor directly comparable.

This paper establishes two complementary results.
For continuous time-independent obstacles and fluxes of the form $A(\nabla u)$, we formulate an obstacle-adapted renormalized entropy solution by evaluating the entropy multiplier in each reaction term at the common obstacle.
After doubling variables, the reaction contributions are controlled by differences in obstacle values, which vanish under spatial localization due to the continuity of $\psi$.
This gives a global $L^1$ comparison principle, $L^1$ stability with respect to the initial value, uniqueness of $\beta(u)$, and, when $\beta$ is strictly increasing, uniqueness of both the function $u$ and the reaction measure $\nu$.
The same one-sided comparison also gives minimality in a natural measure-free entropy supersolution class.

For continuous time-dependent obstacles and fluxes of the form $A(u,\nabla u)$, we prove the existence under a boundary-controlled decomposition of the obstacle.
The two ordered penalization limits we used produce a bounded reaction density and a finite Radon measure.
In the second limit, the measure-valued reaction prevents a direct time pairing with the limit solution.
We overcome it by a one-sided time regularization of a truncation of the limit.
The resulting relative-energy estimate gives strong convergence of the gradients and identifies the nonlinear flux.
The precise decomposition and boundary conditions are stated in Subsection \ref{subsec:existence}.

\subsection{Reaction measures and entropy encoding}

An alternative to the variational formulation is to retain the reaction and seek a pair $(u,\nu)$ satisfying the equation together with a contact or minimality condition.
For a sufficiently regular one-sided problem, the natural complementarity relation is
\begin{equation}
u\geq\psi,\qquad\nu\geq0,\qquad\int_{D\times[0,T)}(u-\psi)\,\dd\nu=0.\label{eq:introduction-complementarity}
\end{equation}
It expresses that the reaction acts only where the constraint is active.
Reaction-measure formulations of this kind have been developed in parabolic potential theory and for equations with measure data; see \cite{pierre1979problems,pierre1980representant,klimsiak2015obstacle,klimsiak2018obstacle}.
In the stochastic literature, the analogous contact relation is usually called the Skorokhod condition; see \cite{matoussi2010obstacle,denis2014obstacle}.

For low-regularity solutions, however, \eqref{eq:introduction-complementarity} need not be meaningful as written.
The solution is initially defined only up to sets of Lebesgue measure zero, whereas the reaction may be singular.
Existing reaction-measure theories therefore interpret contact through suitable quasi-continuous or precise representatives and measures compatible with parabolic capacity.

In the entropy formulation, contact condition \eqref{eq:introduction-complementarity} is encoded in the reaction term of the entropy inequality.
Following the entropy idea of Du and Liu \cite{du2024well}, let $\eta$ be a bounded nondecreasing entropy multiplier and let $\phi\geq0$ be a smooth test function. 
A formal nonlinear test by $\eta(u)\phi$ produces the reaction contribution 
\[
-\int_{D\times[0,T)}\eta(u)\phi\,\dd\nu.
\]
For a singular reaction measure, however, the value of $\eta(u)$ on the singular support of $\nu$ may not be determined by the Lebesgue equivalence class of $u$.
A key observation in \cite{du2024well} is that the continuous obstacle provides a canonical value there.
This motivates encoding the reaction contribution as
\begin{equation}
-\int_{D\times[0,T)}\eta(\psi)\phi\,\dd\nu.\label{eq:introduction-entropy-replacement}
\end{equation}
Under the classical contact relation \eqref{eq:introduction-complementarity}, one has $u=\psi$ $\nu$-a.e., so \eqref{eq:introduction-entropy-replacement} agrees with the formal term involving $\eta(u)$. 

For the present comparison argument, this encoding has a further role.
When the entropy inequalities for two solutions are doubled, the reaction terms can be estimated through differences of the common obstacle, leading to errors of the form $\eta(\psi(y)-\psi(x))$.
These errors are controlled by the modulus of continuity of $\psi$ and vanish under spatial localization, without requiring any ordering of the two reaction measures.
Moreover, $\eta(\psi)\phi$ is a bounded continuous coefficient on the compact support of the test, so the reaction contribution is stable under the weak-* measure limits used in the existence proof.

Reaction measures, obstacle-dependent entropy inequalities, and minimal supersolution principles have appeared previously in related problems; see, for example, \cite{klimsiak2015obstacle,klimsiak2018obstacle,levi2001entropy,akdim2012entropy}.
The new point here is to implement the obstacle evaluation in a renormalized entropy formulation for a doubly nonlinear parabolic equation with a finite Radon reaction, and to show that it is strong enough to compare two solutions whose reaction measures are not ordered.
The same sign structure recovers the associated minimality principle: after the reaction is removed and only the negative one-sided entropy inequalities are retained, the solution with finite reaction is the smallest member of the resulting measure-free supersolution class.

\subsection{Comparison, uniqueness, and minimality}
\begin{thm}[Comparison, uniqueness, and minimality]
\label{thm:introduction-comparison}
Assume the hypotheses of Theorem \ref{thm:main-comparison}.
In particular, let $A=A(\nabla u)$ be continuous and monotone, let $\beta$ be
continuous and nondecreasing, and let the common obstacle satisfy
\[
\psi\in C(\overline D),\qquad\psi\leq0\quad\text{on }\partial D,
\]
and be independent of time.
For $i=1,2$, let $(u_i,\nu_i)$ be renormalized entropy solutions with data $(f_i,\xi_i)$.
If
\[
 f_1(x,t,r)\leq f_2(x,t,r)
\]
for almost every $(x,t)\in D\times[0,T]$ and every $r\in\mathbb R$, then, for almost every $t\in(0,T)$,
\begin{equation}
\int_D\bigl(\beta(u_1(t))-\beta(u_2(t))\bigr)^+\,\dd x\leq\mathrm e^{L_2t}\int_D\bigl(\beta(\xi_1)-\beta(\xi_2)\bigr)^+\,\dd x,\label{eq:introduction-comparison}
\end{equation}
where $L_2$ is a Lipschitz constant for $f_2$ with respect to $\beta(u)$.
Consequently, equal source and initial data imply uniqueness of $\beta(u)$.
If $\beta$ is strictly increasing, then both $u$ and $\nu$ are unique.

Moreover, let $(u,\nu)$ be a solution with finite reaction and let $v$ be a measure-free entropy supersolution with the same obstacle, source, and initial value.
Then
\[
 \beta(u)\leq\beta(v)\quad\text{a.e. in }D_T.
\]
If $\beta$ is strictly increasing, then $u\leq v$ almost everywhere.
\end{thm}

The precise statements are Theorem \ref{thm:main-comparison} and Corollary \ref{cor:minimality-finite-reaction}.
Equal source and initial data and strict monotonicity of $\beta$ then give equality of the functions $u_i$, and the distributional equation identifies the reaction measures $\nu_i$.

The proof combines the doubling method on bounded domain \cite{kruvzkov1970first,otto1996l1,blanchard1998renormalized,carrillo1999uniqueness} with the obstacle reaction estimate.
After the two entropy inequalities are added, the reaction contribution is bounded by the spatial oscillation of the common obstacle. 
Taking the spatial localization limit removes this contribution. 
Inward kernels are considered for the homogeneous Dirichlet boundary and preserve
\[
\nabla_x\rho_{\varsigma,j}+\nabla_y\rho_{\varsigma,j}=0.
\]
The same boundary-adapted kernels also yield the averaged strong initial trace of $\beta(u)$ needed to recover the initial term.
The resulting Kato inequality yields the global $L^1$ estimate.

The minimality result follows from the same one-sided comparison.
The measure-free supersolution has no reaction term, while the reaction term of the solution can be discarded in the entropy inequality using the definition of supersolution.

\subsection{Existence for a decomposed time-dependent obstacle}\label{subsec:existence}
The existence argument does not require the full obstacle to have a bounded equation residual as in \cite{bogelein2015obstacle}.
Instead, we decompose
\begin{equation}
\psi=\psi_1+\psi_2,\label{eq:eq:introduction-obstacle-decomposition}
\end{equation}
so that the residual and boundary conditions are imposed only on $\psi_1$, while $\psi_2$ is required to be nonnegative and spatially localized.
More precisely, $\psi_1$ has a nonpositive boundary trace and a bounded equation residual, whereas $\psi_2$ vanishes outside a compact set $K_2\subset\overline D$.
If $K_2\Subset D$, then $\psi=\psi_1$ near the boundary.
If $K_2$ meets the boundary, the full obstacle is instead assumed to remain strictly below the zero Dirichlet value on $K_2\cap\partial D$.
The precise assumptions are stated in Assumption \ref{assu:existence-obstacle}.

\begin{thm}[Existence for a decomposed time-dependent obstacle]
\label{thm:introduction-existence}
Under the hypotheses of Theorem \ref{thm:existence-general-time-obstacle}, there exists a renormalized entropy solution $(u,\nu)$ with $u\geq\psi$ a.e. in $D_T$.
Moreover,
\[
 \nu=\nu_1\,\dd x\,\dd t+\nu_2,
\]
where $\nu_1\in L^\infty(D_T)$ is nonnegative and $\nu_2$ is a nonnegative finite Radon measure concentrated on $(K_2\cap D)\times[0,T)$.
\end{thm}

The construction uses two ordered penalization limits.  
The first enforces $\psi_1$ and produces a bounded reaction density.
The second further enforces the localized correction $\psi_2$ and produces a finite Radon measure.
The comparison principle orders the penalized families and therefore gives monotone pointwise limits. 
Uniform energy estimates provide the compactness of the solutions and a uniform mass bound for the second penalization term.

Another key point for the proof in the second limit is the identification of the nonlinear flux. 
The reaction measure $\nu_2$ prevents a direct time pairing with the limit solution $u$.
We therefore use a one-sided time regularization \cite{landes1981existence,schatzler2020obstacle,farroni2023noncoercive} of a truncation of the limit solution with parameter $\lambda$.
The regularized comparison function stays above the obstacle up to an error $c_\lambda$.
The associated penalty term is bounded by $c_\lambda$ times the uniform reaction mass and vanishes as $\lambda\downarrow0$.
The relative-energy estimate then gives strong convergence of gradients and identifies the nonlinear flux.
See Lemma \ref{lem:minty-obstacle-reactions} and Remark \ref{rem:time regularization} for details.

For the prototype in \eqref{eq:p-laplace porous media}, the comparison theorem covers $m>0$ and $p>1$, subject to its remaining assumptions.
The existence theorem applies for $m>0$ and $p\geq2$, which is due to the assumption of $p$-monotonicity \eqref{eq:A-uniform-monotonicity} as in Alt and Luckhaus \cite{alt1983quasilinear}.

The scopes of the two main results are complementary.
The existence theorem allows $A=A(u,\nabla u)$ and a decomposed time-dependent obstacle.
The comparison theorem covers the full range $p>1$ and general continuous spatial
obstacles, but requires $A=A(\nabla u)$ and a time-independent obstacle.
Their intersection gives Corollary \ref{cor:wellposedness-common-class}.

Section \ref{sec:entropy formulation} introduces the entropy formulation and states the full comparison theorem.
Section \ref{sec:comparison-dirichlet} proves comparison, the averaged initial trace, and minimality in measure-free supersolutions.
Section \ref{sec:existence-general} develops the two-penalty construction and proves existence.
The appendix records the space--time approximation used in the Alt--Luckhaus scheme.

\section{Entropy formulation and main result}\label{sec:entropy formulation}

Write $D_T:=D\times(0,T)$, $p\in(1,\infty)$ and $p':={p}/({p-1})$.
For a locally compact Hausdorff space $X$, we denote by $\mathcal M(X)$ the space of finite signed Radon measures on $X$, equipped with the total variation norm
\[
\|\mu\|_{\mathcal M(X)}:=|\mu|(X),
\]
and by $\mathcal M_+(X)$ the space of finite nonnegative Radon measures on $X$.
We first consider
\begin{align}
\partial_t\beta(u)&=\operatorname{div}A(\nabla u)+f(x,t,u)+\nu &&\text{in }D_T,\nonumber\\
u&=0&&\text{on }\partial D\times(0,T),\label{eq:ito}\\
u(\cdot,0)&=\xi&&\text{in }D,\nonumber
\end{align}
with the lower obstacle
\begin{equation}
u\geq\psi\quad\text{a.e. in }D_T.\label{eq:lower obstacle}
\end{equation}

\begin{assumption}[Time nonlinearity]\label{assu:assumption for b}
The function $\beta:\mathbb R\to\mathbb R$ is continuous and nondecreasing satisfying $\beta(0)=0$.
\end{assumption}

\begin{assumption}[Diffusion]\label{assu:assumption for a}
The map $A:\mathbb R^d\to\mathbb R^d$ is continuous, $A(0)=0$, and there is $a_1>0$ such that
\begin{align}
|A(\zeta)|&\leq a_1(1+|\zeta|^{p-1}), \label{eq:A-growth-uniqueness}\\
\bigl(A(\zeta)-A(\widehat\zeta)\bigr) \cdot(\zeta-\widehat\zeta)&\geq0 \label{eq:A-monotonicity-uniqueness}
\end{align}
for all $\zeta,\widehat\zeta\in\mathbb R^d$.
\end{assumption}

\begin{assumption}[Source and initial value]\label{assu:assumption for initial}
The function $f:D\times[0,T]\times\mathbb R\to\mathbb R$ is Carath\'eodory, $f(\cdot,\cdot,0)\in L^1(D\times[0,T])$, and there is $L_f\geq0$ such that
\begin{equation}
|f(x,t,r)-f(x,t,s)| \leq L_f|\beta(r)-\beta(s)| \label{eq:f-beta-Lipschitz-uniqueness}
\end{equation}
for almost every $(x,t)\in D\times[0,T]$ and all $r,s\in\mathbb R$.
The initial value $\xi:D\to\mathbb R$ is measurable and $\beta(\xi)\in L^1(D)$.
\end{assumption}
\begin{assumption}[Static obstacle]\label{assu:assumption for barrier}
The obstacle $\psi$ belongs to $C(\overline D)$ and satisfies $\psi\leq0$ on $\partial D$.
Set $M:=\|\psi\|_{L^\infty(D)}$.
\end{assumption}
We impose the compatibility condition
\begin{equation*}
\xi\geq\psi\quad\text{a.e. in }D. 
\end{equation*}

Condition \eqref{eq:f-beta-Lipschitz-uniqueness} and $\beta(u)\in L^\infty(0,T;L^1(D))$ imply
\begin{equation}
|f(x,t,u)|\leq |f(x,t,0)|+L_f|\beta(u)|, \qquad f(\cdot,\cdot,u)\in L^1(D\times[0,T]). \label{eq:source-integrability-uniqueness}
\end{equation}

For $k>0$, define
\begin{equation*}
\mathcal T_k(r):=\max\{-k,\min\{r,k\}\}.
\end{equation*}
If $\mathcal T_k(u)\in L^p(0,T;W_0^{1,p}(D))$ for every $k>0$, there is a unique measurable vector field, denoted by $\nabla u$, such that
\begin{equation*}
\nabla\mathcal T_k(u)=\one_{\{|u|<k\}}\nabla u \quad\text{a.e. in }D_T.
\end{equation*}
Indeed, the Sobolev chain rule shows that the gradients of $\mathcal T_k(u)$ and $\mathcal T_l(u)$ agree on $\{|u|<\min\{k,l\}\}$.

Set
\[
\mathcal E:=\{\eta\in C^1(\mathbb R):\eta'\geq0, \ \eta\in L^\infty(\mathbb R)\}, \qquad \mathcal E_0:=\{\eta\in\mathcal E:\eta(0)=0\}.
\]
For $l>0$, let
\[
\mathcal H_l:=\left\{h\in C_c^2(\mathbb R;[0,1]): h=1\text{ on }[-l,l],\quad h=0\text{ on }\mathbb R\setminus(-l-1,l+1)\right\}.
\]
Fix once and for all a family $(h_l)_{l>0}$ such that
\begin{equation}
h_l\in\mathcal H_l\quad\text{for every }l>0, \qquad \sup_{l>0}\|h_l'\|_{L^\infty(\mathbb R)}<\infty.\label{eq:admissible-cutoff-family}
\end{equation}

For $\eta\in\mathcal E$, a nonnegative function $\phi\in C^\infty(\overline D\times[0,T))$ is called admissible for $\eta$ if $\phi$ vanishes in a neighborhood of  $t=T$ and one of the following conditions holds:
\begin{enumerate}
\item there exists a compact set $K_\phi\Subset D$ such that $\operatorname{supp}\phi\subset K_\phi\times[0,T)$;
\item $\eta\in\mathcal E_0$, in which case $\phi$ may meet $\partial D\times[0,T)$.
\end{enumerate}
In the second case, the function $r\longmapsto \eta(r)h_l(r)$ is Lipschitz and vanishes at the origin. 
Hence $\eta(u)h_l(u)$ has zero Sobolev trace.
Throughout the paper, integrals with respect to $\dd\beta$ are understood as Lebesgue--Stieltjes integrals associated with the continuous nondecreasing function $\beta$.

\begin{defn}[Renormalized entropy solution]\label{def:entropy solution with ob}
A pair $(u,\nu)$ is a renormalized entropy solution of \eqref{eq:ito}--\eqref{eq:lower obstacle} if the following properties hold.
\begin{itemize}
\item[(i)] The function $u:D\times[0,T]\to\mathbb R$ is measurable, $u\geq\psi$ a.e.,
\[
\beta(u)\in L^\infty(0,T;L^1(D)),\qquad\mathcal T_k(u)\in L^p(0,T;W_0^{1,p}(D)) \quad\text{for every }k>0.
\]
\item[(ii)] The reaction $\nu$ is a nonnegative finite Radon measure on $D\times[0,T)$.
\item[(iii)] For every $K>0$, there exists a family $(\mu_{l,K}^u)_{l>M+1}\subset\mathcal M_+(D\times[0,T))$ such that for every $l>M+1$, $\eta\in\mathcal E$ with $\|\eta\|_{L^\infty(\mathbb R)}\leq K$, and every test function $\phi$ admissible for $\eta$,
\begin{align}
&-\int_{D_T}\partial_t\phi\int_{\xi}^{u}\eta(r)h_l(r)\,\dd\beta(r)\,\dd x\,\dd t -\int_{D\times[0,T)}\eta(\psi)\phi\,\dd\nu\nonumber\\
&\leq-\int_{D_T}h_l(u)A(\nabla u)\cdot\nabla\bigl(\eta(u)\phi\bigr)\,\dd x\,\dd t \nonumber\\
&\quad+\int_{D_T}f(x,t,u)\eta(u)h_l(u)\phi\,\dd x\,\dd t +\int_{D\times[0,T)}\phi\,\dd\mu_{l,K}^u, \label{eq:entropy formula-1}
\end{align}
and
\begin{equation}
\lim_{l\to\infty}\mu_{l,K}^u(D\times[0,T))=0.\label{eq:defect-tail-uniqueness}
\end{equation}
\end{itemize}
\end{defn}

\begin{rem}[Uniform tail control]
The measure $\mu_{l,K}^u$ is uniform over all entropies with $\|\eta\|_{L^\infty(\mathbb{R})}\leq K$. 
This point is essential: in the doubling argument the entropy is evaluated at the other solution and is therefore not fixed in advance.
\end{rem}

\begin{thm}[Comparison and uniqueness]\label{thm:main-comparison}
Let Assumption \ref{assu:assumption for b}, Assumption \ref{assu:assumption for a}, and
Assumption \ref{assu:assumption for barrier} hold.
For $i=1,2$, let $(f_i,\xi_i)$ satisfy Assumption \ref{assu:assumption for initial} and $\xi_i\geq\psi$ a.e. in $D$.
Let $(u_i,\nu_i)$ be a renormalized entropy solution with data $(f_i,\xi_i)$ and the same obstacle $\psi$.
Suppose
\[
f_1(x,t,r)\leq f_2(x,t,r) \quad\text{for a.e. }(x,t)\in D_T\text{ and every }r\in\mathbb R.
\]
If $L_2$ is a Lipschitz constant for $f_2$ in \eqref{eq:f-beta-Lipschitz-uniqueness}, then, for almost every $t\in(0,T)$,
\begin{equation}
\int_D\bigl(\beta(u_1(t))-\beta(u_2(t))\bigr)^+\,\dd x \leq \mathrm e^{L_2t}\int_D\bigl(\beta(\xi_1)-\beta(\xi_2)\bigr)^+\,\dd x. \label{eq:main-comparison}
\end{equation}
If, in addition $f_1(x,t,r)=f_2(x,t,r)$ for a.e. $(x,t)\in D_T$ and every $r\in\mathbb{R}$, and $\xi_1=\xi_2$ a.e. in $D$, then $\beta(u_1)=\beta(u_2)$ a.e. in $D_T$.
If $\beta$ is strictly increasing, then
\[
u_1=u_2\quad\text{a.e. in }D_T, \qquad \nu_1=\nu_2\quad\text{in }\mathcal M(D\times[0,T)).
\]
\end{thm}

\section{Comparison under the homogeneous Dirichlet boundary condition}\label{sec:comparison-dirichlet}
To prepare for the proof of the comparison theorem, we first introduce some essential notation and lemmas.
Throughout this section, the hypotheses of Theorem \ref{thm:main-comparison} are in force.
We write
\[
(u_1,\nu_1,f_1,\xi_1)=(u,\nu,f,\xi),\qquad (u_2,\nu_2,f_2,\xi_2)=(\tilde u,\tilde\nu,\tilde f,\tilde\xi).
\]
Thus
\[
f(x,t,r)\leq\tilde f(x,t,r)\quad\text{for a.e. }(x,t)\in D_T\text{ and every }r\in\mathbb R.
\]

The absolute value of each entropy function $\eta$ used below is at most one.
Set
\[
\mathfrak e_l:=\mu_{l,1}^u(D\times[0,T))+\mu_{l,1}^{\tilde u}(D\times[0,T)).
\]
Definition \ref{def:entropy solution with ob} gives $\mathfrak e_l\to0$ as $l\to\infty$.

For $r\in\mathbb R$, write $r^+:=\max\{r,0\}$ and $r^-:=\max\{-r,0\}$.
For a measurable set $E$, let $\one_E$ denote its indicator.
For $\delta>0$, choose a nondecreasing $q_\delta\in C^1(\mathbb R)$ such that
\begin{equation}\label{eq:property of q1}
0\leq q_\delta\leq1,\qquad q_\delta(r)=0\ \text{for }r\leq0,\qquad q_\delta(r)=1\ \text{for }r\geq\delta,
\end{equation}
and
\begin{equation}\label{eq:property of q2}
q_\delta'\geq0,\qquad\operatorname{supp}q_\delta'\subset(0,\delta),\qquad\|q_\delta'\|_{L^\infty(\mathbb R)}\leq C\delta^{-1}.
\end{equation}

Fix $l>M+1$ and use the cutoff $h_l$ from \eqref{eq:admissible-cutoff-family}.
Define
\begin{equation}\label{eq:local-beta-l}
H_l(r):=\int_0^r h_l(\tilde{r})\,\dd\beta(\tilde{r}),
\end{equation}
and
\begin{equation}\label{eq:defBl}
\mathcal B_l(a,b):=\one_{\{a>b\}}\int_b^a h_l(r)\,\dd\beta(r)=\bigl(H_l(a)-H_l(b)\bigr)^+.
\end{equation}
We use the moduli
\[
\omega_{\beta,l}(r):=\sup_{\substack{|a|,|b|\leq l+2\\|a-b|\leq r}}|\beta(a)-\beta(b)|,\qquad
\omega_\psi(r):=\sup_{\substack{x,y\in\overline D\\|x-y|\leq r}}|\psi(x)-\psi(y)|.
\]

Choose an even function $\varrho\in C_c^\infty((-1,1))$ such that $\varrho\geq0$ and $\int_{\mathbb R}\varrho(r)\,\dd r=1$.
Define
\[
\varrho_\theta(r):=\theta^{-1}\varrho(r/\theta).
\]

\subsection{Boundary charts and inward spatial mollifiers}

\begin{lem}[Boundary covering]\label{lem:boundary covering}
There exist open sets $B_0,\ldots,B_J$, open sets $V_j$ with $\operatorname{supp}\chi_j\Subset V_j\Subset B_j$ for $j\in\{1,\ldots,J\}$, and functions $\chi_j\in C_c^\infty(B_j)$, $0\leq\chi_j\leq1$, such that
\[
\sum_{j=0}^J\chi_j=1\quad\text{on }\overline D.
\]
Moreover, $\overline{B_0}\subset D$. 
For each $j\in\{1,\ldots,J\}$, there exist an open set $B_j'$ with $B_j\Subset B_j'$, an orthogonal matrix $R_j$, a vector $c_j$, and a Lipschitz function $\gamma_j$ such that, in the coordinates $X=R_jx+c_j$,
\[
R_j(D\cap B_j')+c_j=\left\{X=(X',X_d)\in R_jB_j'+c_j:X_d>\gamma_j(X')\right\}.
\]
We denote the Lipschitz constant of $\gamma_j$ by $\tilde{L}_j$.
\end{lem}

\begin{proof}
The assertion follows from the compactness of $\partial D$, the local graph representation of a bounded Lipschitz domain, and the smooth partition of unity theorem.
\end{proof}
Choose $\rho\in C_c^\infty(\{z\in\mathbb{R}^d:|z|<1\})$ satisfying $\rho\geq0$ and $\int_{\mathbb R^d}\rho(z)\,\dd z=1$. 
For $j\in\{1,\ldots,J\}$, fix $\lambda_j>\tilde{L}_j+2$ and define
\[
\rho_{\varsigma,j}(x,y):=\varsigma^{-d}\rho\left(\frac{R_j(x-y)-\lambda_j\varsigma e_d}{\varsigma}\right),
\]
where $e_d$ denotes the unit vector in $\mathbb R^d$ with a $1$ in the $d$-th coordinate and 0 elsewhere. 
For the interior chart, set
\[
\rho_{\varsigma,0}(x,y):=\rho_{\varsigma}(x-y):=\varsigma^{-d}\rho\left(\frac{x-y}{\varsigma}\right).
\]

\begin{lem}[Properties of the inward kernels]
\label{lem:inward-kernel-properties}
For all sufficiently small $\varsigma>0$, $j\in\{1,\ldots,J\}$ and $y\in\overline D\cap V_j$, the function $x\mapsto\rho_{\varsigma,j}(x,y)$ is compactly supported in $D\cap B_j'$, and satisfies
\[
\int_D\rho_{\varsigma,j}(x,y)\,\dd x=1,\qquad
\nabla_x\rho_{\varsigma,j}(x,y)+\nabla_y\rho_{\varsigma,j}(x,y)=0.
\]

If
\[
\sigma_{\varsigma,j}(x):=\int_D\rho_{\varsigma,j}(x,y)\,\dd y,
\]
then $0\leq\sigma_{\varsigma,j}\leq1$ and $\sigma_{\varsigma,j}(x)\to1$ for every $x\in D$.
For $\chi\in C_c^\infty(V_j)$, we have $\chi\sigma_{\varsigma,j}\in C_c^\infty(D)$.
For $j=0$, the same conclusion holds for every $\chi\in C_c^\infty(D)$.
\end{lem}

\begin{proof}
Fix $j\in\{1,\ldots,J\}$ and write $X=R_jx+c_j$ and $Y=R_jy+c_j$.
If $\rho_{\varsigma,j}(x,y)\neq0$, then
\[
X-Y=\lambda_j\varsigma e_d+\varsigma z
\]
for some $z\in B_1(0)$. 
Hence, $X'=Y'+\varsigma z'$ and $X_d=Y_d+\lambda_j\varsigma+\varsigma z_d$. 
Since $y\in\overline D\cap B_j'$, we have
\[
Y_d\geq\gamma_j(Y').
\]
Consequently,
\[
\begin{aligned}
X_d-\gamma_j(X')&\geq Y_d+\lambda_j\varsigma-\varsigma-\gamma_j(Y')-\tilde{L}_j\varsigma  \\
&\geq(\lambda_j-\tilde{L}_j-1)\varsigma>0.
\end{aligned}
\]
Thus the $x$-support lies strictly inside $D$. 
After reducing $\varsigma$ if necessary, it also lies inside $B_j'$. 
Therefore, the integral over $D$ equals the integral over $\mathbb R^d$.

The gradient identity follows directly from the dependence of the kernel on $x-y$. 
Since the integral defining $\sigma_{\varsigma,j}$ is taken over a subset of $\mathbb R^d$, $0\leq\sigma_{\varsigma,j}\leq1$. 
For every fixed $x\in D$, the full $y$-support is contained in $D$ when $\varsigma$ is sufficiently
small, and hence $\sigma_{\varsigma,j}(x)=1$.

Since $\operatorname{supp}\chi\Subset V_j$, the preceding estimate is uniform for $x\in\operatorname{supp}\chi$. 
Note that $\sigma_{\varsigma,j}$ is smooth in $x$. 
Then, for all sufficiently small $\varsigma$, the product $\chi\sigma_{\varsigma,j}$ is smooth and vanishes in a fixed neighborhood of $\partial D$.
Therefore, we have $\chi\sigma_{\varsigma,j}\in C_c^\infty(D)$.
The interior-chart case is standard.
\end{proof}

Define the global inward kernel
\[
K_\varsigma(x,y):=\sum_{j=0}^J\chi_j(x)\rho_{\varsigma,j}(x,y).
\]
We have $K_\varsigma\geq0$, its support is contained in $\{|x-y|\leq C\varsigma\}$, and for every $g\in L^1(D)$,
\[
\int_{D\times D}K_\varsigma(x,y)|g(x)-g(y)|\,\dd x\,\dd y\longrightarrow0,\qquad\text{as }\varsigma\downarrow0.
\]
The last assertion follows by extending $g$ by zero outside $D$ and using the continuity of translations in $L^1(\mathbb R^d)$.

\subsection{Zero-level entropy functionals}

Let $(v,\nu_v)$ be a renormalized entropy solution with initial value $\xi_v$ and source $f_v$.
We write $\mu_{l,1}^v$ for the defect measure from Definition \ref{def:entropy solution with ob}. 
For every smooth function $\zeta$ on $\overline D\times(0,T)$ whose time support is compactly contained in $(0,T)$, define
\begin{align*}
\mathscr L_{l,v}^{-}(\zeta)&:=-\int_{D_T}\partial_t\zeta\int_{\xi_v}^{v}\one_{\{r<0\}}h_l(r)\,\dd\beta(r)\,\dd x\,\dd t\notag\\
&\quad-\int_{D_T}\one_{\{\psi<0\}}\zeta\,\dd\nu_v+\int_{D_T}\one_{\{v<0\}}h_l(v)A(\nabla v)\cdot\nabla\zeta\,\dd x\,\dd t\notag\\
&\quad-\int_{D_T}f_v(x,t,v)\one_{\{v<0\}}h_l(v)\zeta\,\dd x\,\dd t+\int_{D_T}\zeta\,\dd\mu_{l,1}^v,
\end{align*}
and
\begin{align*}
\mathscr L_{l,v}^{+}(\zeta)&:=\int_{D_T}\partial_t\zeta\int_{\xi_v}^{v}\one_{\{r>0\}}h_l(r)\,\dd\beta(r)\,\dd x\,\dd t\notag\\
&\quad+\int_{D_T}\one_{\{\psi>0\}}\zeta\,\dd\nu_v-\int_{D_T}\one_{\{v>0\}}h_l(v)A(\nabla v)\cdot\nabla\zeta\,\dd x\,\dd t\notag\\
&\quad+\int_{D_T}f_v(x,t,v)\one_{\{v>0\}}h_l(v)\zeta\,\dd x\,\dd t+\int_{D_T}\zeta\,\dd\mu_{l,1}^v.
\end{align*}

The functionals are well defined.
The factors $h_l(v)A(\nabla v)$ are integrable on $D_T$, since they only use $\nabla\mathcal T_{l+1}(v)$.
The test function $\zeta$ may meet $\partial D$ because the regularized zero-level entropies belong to $\mathcal E_0$.

\begin{lem}[Positivity]
For every nonnegative $\zeta\in C^1_c(\overline{D}\times(0,T))$, we have
\[
\mathscr L_{l,v}^{-}(\zeta)\geq0,\qquad\mathscr L_{l,v}^{+}(\zeta)\geq0.
\]
Consequently, there exist nonnegative Radon measures $\lambda_{l,v}^{\pm}$ on $\overline D\times(0,T)$ such that
\begin{equation}\label{eq:riesz-representation}
\mathscr L_{l,v}^{\pm}(\zeta)=\int_{\overline D\times(0,T)}\zeta\,\dd\lambda_{l,v}^{\pm}
\end{equation}
for every smooth $\zeta$ with compact time support in $(0,T)$.
\end{lem}

\begin{proof}
Use the entropy inequality \eqref{eq:entropy formula-1} for $\nu$ with $\eta(r)=-q_\delta(-r)$. 
Its derivative is $q_\delta'(-r)\geq0$, and it satisfies $\eta(0)=0$. 
Expanding the diffusion term gives
\[
\begin{aligned}
&\int_{D_T}h_l(v)A(\nabla v)\cdot\nabla\bigl(q_{\delta}(-v)\zeta\bigr)\,\dd x\,\dd t \\
&=-\int_{D_T}q_\delta'(-v)h_l(v)A(\nabla v)\cdot\nabla v\,\zeta\,\dd x\,\dd t+\int_{D_T}q_\delta(-v)h_l(v)A(\nabla v)\cdot\nabla\zeta\,\dd x\,\dd t.
\end{aligned}
\]
By Assumption \ref{assu:assumption for a} and $A(0)=0$, we have $A(\nabla v)\cdot\nabla v\geq0$. Dropping this nonpositive dissipation term, rearranging the entropy inequality, and letting $\delta\downarrow0$ gives $\mathscr L_{l,v}^{-}(\zeta)\geq0$, which is based on the dominated convergence theorem.
The same argument with $\eta(r)=q_\delta(r)$ gives $\mathscr L_{l,v}^{+}(\zeta)\geq0$.

To justify \eqref{eq:riesz-representation}, fix a compact set $K\subset\overline D\times(0,T)$ and choose $\chi_K\in C_c^\infty(\overline D\times(0,T))$ such that $\chi_K\geq1$ on $K$.
If $\zeta$ is supported in $K$, then positivity gives
\[
-\|\zeta\|_{L^\infty(D_T)}\chi_K\leq\zeta\leq\|\zeta\|_{L^\infty(D_T)}\chi_K,\qquad\text{and}\qquad
|\mathscr L_{l,v}^{\pm}(\zeta)|\leq\mathscr L_{l,v}^{\pm}(\chi_K)\|\zeta\|_{L^\infty(\overline D\times(0,T))}.
\]
Hence $\mathscr L_{l,v}^{\pm}$ is locally of order zero and extends uniquely to a positive linear functional on $C_c(\overline D\times(0,T))$.
The Riesz--Markov theorem then yields a unique nonnegative Radon measure $\lambda_{l,v}^{\pm}$ satisfying \eqref{eq:riesz-representation}.
\end{proof}

\begin{lem}[Vanishing of boundary-layer functional terms]\label{lem:L-boundary-vanishing}
Let $\varphi\in C_c^\infty((0,T))$ be nonnegative. 
Then, for every boundary chart,
\[
\lim_{\widehat\varsigma\downarrow0}\lim_{\varsigma\downarrow0}\mathscr L_{l,v}^{\pm}\left(\chi_j\varphi(1-\sigma_{\widehat\varsigma,j})\sigma_{\varsigma,j}\right)=0.
\]
\end{lem}

\begin{proof}
For fixed $\widehat\varsigma,\varsigma>0$, the function $\chi_j\varphi(1-\sigma_{\widehat\varsigma,j})\sigma_{\varsigma,j}$ is compactly supported in $D\times(0,T)$. 
Hence
\[
\begin{aligned}
&\mathscr L_{l,v}^{\pm}\left(\chi_j\varphi(1-\sigma_{\widehat\varsigma,j})\sigma_{\varsigma,j}\right)=\int_{\overline D\times(0,T)}\chi_j\varphi(1-\sigma_{\widehat\varsigma,j})\sigma_{\varsigma,j}\,\dd\lambda_{l,v}^{\pm}.
\end{aligned}
\]
For $x\in D$, we have $\sigma_{\varsigma,j}(x)\to1$ as $\varsigma\downarrow0$, whereas the product $\chi_j\sigma_{\varsigma,j}$ vanishes on $\partial D$. 
Dominated convergence theorem therefore gives
\[
\begin{aligned}
&\lim_{\varsigma\downarrow0}\mathscr L_{l,v}^{\pm}\left(\chi_j\varphi(1-\sigma_{\widehat\varsigma,j})\sigma_{\varsigma,j}\right) =\int_{D\times(0,T)}\chi_j\varphi(1-\sigma_{\widehat\varsigma,j})\,\dd\lambda_{l,v}^{\pm}.
\end{aligned}
\]
Since $1-\sigma_{\widehat\varsigma,j}(x)\to0$ as $\widehat\varsigma\downarrow0$ for every $x\in D$, a second application of the dominated convergence theorem proves the result.
\end{proof}

\subsection{A time-diagonal source lemma}

\begin{lem}[Time-diagonal source estimate]\label{lem:time-diagonal-source}
Let $a,b$ be finite measurable functions on $D_T$, let $F,G\in L^1(D_T)$, and let $\zeta\in L^\infty(D_T)$ be nonnegative with $\operatorname{supp}_t\zeta\Subset(0,T)$. 
Assume
\[
F-G\leq0\quad\text{a.e. on }\{a=b\}.
\]
Then
\begin{align}
&\limsup_{\theta\downarrow0}\int_0^T\!\int_0^T\!\int_D\zeta(x,t)\varrho_\theta(s-t)\one_{\{a(x,s)>b(x,t)\}}\bigl[F(x,s)-G(x,t)\bigr]\,\dd x\,\dd s\,\dd t\notag\\
&\leq\int_{D_T}\zeta(x,t)\one_{\{a>b\}}(F-G)\,\dd x\,\dd t.
\label{eq:time-diagonal-source}
\end{align}
\end{lem}

\begin{proof}
Extend all functions by zero outside $(0,T)$ in the time variable and set
\[
J(r):=\int_{D_T}\zeta(x,t)\one_{\{a(x,t+r)>b(x,t)\}}\bigl[F(x,t+r)-G(x,t)\bigr]\,\dd x\,\dd t.
\]
We claim that
\[
\limsup_{r\to0}J(r)\leq\int_{D_T}\zeta\one_{\{a>b\}}(F-G)\,\dd x\,\dd t.
\]

Let $r_n\to0$. 
Since translations are continuous in $L^1$, after passing to a subsequence, we have $F(x,t+r_n)\to F(x,t)$ almost everywhere. 
Since translations of a finite measurable function converge locally in measure, after a further subsequence, $a(x,t+r_n)\to a(x,t)$ almost everywhere.

On $\{a>b\}$, the indicator converges to one, and on $\{a<b\}$ it converges to zero.
On $\{a=b\}$, the assumption on $F-G$ gives
\[
\limsup_{n\to\infty}\one_{\{a(x,t+r_n)>b(x,t)\}}\bigl[F(x,t+r_n)-G(x,t)\bigr]\leq0.
\]
For bounded $F$ and $G$, the claim follows from the reverse Fatou inequality. 
We now consider general $F,G\in L^1(D_T)$.
Let $T_N(r):=\max\{-N,\min\{r,N\}\}$.
Since time translations are isometries in $L^1(\mathbb R;L^1(D))$, we have
\begin{align*}
&\sup_{|r|<r_0}\|F(\cdot,\cdot+r)-T_N(F(\cdot,\cdot+r))\|_{L^1(D_T)}=\|F-T_N(F)\|_{L^1(D_T)}\longrightarrow0\quad\text{as }N\to\infty.
\end{align*}
The same estimate holds for $G$.
Therefore, the translated tails are uniformly small for $r$ near zero.
We may first apply the bounded argument to $T_N(F)$ and $T_N(G)$, and then let $N\to\infty$.
Consequently, we obtain the claimed upper bound for general $F,G\in L^1(D_T)$.

Finally, note that
\begin{align*}
&\int_0^T\!\int_0^T\!\int_D\zeta(x,t)\varrho_\theta(s-t)\one_{\{a(x,s)>b(x,t)\}}\bigl[F(x,s)-G(x,t)\bigr]\,\dd x\,\dd s\,\dd t  \\
&=\int_{\mathbb R}\varrho_\theta(r)J(r)\,\dd r.
\end{align*}
Since $\varrho_\theta$ is nonnegative and concentrates at zero, \eqref{eq:time-diagonal-source} follows.
\end{proof}

\subsection{Interior doubled inequality}
\begin{lem}[Interior local Kato inequality]\label{lem:interior-local-Kato}
Let $\chi\in C_c^\infty(D)$ and $\varphi\in C_c^\infty((0,T))$ be nonnegative. 
Then, for every $l>M+1$,
\begin{align*}
&-\int_{D_T}\mathcal B_l(u,\tilde u)\chi\varphi'\,\dd x\,\dd t\notag\\
&\leq\int_{D_T}\one_{\{u>\tilde u\}}\left[h_l(\tilde u)A(\nabla\tilde u)-h_l(u)A(\nabla u)\right]\cdot\nabla\chi\,\varphi\,\dd x\,\dd t\notag\\
&\quad+\int_{D_T}\one_{\{u>\tilde u\}}\left[f(x,t,u)h_l(u)-\tilde f(x,t,\tilde u)h_l(\tilde u)\right]\chi\varphi\,\dd x\,\dd t+C\mathfrak e_l,
\end{align*}
where the constant $C$ depends on $\Vert\chi\Vert_{L^\infty(D)}$ and $\Vert\varphi\Vert_{L^\infty(0,T)}$.
\end{lem}

\begin{proof}
Choose $\theta>0$ such that $
2\theta<\operatorname{dist}\bigl(\operatorname{supp}\varphi,\{0,T\}\bigr)$
and set 
\[
\Phi_{\theta,\varsigma}(x,t,y,s):=\chi(x)\rho_\varsigma(x-y)\varphi(t)\varrho_\theta(s-t).
\]
We will apply the space--time doubling-of-variables method \cite{kruvzkov1970first}. 
Note that throughout this section, unmarked integrals are taken over $D^2\times(0,T)^2$ with respect to $\dd x\,\dd t\,\dd y\,\dd s$. 
Write
\[
U:=u(y,s),\qquad V:=\tilde u(x,t),\qquad P:=\nabla_yu(y,s),\qquad Q:=\nabla_x\tilde u(x,t).
\]

For $\delta\in(0,1)$, use the entropy $r\mapsto q_\delta(r-V)$ in the entropy inequality \eqref{eq:entropy formula-1} for $u$ and $r\mapsto-q_\delta(U-r)$ in the entropy inequality for $\tilde u$.
After integrating in the doubled variables and adding, one obtains
\[
I_1\leq I_2+I_3+I_4+I_5+C\mathfrak e_l,
\]
where
\begin{align*}
I_1&=\int\partial_t\Phi_{\theta,\varsigma}\int_{\tilde\xi(x)}^Vq_\delta(U-r)h_l(r)\,\dd\beta(r)-\int\partial_s\Phi_{\theta,\varsigma}\int_{\xi(y)}^Uq_\delta(r-V)h_l(r)\,\dd\beta(r),\\
I_2&=\int q_\delta(\psi(y)-V)\Phi_{\theta,\varsigma}\,\dd\nu(y,s)\,\dd x\,\dd t-\int q_\delta(U-\psi(x))\Phi_{\theta,\varsigma}\,\dd\tilde\nu(x,t)\,\dd y\,\dd s,\\
I_3&=-\int q_\delta'(U-V)\Phi_{\theta,\varsigma}\left[h_l(U)A(P)\cdot P+h_l(V)A(Q)\cdot Q\right],\\
I_4&=\int q_\delta(U-V)\left[h_l(V)A(Q)\cdot\nabla_x\Phi_{\theta,\varsigma}-h_l(U)A(P)\cdot\nabla_y\Phi_{\theta,\varsigma}\right],\\
I_5&=\int\left[f(y,s,U)h_l(U)-\tilde f(x,t,V)h_l(V)\right]q_\delta(U-V)\Phi_{\theta,\varsigma}.
\end{align*}

Define
\[
\mathcal B_{\delta,l}^{1}(a,b):=\one_{\{a>b\}}\int_b^aq_\delta(r-b)h_l(r)\,\dd\beta(r)
\]
and
\[
\mathcal B_{\delta,l}^{2}(a,b):=\one_{\{a>b\}}\int_b^aq_\delta(a-r)h_l(r)\,\dd\beta(r).
\]

For the first time term, split the interval at $U$:
\[
\begin{aligned}
\int_{\tilde\xi(x)}^Vq_\delta(U-r)h_l(r)\,\dd\beta(r)&=\int_{\tilde\xi(x)}^Uq_\delta(U-r)h_l(r)\,\dd\beta(r)-\mathcal B_{\delta,l}^{2}(U,V).
\end{aligned}
\]
The first term on the right depends on $(x,y,s)$ but not on $t$.
Therefore, by Fubini's theorem,
\[
\begin{aligned}
&\int\partial_t\Phi_{\theta,\varsigma}\int_{\tilde\xi(x)}^Uq_\delta(U-r)h_l(r)\,\dd\beta(r)\\
&=\int_{D^2\times(0,T)}\left[\int_{\tilde\xi(x)}^Uq_\delta(U-r)h_l(r)\,\dd\beta(r)\right]\left[\int_0^T\partial_t\Phi_{\theta,\varsigma}(x,t,y,s)\,\dd t\right]\dd x\,\dd y\,\dd s.
\end{aligned}
\]
Since $\Phi_{\theta,\varsigma}$ vanishes near $t=0$ and $T$, the inner $t$-integral is zero.
Similarly, we have
\[
\begin{aligned}
\int_{\xi(y)}^Uq_\delta(r-V)h_l(r)\,\dd\beta(r)&=\int_{\xi(y)}^Vq_\delta(r-V)h_l(r)\,\dd\beta(r)+\mathcal B_{\delta,l}^{1}(U,V).
\end{aligned}
\]
The first term on the right is independent of $s$, and hence its product with $\partial_s\Phi_{\theta,\varsigma}$ also integrates to zero. 
Thus
\[
I_1=-\int\mathcal B_{\delta,l}^{2}(U,V)\partial_t\Phi_{\theta,\varsigma}-\int\mathcal B_{\delta,l}^{1}(U,V)\partial_s\Phi_{\theta,\varsigma}.
\]

For $i=1,2$, with the definition of $\mathcal B_{l}$ in \eqref{eq:defBl}, we have
\[
\left|\mathcal B_{\delta,l}^{i}(a,b)-\mathcal B_l(a,b)\right|\leq\omega_{\beta,l}(2\delta).
\]
It follows that
\[
I_1=-\int\mathcal B_l(U,V)(\partial_t+\partial_s)\Phi_{\theta,\varsigma}+R_{\delta,\theta,\varsigma}^{t},
\]
where
\[
(\partial_t+\partial_s)\Phi_{\theta,\varsigma}=\chi(x)\rho_\varsigma(x-y)\varphi'(t)\varrho_\theta(s-t),\qquad|R_{\delta,\theta,\varsigma}^{t}|\leq C\omega_{\beta,l}(2\delta)\left(1+\theta^{-1}\right).
\]
Moreover, since $V\geq\psi(x)$ and $q_\delta(0)=0$, we have
\[
q_\delta(\psi(y)-V)\leq q_\delta(\psi(y)-\psi(x))\leq\Vert q_\delta'\Vert_{L^\infty(\mathbb R)}|\psi(y)-\psi(x)|.
\]
Using the property of $q_\delta$ and $\omega_\psi$ and noting that the second term in $I_2$ is nonpositive, we have
\[
I_2\leq C\frac{\omega_\psi(C\varsigma)}{\delta}\nu(D\times[0,T)).
\]

Using $\nabla_x\rho_\varsigma+\nabla_y\rho_\varsigma=0$, a direct calculation gives
\[
\begin{aligned}
I_3+I_4&=-\int q_\delta'(U-V)\Phi_{\theta,\varsigma}\left[h_l(U)A(P)-h_l(V)A(Q)\right]\cdot(P-Q)\\
&\quad+\int q_\delta(U-V)\left[h_l(V)A(Q)-h_l(U)A(P)\right]\cdot\nabla\chi(x)\rho_\varsigma(x-y)\varphi(t)\varrho_\theta(s-t).
\end{aligned}
\]
Note that
\[
\begin{aligned}
&\left[h_l(U)A(P)-h_l(V)A(Q)\right]\cdot(P-Q)\\
&=h_l(U)\left[A(P)-A(Q)\right]\cdot(P-Q)+\left[h_l(U)-h_l(V)\right]A(Q)\cdot(P-Q).
\end{aligned}
\]
The first term on the right-hand side is nonnegative by \eqref{eq:A-monotonicity-uniqueness}.
The bounds
\[
|h_l(U)-h_l(V)|\leq\|h_l'\|_{L^\infty(\mathbb{R})}|U-V|,\qquad q_\delta'(U-V)|U-V|\leq C\one_{\{0<U-V<\delta\}}
\]
control the remaining term
\[
E_{\delta,\theta,\varsigma}:=-\int q_\delta'(U-V)\Phi_{\theta,\varsigma}[h_l(U)-h_l(V)]A(Q)\cdot(P-Q),
\]
which satisfies
\[
\begin{aligned}
|E_{\delta,\theta,\varsigma}|\leq C\int&\one_{\{0<U-V<\delta\}}\Phi_{\theta,\varsigma}\left[1+|\nabla\mathcal T_{l+2}(u(y,s))|^p+|\nabla\mathcal T_{l+2}(\tilde u(x,t))|^p\right].
\end{aligned}
\]

Fix $\delta,\theta>0$.
Since the family $\rho_\varsigma$ is a spatial approximate identity and all truncated fluxes are integrable, we let $\varsigma\downarrow0$.
Therefore, the spatially doubled terms converge to their spatial diagonal counterparts.
In particular, we have
\begin{align*}
\limsup_{\varsigma\downarrow0}|E_{\delta,\theta,\varsigma}|&\leq C\int_0^T\!\int_0^T\!\int_D\one_{\{0<u(x,s)-\tilde u(x,t)<\delta\}}\chi(x)\varphi(t)\varrho_\theta(s-t)\\
&\quad\times\left(1+|\nabla\mathcal T_{l+2}(u(x,s))|^p+|\nabla\mathcal T_{l+2}(\tilde u(x,t))|^p\right)\,\dd x\,\dd s\,\dd t.
\end{align*}
For fixed $\theta$, the integrand on the right-hand side is dominated by an integrable function.
Moreover, the indicator converges to zero almost everywhere as $\delta\downarrow0$.
Therefore, the dominated convergence theorem gives
\[
\lim_{\delta\downarrow0}\limsup_{\varsigma\downarrow0}|E_{\delta,\theta,\varsigma}|=0.
\]

For the source term, set
\[
F_l(x,t):=f(x,t,u(x,t))h_l(u(x,t)),\qquad\tilde F_l(x,t):=\tilde f(x,t,\tilde u(x,t))h_l(\tilde u(x,t)).
\]
After letting $\varsigma\downarrow0$ and then $\delta\downarrow0$, we have
\[
\begin{aligned}
I_5\longrightarrow\int_0^T\!\int_0^T\!\int_D\one_{\{u(x,s)>\tilde u(x,t)\}}\bigl[F_l(x,s)-\tilde F_l(x,t)\bigr]\chi(x)\varphi(t)\varrho_\theta(s-t)\,\dd x\,\dd s\,\dd t.
\end{aligned}
\]
On $\{(x,t)\in D_T: u(x,t)=\tilde u(x,t)\}$, we have almost everywhere
\[
F_l(x,t)-\tilde F_l(x,t)=h_l(u)\bigl[f(x,t,u)-\tilde f(x,t,u)\bigr]
\leq0.
\]
Lemma \ref{lem:time-diagonal-source} therefore gives
\[
\begin{aligned}
\limsup_{\theta\downarrow0}\limsup_{\delta\downarrow0}\limsup_{\varsigma\downarrow0}I_5\leq\int_{D_T}\one_{\{u>\tilde u\}}\bigl[F_l-\tilde F_l\bigr]\chi\varphi\,\dd x\,\dd t.
\end{aligned}
\]
Moreover, by passing to the limits in the order
\[
\varsigma\downarrow0,\qquad\delta\downarrow0,\qquad\theta\downarrow0,
\]
and using the continuity of translation in $L^1(D_T)$, the time term converges to $-\int_{D_T}\mathcal B_l(u,\tilde u)\chi\varphi'\,\dd x\,\dd t$, and the localization flux converges to
\[
\int_{D_T}\one_{\{u>\tilde u\}}\left[h_l(\tilde u)A(\nabla\tilde u)-h_l(u)A(\nabla u)\right]\cdot\nabla\chi\,\varphi\,\dd x\,\dd t.
\]
This proves the lemma.
\end{proof}

\subsection{Signed chartwise boundary inequality}

Define
\[
\mathscr G_l(u,\tilde u):=\one_{\{u>\tilde u\}}\left[h_l(\tilde u)A(\nabla\tilde u)-h_l(u)A(\nabla u)\right]
\]
and
\[
\mathscr S_l(u,\tilde u):=\one_{\{u>\tilde u\}}\left[f(x,t,u)h_l(u)-\tilde f(x,t,\tilde u)h_l(\tilde u)\right].
\]
For a smooth spatial function $\chi\in C_c^\infty(\mathbb{R}^d)$, set
\[
\begin{aligned}
\mathscr F_l(\chi,\varphi)&:=-\int_{D_T}\mathcal B_l(u,\tilde u)\chi\varphi'\,\dd x\,\dd t-\int_{D_T}\mathscr G_l(u,\tilde u)\cdot\nabla\chi\,\varphi\,\dd x\,\dd t-\int_{D_T}\mathscr S_l(u,\tilde u)\chi\varphi\,\dd x\,\dd t.
\end{aligned}
\]

We also use
\[
\mathbf A_l^+(v):=\one_{\{v>0\}}h_l(v)A(\nabla v),\qquad\mathbf A_l^-(v):=-\one_{\{v<0\}}h_l(v)A(\nabla v).
\]

\begin{lem}[Chartwise signed Kato inequality]\label{lem:chartwise-signed-Kato}
For $j\in\{1,\ldots,J\}$, let $\chi\in C_c^\infty(V_j)$ be nonnegative, and let $\varphi\in C_c^\infty((0,T))$ be nonnegative.
Then
\[
\begin{aligned}
\mathscr F_l(\chi,\varphi)&\leq\limsup_{\varsigma\downarrow0}\mathscr L_{l,\tilde u}^{-}\left(\chi\varphi\sigma_{\varsigma,j}\right)+\limsup_{\varsigma\downarrow0}\mathscr L_{l,u}^{+}\left(\chi\varphi\sigma_{\varsigma,j}\right)+C\mathfrak e_l.
\end{aligned}
\]
Here $C$ depends only on $\|\chi\|_{L^\infty(D)}$ and $\|\varphi\|_{L^\infty(0,T)}$.
In particular, it is independent of derivatives of $\chi$.
\end{lem}
\begin{rem}
The function $\chi$ in Lemma \ref{lem:chartwise-signed-Kato} is not required to coincide with the partition function $\chi_j$ from Lemma \ref{lem:boundary covering}.
In the proof of Lemma \ref{lem:boundary-localization}, we apply the chartwise inequality with 
\[
\chi=\chi_j(1-\sigma_{\widehat\varsigma,j}).
\]
Since the constant in Lemma \ref{lem:chartwise-signed-Kato} is independent of the derivatives of $\chi$, we may subsequently let $\widehat\varsigma\downarrow0$.
\end{rem}

\begin{proof}[Proof of Lemma \ref{lem:chartwise-signed-Kato}]
The argument follows the space--time doubling-of-variables method of \cite{kruvzkov1970first,otto1996l1,carrillo1999uniqueness}, adapted here to the obstacle reaction and the inward boundary kernels.
Choose $\theta>0$ so that
\[
2\theta<\operatorname{dist}\bigl(\operatorname{supp}\varphi,\{0,T\}\bigr)
\]
and for simplicity, set
\[
\Phi(x,t,y,s):=\chi(x)\rho_{\varsigma,j}(x,y)\varphi(t)\varrho_\theta(s-t).
\]
If $\Phi(x,t,y,s)\neq0$, then $x\in\operatorname{supp}\chi\Subset V_j$ and $|x-y|\leq C\varsigma$.
Hence $y\in V_j$ for all small $\varsigma$.
Lemma \ref{lem:inward-kernel-properties} then shows that every $x$-test used below is compactly supported in $D$.
All unmarked integrals below are taken over $D^2\times(0,T)^2$ with respect to $\dd x\,\dd t\,\dd y\,\dd s$.

\medskip
\noindent
\textbf{Step 1. Positive-part inequalities.}

Write
\[
U:=u(y,s),\qquad V:=\tilde u(x,t),\qquad P:=\nabla_yu(y,s),\qquad Q:=\nabla_x\tilde u(x,t).
\]
For fixed $(x,t)$, use $\eta(r)=q_\delta(r-V^+)$ in the entropy inequality for $u$. 
Since $\eta(0)=0$, the $y$-test may touch $\partial D$. 
We obtain
\begin{equation*}
\begin{aligned}
&-\int\partial_s\Phi\int_{\xi(y)}^Uq_\delta(r-V^+)h_l(r)\,\dd\beta(r)-\int q_\delta(\psi(y)-V^+)\Phi\,\dd\nu(y,s)\,\dd x\,\dd t\\
&\leq-\int h_l(U)A(P)\cdot\nabla_y\left[q_\delta(U-V^+)\Phi\right]\\
&\quad+\int f(y,s,U)h_l(U)q_\delta(U-V^+)\Phi+C\mu_{l,1}^u(D_T),
\end{aligned}
\end{equation*}
where $C$ depends only on $\|\chi\|_{L^\infty(D)}$ and $\|\varphi\|_{L^\infty(0,T)}$.
Since $V^+\geq0$, we have
\[
\int_{\xi(y)}^Uq_\delta(r-V^+)h_l(r)\,\dd\beta(r)=\int_{\xi^+(y)}^{U^+}q_\delta(r-V^+)h_l(r)\,\dd\beta(r)
\]
and $q_\delta(U-V^+)=q_\delta(U^+-V^+)$. 
Thus, we have
\begin{equation}\label{eq:entropy for u-v+}
\begin{aligned}
&-\int\partial_s\Phi\int_{\xi^+(y)}^{U^+}q_\delta(r-V^+)h_l(r)\,\dd\beta(r)-\int q_\delta(\psi^+(y)-V^+)\Phi\,\dd\nu(y,s)\,\dd x\, \dd t\\
&\leq-\int\mathbf A_l^+(U)\cdot\nabla_y\left[q_\delta(U^+-V^+)\Phi\right]\\
&\quad+\int\one_{\{U>0\}}f(y,s,U)h_l(U)q_\delta(U^+-V^+)\Phi+C\mu_{l,1}^u(D_T).
\end{aligned}
\end{equation}
Let $\pi_\iota$ be a smooth approximation of $r^+$ as $\iota\downarrow0$ such that $0\leq\pi_\iota'\leq1$, $\pi_\iota(r)\to r^+$, and $\pi_\iota'(r)\to\one_{\{r>0\}}$.
For fixed $(y,s)$, use
\[
\eta(r)=-q_\delta\bigl(U^+-\pi_\iota(r)\bigr)
\]
in the entropy inequality for $\tilde u$. 
The $x$-test is compactly supported in $D$. 
Letting $\iota\downarrow0$ and using dominated convergence theorem gives
\begin{equation}\label{eq:entropy for r+}
\begin{aligned}
&\int\partial_t\Phi\int_{\tilde\xi(x)}^Vq_\delta(U^+-r^+)h_l(r)\,\dd\beta(r)+\int q_\delta(U^+-\psi^+(x))\Phi\,\dd\tilde\nu(x,t)\,\dd y\,\dd s\\
&\leq\int h_l(V)A(Q)\cdot\nabla_x\left[q_\delta(U^+-V^+)\Phi\right]\\
&\quad-\int\tilde f(x,t,V)h_l(V)q_\delta(U^+-V^+)\Phi+\int_{D_T}\chi(x)\varphi(t)\sigma_{\varsigma,j}(x)\,\dd\mu_{l,1}^{\tilde u}(x,t).
\end{aligned}
\end{equation}

\medskip
\noindent
\textbf{Step 2. Extraction of $\mathscr L_{l,\tilde u}^{-}$.}

For every $k\geq0$ and $a,b\in\mathbb R$, we have
\[
\begin{aligned}
\int_b^a q_\delta(k-r^+)h_l(r)\,\dd\beta(r) &=\int_{b^+}^{a^+}q_\delta(k-r)h_l(r)\,\dd\beta(r)+q_\delta(k)\int_b^a\one_{\{r<0\}}h_l(r)\,\dd\beta(r).
\end{aligned}
\]
Define
\[
\zeta_{\delta,\theta,\varsigma}^{-}(x,t):=\int_{D_T}q_\delta(U^+)\Phi(x,t,y,s)\,\dd y\,\dd s.
\]
Then we have
\[
0\leq\zeta_{\delta,\theta,\varsigma}^{-}\leq\chi\varphi\sigma_{\varsigma,j}.
\]
The time term of \eqref{eq:entropy for r+} splits as
\[
\begin{aligned}
&\int\partial_t\Phi\int_{\tilde\xi(x)}^Vq_\delta(U^+-r^+)h_l(r)\,\dd\beta(r)\\
&=\int\partial_t\Phi\int_{\tilde\xi^+(x)}^{V^+}q_\delta(U^+-r)h_l(r)\,\dd\beta(r)+\int_{D_T}\partial_t\zeta_{\delta,\theta,\varsigma}^{-}\left[\int_{\tilde\xi(x)}^V\one_{\{r<0\}}h_l(r)\,\dd\beta(r)\right]\dd x\,\dd t.
\end{aligned}
\]
The obstacle factor satisfies
\[
q_\delta(U^+-\psi^+(x))=\one_{\{\psi(x)\geq0\}}q_\delta(U^+-\psi(x))+\one_{\{\psi(x)<0\}}q_\delta(U^+).
\]
Moreover, note that
\[
\nabla_xq_\delta(U^+-V^+)=-q_\delta'(U^+-V^+)\one_{\{V>0\}}Q.
\]
Since $\nabla V=0$ almost everywhere on $\{V=0\}$ and $A(0)=0$, we have
\[
\begin{aligned}
&\int h_l(V)A(Q)\cdot\nabla_x\left[q_\delta(U^+-V^+)\Phi\right]\\
&=\int\mathbf A_l^+(V)\cdot\nabla_x\left[q_\delta(U^+-V^+)\Phi\right]+\int_{D_T}\one_{\{V<0\}}h_l(V)A(Q)\cdot\nabla_x\zeta_{\delta,\theta,\varsigma}^{-}\,\dd x\,\dd t.
\end{aligned}
\]
The source term splits exactly as
\[
\begin{aligned}
&-\int\tilde f(x,t,V)h_l(V)q_\delta(U^+-V^+)\Phi\\
&=-\int\one_{\{V\geq0\}}\tilde f(x,t,V)h_l(V)q_\delta(U^+-V^+)\Phi-\int_{D_T}\one_{\{V<0\}}\tilde f(x,t,V)h_l(V)\zeta_{\delta,\theta,\varsigma}^{-}\,\dd x\,\dd t.
\end{aligned}
\]

Define
\[
\begin{aligned}
\mathcal T_{\delta,\theta,\varsigma}^{+}&:=-\int\partial_s\Phi\int_{\xi^+(y)}^{U^+}q_\delta(r-V^+)h_l(r)\,\dd\beta(r)+\int\partial_t\Phi\int_{\tilde\xi^+(x)}^{V^+}q_\delta(U^+-r)h_l(r)\,\dd\beta(r),\\
\mathcal O_{\delta,\theta,\varsigma}^{+}&:=-\int q_\delta(\psi^+(y)-V^+)\Phi\,\dd\nu(y,s)\,\dd x\,\dd t+\int_{\{\psi(x)\geq0\}}q_\delta(U^+-\psi(x))\Phi\,\dd\tilde\nu(x,t)\,\dd y\,\dd s,\\
\mathcal A_{\delta,\theta,\varsigma}^{+}&:=-\int\mathbf A_l^+(U)\cdot\nabla_y\left[q_\delta(U^+-V^+)\Phi\right]+\int\mathbf A_l^+(V)\cdot\nabla_x\left[q_\delta(U^+-V^+)\Phi\right],\\
\mathcal S_{\delta,\theta,\varsigma}^{+}&:=\int q_\delta(U^+-V^+)\Phi\Big[\one_{\{U>0\}}f(y,s,U)h_l(U)-\one_{\{V\geq0\}}\tilde f(x,t,V)h_l(V)\Big].
\end{aligned}
\]
Then \eqref{eq:entropy for u-v+} and \eqref{eq:entropy for r+} yield
\begin{equation}\label{eq:positive part entropy}
\begin{aligned}
\mathcal T_{\delta,\theta,\varsigma}^{+}+\mathcal O_{\delta,\theta,\varsigma}^{+}-\mathcal A_{\delta,\theta,\varsigma}^{+}-\mathcal S_{\delta,\theta,\varsigma}^{+}&\leq\mathscr L_{l,\tilde u}^{-}\left(\zeta_{\delta,\theta,\varsigma}^{-}\right)+C\mathfrak e_l\\
&\leq\mathscr L_{l,\tilde u}^{-}\left(\chi\varphi\sigma_{\varsigma,j}\right)+C\mathfrak e_l.
\end{aligned}
\end{equation}

\medskip
\noindent
\textbf{Step 3. Negative-part inequalities.}

We interchange $u$ and $\tilde{u}$, and write
\[
\tilde U:=\tilde u(y,s),\qquad W:=u(x,t),\qquad \tilde P:=\nabla_y\tilde u(y,s),\qquad R:=\nabla_xu(x,t).
\]
For fixed $(x,t)$, use
\[
\eta(r)=-q_\delta(-r-W^-)=-q_\delta(r^--W^-)
\]
in the entropy inequality for $\tilde u$. 
This entropy belongs to $\mathcal E_0$, and hence
\begin{equation}\label{eq:entropy for r--W-}
\begin{aligned}
&\int\partial_s\Phi\int_{\tilde\xi(y)}^{\tilde U}q_\delta(r^--W^-)h_l(r)\,\dd\beta(r)+\int q_\delta(\psi^-(y)-W^-)\Phi\,\dd\tilde\nu(y,s)\,\dd x\,\dd t\\
&\leq\int h_l(\tilde U)A(\tilde P)\cdot\nabla_y\left[q_\delta(\tilde U^--W^-)\Phi\right]\\
&\quad-\int\tilde f(y,s,\tilde U)h_l(\tilde U)q_\delta(\tilde U^--W^-)\Phi+C\mu_{l,1}^{\tilde u}(D_T).
\end{aligned}
\end{equation}

Let $\pi_\iota^-$ be a smooth approximation of $r^-$ with $-1\leq(\pi_\iota^-)' \leq0$ and $\pi_\iota^-(r)\to r^-$. 
Use
\[
\eta(r)=q_\delta\bigl(\tilde U^--\pi_\iota^-(r)\bigr)
\]
in the entropy inequality for $u$. 
Letting $\iota\downarrow0$ and using dominated convergence theorem gives
\begin{equation}\label{eq:entropy for u--r-}
\begin{aligned}
&-\int\partial_t\Phi\int_{\xi(x)}^W q_\delta(\tilde U^--r^-)h_l(r)\,\dd\beta(r)-\int q_\delta(\tilde U^--\psi^-(x))\Phi\,\dd\nu(x,t)\,\dd y\,\dd s\\
&\leq-\int h_l(W)A(R)\cdot\nabla_x\left[q_\delta(\tilde U^--W^-)\Phi\right]\\
&\quad+\int f(x,t,W)h_l(W)q_\delta(\tilde U^--W^-)\Phi+C\mu_{l,1}^u(D_T).
\end{aligned}
\end{equation}

\medskip
\noindent
\textbf{Step 4. Extraction of $\mathscr L_{l,u}^{+}$.}

The time primitive satisfies
\[
\begin{aligned}
&\int_{\xi(x)}^Wq_\delta(\tilde U^--r^-)h_l(r)\,\dd\beta(r)\\
&=\int_{-\xi^-(x)}^{-W^-}q_\delta(\tilde U^-+r)h_l(r)\,\dd\beta(r)+q_\delta(\tilde U^-)\int_{\xi(x)}^W\one_{\{r>0\}}h_l(r)\,\dd\beta(r).
\end{aligned}
\]
Set
\[
\zeta_{\delta,\theta,\varsigma}^{+}(x,t):=\int_{D_T}q_\delta(\tilde U^-)\Phi(x,t,y,s)\,\dd y\,\dd s.
\]
Then we have
\[
0\leq\zeta_{\delta,\theta,\varsigma}^{+}\leq\chi\varphi\sigma_{\varsigma,j}.
\]
Note that
\[
\begin{aligned}
q_\delta(\tilde U^--\psi^-(x))&=\one_{\{\psi(x)\leq0\}}q_\delta(\tilde U^-+\psi(x))+\one_{\{\psi(x)>0\}}q_\delta(\tilde U^-),
\end{aligned}
\]
and
\[
\nabla_xq_\delta(\tilde U^--W^-)=q_\delta'(\tilde U^--W^-)\one_{\{W<0\}}R.
\]

Define
\[
\begin{aligned}
\mathcal T_{\delta,\theta,\varsigma}^{-}&:=\int\partial_s\Phi\int_{-\tilde\xi^-(y)}^{-\tilde U^-}q_\delta(-r-W^-)h_l(r)\,\dd\beta(r)-\int\partial_t\Phi\int_{-\xi^-(x)}^{-W^-}q_\delta(\tilde U^-+r)h_l(r)\,\dd\beta(r),\\
\mathcal O_{\delta,\theta,\varsigma}^{-}&:=\int q_\delta(\psi^-(y)-W^-)\Phi\,\dd\tilde\nu(y,s)\,\dd x\,\dd t-\int_{\{\psi(x)\leq0\}}q_\delta(\tilde U^-+\psi(x))\Phi\,\dd\nu(x,t)\,\dd y\,\dd s,\\
\mathcal A_{\delta,\theta,\varsigma}^{-}&:=-\int\mathbf A_l^-(\tilde U)\cdot\nabla_y\left[q_\delta(\tilde U^--W^-)\Phi\right]+\int\mathbf A_l^-(W)\cdot\nabla_x\left[q_\delta(\tilde U^--W^-)\Phi\right],\\
\mathcal S_{\delta,\theta,\varsigma}^{-}&:=\int q_\delta(\tilde U^--W^-)\Phi\Big[\one_{\{W\leq0\}}f(x,t,W)h_l(W)-\one_{\{\tilde U<0\}}\tilde f(y,s,\tilde U)h_l(\tilde U)\Big].
\end{aligned}
\]
The remaining positive-level terms form $\mathscr L_{l,u}^{+}(\zeta_{\delta,\theta,\varsigma}^{+})$, and therefore combining \eqref{eq:entropy for r--W-} and \eqref{eq:entropy for u--r-}, we have
\begin{equation}\label{eq:negative part entropy}
\begin{aligned}
\mathcal T_{\delta,\theta,\varsigma}^{-}+\mathcal O_{\delta,\theta,\varsigma}^{-}-\mathcal A_{\delta,\theta,\varsigma}^{-}-\mathcal S_{\delta,\theta,\varsigma}^{-}&\leq\mathscr L_{l,u}^{+}\left(\zeta_{\delta,\theta,\varsigma}^{+}\right)+C\mathfrak e_l\\
&\leq\mathscr L_{l,u}^{+}\left(\chi\varphi\sigma_{\varsigma,j}\right)+C\mathfrak e_l.
\end{aligned}
\end{equation}

Adding \eqref{eq:positive part entropy} and \eqref{eq:negative part entropy} gives
\[
\begin{aligned}
&\mathcal T_{\delta,\theta,\varsigma}^{+}+\mathcal T_{\delta,\theta,\varsigma}^{-}+\mathcal O_{\delta,\theta,\varsigma}^{+}+\mathcal O_{\delta,\theta,\varsigma}^{-}-\mathcal A_{\delta,\theta,\varsigma}^{+}-\mathcal A_{\delta,\theta,\varsigma}^{-}-\mathcal S_{\delta,\theta,\varsigma}^{+}-\mathcal S_{\delta,\theta,\varsigma}^{-}\\
&\leq\mathscr L_{l,\tilde u}^{-}\left(\chi\varphi\sigma_{\varsigma,j}\right)+\mathscr L_{l,u}^{+}\left(\chi\varphi\sigma_{\varsigma,j}\right)+C\mathfrak e_l.
\end{aligned}
\]

\medskip
\noindent
\textbf{Step 5. Time and Obstacle terms.}

Define
\[
\begin{aligned}
\mathcal B_{\delta,l}^{+,1}(a,b)&:=\int_{b^+}^{a^+}q_\delta(r-b^+)h_l(r)\,\dd\beta(r),\\
\mathcal B_{\delta,l}^{+,2}(a,b)&:=\int_{b^+}^{a^+}q_\delta(a^+-r)h_l(r)\,\dd\beta(r),\\
\mathcal B_l^+(a,b)&:=\one_{\{a^+>b^+\}}\int_{b^+}^{a^+}h_l(r)\,\dd\beta(r),
\end{aligned}
\]
and
\[
\begin{aligned}
\mathcal B_{\delta,l}^{-,1}(a,b)&:=\int_{-b^-}^{-a^-}q_\delta(-r-a^-)h_l(r)\,\dd\beta(r),\\
\mathcal B_{\delta,l}^{-,2}(a,b)&:=\int_{-b^-}^{-a^-}q_\delta(b^-+r)h_l(r)\,\dd\beta(r),\\
\mathcal B_l^-(a,b)&:=\one_{\{b^->a^-\}}\int_{-b^-}^{-a^-}h_l(r)\,\dd\beta(r).
\end{aligned}
\]
Splitting the positive time primitives at $V^+$ and $U^+$ gives
\[
\begin{aligned}
\mathcal T_{\delta,\theta,\varsigma}^{+}&=-\int\mathcal B_{\delta,l}^{+,1}(U,V)\partial_s\Phi-\int\mathcal B_{\delta,l}^{+,2}(U,V)\partial_t\Phi.
\end{aligned}
\]
The initial-data pieces vanish exactly as in the proof of Lemma \ref{lem:interior-local-Kato}: the term containing $\xi(y)$ is independent of $s$, the term containing $\tilde\xi(x)$ is independent of $t$, and $\Phi$ vanishes at the time endpoints.

Likewise, we have
\[
\mathcal T_{\delta,\theta,\varsigma}^{-}=-\int\mathcal B_{\delta,l}^{-,1}(W,\tilde U)\partial_s\Phi-\int\mathcal B_{\delta,l}^{-,2}(W,\tilde U)\partial_t\Phi.
\]
For $i=1,2$, note that
\[
\left|\mathcal B_{\delta,l}^{\pm,i}(a,b)-\mathcal B_l^\pm(a,b)\right|\leq\omega_{\beta,l}(2\delta).
\]
Thus, we have
\[
\begin{aligned}
\mathcal T_{\delta,\theta,\varsigma}^{+}&=-\int\mathcal B_l^+(U,V)(\partial_t+\partial_s)\Phi+R_{\delta,\theta,\varsigma}^{t,+},\\
\mathcal T_{\delta,\theta,\varsigma}^{-}&=-\int\mathcal B_l^-(W,\tilde U)(\partial_t+\partial_s)\Phi+R_{\delta,\theta,\varsigma}^{t,-},
\end{aligned}
\]
where
\[
\left|R_{\delta,\theta,\varsigma}^{t,+}\right|+\left|R_{\delta,\theta,\varsigma}^{t,-}\right|\leq C\omega_{\beta,l}(2\delta)\left(1+\theta^{-1}\right),
\]
and
\[
(\partial_t+\partial_s)\Phi=\chi(x)\rho_{\varsigma,j}(x,y)\varphi'(t)\varrho_\theta(s-t).
\]

For the obstacle terms, since $V^+\geq\psi^+(x)$, $\tilde U^-\leq\psi^-(y)$, and the maps $r\mapsto r^\pm$ are $1$-Lipschitz, following the proof of  Lemma \ref{lem:interior-local-Kato}, we have
\[
-\mathcal O_{\delta,\theta,\varsigma}^{+}-\mathcal O_{\delta,\theta,\varsigma}^{-}\leq C\frac{\omega_\psi(C\varsigma)}{\delta}\left[\nu(D\times[0,T))+\tilde\nu(D\times[0,T))\right].
\]

\medskip
\noindent
\textbf{Step 6. Diffusion terms.}

Using $\nabla_x\rho_{\varsigma,j} +\nabla_y\rho_{\varsigma,j}=0$, one obtains
\[
\mathcal A_{\delta,\theta,\varsigma}^{+}=-\mathcal D_{\delta,\theta,\varsigma}^{+}+\mathcal G_{\delta,\theta,\varsigma}^{+},
\]
where
\[
\begin{aligned}
\mathcal D_{\delta,\theta,\varsigma}^{+}&:=\int q_\delta'(U^+-V^+)\Phi\left[\mathbf A_l^+(U)-\mathbf A_l^+(V)\right]\cdot\left[\nabla U^+-\nabla V^+\right],\\
\mathcal G_{\delta,\theta,\varsigma}^{+}&:=\int q_\delta(U^+-V^+)\left[\mathbf A_l^+(V)-\mathbf A_l^+(U)\right]\cdot\nabla\chi\rho_{\varsigma,j}(x,y)\varphi(t)\varrho_\theta(s-t).
\end{aligned}
\]
Using $\one_{\{v>0\}}A(\nabla v)=A(\nabla v^+)$, as in the proof of  Lemma \ref{lem:interior-local-Kato}, the monotone part of $\mathcal D_{\delta,\theta,\varsigma}^{+}$ is nonnegative. 
Its truncation mismatch satisfies
\[
\begin{aligned}
|E_{\delta,\theta,\varsigma}^{+}|\leq C\int\one_{\{0<U^+-V^+<\delta\}}\Phi\left[1+|\nabla\mathcal T_{l+2}(u(y,s))|^p+|\nabla\mathcal T_{l+2}(\tilde u(x,t))|^p\right].
\end{aligned}
\]

Similarly, we have
\[
\mathcal A_{\delta,\theta,\varsigma}^{-}=-\mathcal D_{\delta,\theta,\varsigma}^{-}+\mathcal G_{\delta,\theta,\varsigma}^{-},
\]
where
\[
\begin{aligned}
\mathcal G_{\delta,\theta,\varsigma}^{-}&:=\int q_\delta(\tilde U^--W^-)\left[\mathbf A_l^-(W)-\mathbf A_l^-(\tilde U)\right]\cdot\nabla\chi\rho_{\varsigma,j}(x,y)\varphi(t)\varrho_\theta(s-t),
\end{aligned}
\]
and
\[
\begin{aligned}
|E_{\delta,\theta,\varsigma}^{-}|\leq C\int&\one_{\{0<\tilde U^--W^-<\delta\}}\Phi\left[1+|\nabla\mathcal T_{l+2}(\tilde u(y,s))|^p+|\nabla\mathcal T_{l+2}(u(x,t))|^p\right].
\end{aligned}
\]
For fixed $\theta$, first take $\varsigma\downarrow0$ and then $\delta\downarrow0$.
The preceding bounds give
\[
\lim_{\delta\downarrow0}\limsup_{\varsigma\downarrow0}\bigl(|E_{\delta,\theta,\varsigma}^{+}|+|E_{\delta,\theta,\varsigma}^{-}|\bigr)=0.
\]
\medskip
\noindent
\textbf{Step 7. Source terms.}

After $\varsigma\downarrow0$ and then $\delta\downarrow0$, the continuity of translations in $L^1(D_T)$ and the dominated convergence theorem show that the positive source term becomes
\[
\begin{aligned}
\mathcal S_\theta^{+}:=\int_0^T\!\int_0^T\!\int_D&\one_{\{u^+(x,s)>\tilde u^+(x,t)\}}\chi(x)\varphi(t)\varrho_\theta(s-t)\\
&\times\Big[\one_{\{u(x,s)>0\}}f(x,s,u(x,s))h_l(u(x,s))\\
&\hspace{5em}-\one_{\{\tilde u(x,t)\geq0\}}\tilde f(x,t,\tilde u(x,t))h_l(\tilde u(x,t))\Big]\,\dd x\,\dd s\,\dd t.
\end{aligned}
\]
The negative source term becomes
\[
\begin{aligned}
\mathcal S_\theta^{-}:=\int_0^T\!\int_0^T\!\int_D&\one_{\{\tilde u^-(x,s)>u^-(x,t)\}}\chi(x)\varphi(t)\varrho_\theta(s-t)\\
&\times\Big[\one_{\{u(x,t)\leq0\}}f(x,t,u(x,t))h_l(u(x,t))\\
&\hspace{5em}-\one_{\{\tilde u(x,s)<0\}}\tilde f(x,s,\tilde u(x,s))h_l(\tilde u(x,s))\Big]\,\dd x\,\dd s\,\dd t.
\end{aligned}
\]

Interchanging $s$ and $t$ in $\mathcal S_\theta^{-}$ and using the evenness of $\varrho_\theta$ gives the same integral with $\varphi(s)$ in place of $\varphi(t)$. 
Since $|s-t|\leq\theta$ on the support of $\varrho_\theta$, we have
\[
\left|\mathcal S_\theta^{-}-\widehat{\mathcal S}_\theta^{-}\right|\leq C_l\theta\longrightarrow0,\qquad\text{as }\theta\downarrow0,
\]
where $\widehat{\mathcal S}_\theta^{-}$ is obtained by replacing $\varphi(s)$ by $\varphi(t)$. 
This follows from $|\varphi(s)-\varphi(t)|\leq C\theta$ and the integrability of the terms involving $f$ and $\tilde{f}$. 
Note that for every $a,b,F,G\in\mathbb R$, we have
\[
\begin{aligned}
&\one_{\{a^+>b^+\}}\left[\one_{\{a>0\}}F-\one_{\{b\geq0\}}G\right]+\one_{\{b^->a^-\}}\left[\one_{\{a\leq0\}}F-\one_{\{b<0\}}G\right]=\one_{\{a>b\}}(F-G).
\end{aligned}
\]
There exists a remainder $R_{l,\theta}$ such that $|R_{l,\theta}|\leq C_l\theta$ and
\begin{align*}
\mathcal S_\theta^{+}+\mathcal S_\theta^{-}&=\int_0^T\!\int_0^T\!\int_D\one_{\{u(x,s)>\tilde u(x,t)\}}\chi(x)\varphi(t)\varrho_\theta(s-t)\\
&\quad\times\Big[ f(x,s,u(x,s))h_l(u(x,s)) - \tilde f(x,t,\tilde u(x,t))h_l(\tilde u(x,t))\Big]\,\dd x\,\dd s\,\dd t+R_{l,\theta}.
\end{align*}
On $\{u=\tilde u\}$, the difference of the two diagonal source densities is nonpositive. 
Lemma \ref{lem:time-diagonal-source} therefore gives
\[
\begin{aligned}
\limsup_{\theta\downarrow0}\left(\mathcal S_\theta^{+}+\mathcal S_\theta^{-}\right)\leq\int_{D_T}\mathscr S_l(u,\tilde u)\chi\varphi\,\dd x\,\dd t.
\end{aligned}
\]

\medskip
\noindent
\textbf{Step 8. Passage to the limit.}

We take the limits in the order
\[
\varsigma\downarrow0,\qquad\delta\downarrow0,\qquad\theta\downarrow0.
\]
For fixed $\delta$ and $\theta$, the inward kernels are spatial approximate identities. 
The obstacle error tends to zero because $\delta$ is fixed and $\omega_\psi(C\varsigma)\to0$.

For fixed $\theta$, the time commutator and the two truncation mismatches tend to zero as $\delta\downarrow0$.
Finally, the time approximate identity yields as $\theta\downarrow0$
\[
\begin{aligned}
&\int_0^T\!\int_0^T\!\int_D\mathcal B_l^+\bigl(u(x,s),\tilde u(x,t)\bigr)\chi(x)\varphi'(t)\varrho_\theta(s-t)\,\dd x\,\dd s\,\dd t\longrightarrow\int_{D_T}\mathcal B_l^+(u,\tilde u)\chi\varphi'\,\dd x\,\dd t
\end{aligned}
\]
and
\[
\begin{aligned}
&\int_0^T\!\int_0^T\!\int_D\mathcal B_l^-\bigl(u(x,t),\tilde u(x,s)\bigr)\chi(x)\varphi'(t)\varrho_\theta(s-t)\,\dd x\,\dd s\,\dd t\longrightarrow\int_{D_T}\mathcal B_l^-(u,\tilde u)\chi\varphi'\,\dd x\,\dd t.
\end{aligned}
\]
Note that for every $a,b\in\mathbb R$,
\[
\mathcal B_l^+(a,b)+\mathcal B_l^-(a,b)=\mathcal B_l(a,b),
\]
and the localization fluxes satisfy
\[
\begin{aligned}
&\one_{\{u^+>\tilde u^+\}}\left[\mathbf A_l^+(\tilde u)-\mathbf A_l^+(u)\right]+\one_{\{\tilde u^->u^-\}}\left[\mathbf A_l^-(u)-\mathbf A_l^-(\tilde u)\right]=\mathscr G_l(u,\tilde u).
\end{aligned}
\]
Combining the time, obstacle, diffusion, and source limits gives
\[
\begin{aligned}
\mathscr F_l(\chi,\varphi)&\leq\limsup_{\varsigma\downarrow0}\mathscr L_{l,\tilde u}^{-}\left(\chi\varphi\sigma_{\varsigma,j}\right)+\limsup_{\varsigma\downarrow0}\mathscr L_{l,u}^{+}\left(\chi\varphi\sigma_{\varsigma,j}\right)+C\mathfrak e_l,
\end{aligned}
\]
which completes the proof.
\end{proof}

\subsection{Elimination of the boundary zero-level terms}
\begin{lem}[Boundary localization]\label{lem:boundary-localization}
Let $j\in\{0,\ldots,J\}$. For every nonnegative $\varphi\in C_c^\infty((0,T))$,
\[
\mathscr F_l(\chi_j,\varphi)\leq C\mathfrak e_l,
\]
where $C$ depends on $\Vert\varphi\Vert_{L^\infty(0,T)}$ and the fixed partition of unity.
\end{lem}
\begin{proof}
For $j=0$, the result follows from Lemma \ref{lem:interior-local-Kato}. 
Fix $j\in\{1,\ldots,J\}$ and $\widehat\varsigma>0$. 
Since $\chi_j\sigma_{\widehat\varsigma,j}\in C_c^\infty(D)$, Lemma \ref{lem:interior-local-Kato} gives
\[
\mathscr F_l\bigl(\chi_j\sigma_{\widehat\varsigma,j},\varphi\bigr)\leq C\mathfrak e_l.
\]
By linearity, we have
\[
\begin{aligned}
\mathscr F_l(\chi_j,\varphi)&=\mathscr F_l\bigl(\chi_j\sigma_{\widehat\varsigma,j},\varphi\bigr)+\mathscr F_l\bigl(\chi_j(1-\sigma_{\widehat\varsigma,j}),\varphi\bigr).
\end{aligned}
\]
Apply Lemma \ref{lem:chartwise-signed-Kato} to the second term on the right-hand side, we have
\[
\begin{aligned}
\mathscr F_l(\chi_j,\varphi)&\leq\limsup_{\varsigma\downarrow0}\mathscr L_{l,\tilde u}^{-}\left(\chi_j\varphi(1-\sigma_{\widehat\varsigma,j})\sigma_{\varsigma,j}\right)\\
&\quad+\limsup_{\varsigma\downarrow0}\mathscr L_{l,u}^{+}\left(\chi_j\varphi(1-\sigma_{\widehat\varsigma,j})\sigma_{\varsigma,j}\right)+C\mathfrak e_l.
\end{aligned}
\]
Letting $\widehat\varsigma\downarrow0$ and using Lemma \ref{lem:L-boundary-vanishing} proves the assertion.
\end{proof}
\begin{thm}[Global truncated Kato inequality]
For every nonnegative $\varphi\in C_c^\infty((0,T))$,
\begin{align}
&-\int_{D_T}\mathcal B_l(u,\tilde u)\varphi'(t)\,\dd x\,\dd t\notag\\
&\leq\int_{D_T}\one_{\{u>\tilde u\}}\left[f(x,t,u)h_l(u)-\tilde f(x,t,\tilde u)h_l(\tilde u)\right]\varphi(t)\,\dd x\,\dd t+C\mathfrak e_l.\label{eq:global-truncated-Kato}
\end{align}
\end{thm}
\begin{proof}
Apply Lemma \ref{lem:boundary-localization} to every $\chi_j$ and sum over $j$. 
Since $\sum_{j=0}^J\chi_j=1$, the time and source terms become global.
Then we have
\[
\sum_{j=0}^J\int_{D_T}\mathscr G_l(u,\tilde u)\cdot\nabla\chi_j\,\varphi\,\dd x\,\dd t=0.
\]
This proves \eqref{eq:global-truncated-Kato}.
\end{proof}

\subsection{Averaged strong initial trace}
\begin{thm}[Averaged strong initial trace]\label{thm:averaged-initial-trace}
Let $(u,\nu)$ be a renormalized entropy solution in the sense of Definition \ref{def:entropy solution with ob}. 
Suppose $\xi\geq\psi$ a.e. in $D$.
Then
\begin{equation}\label{eq:averaged-strong-initial-trace}
\lim_{\tau\downarrow0}\frac1\tau\int_0^\tau\|\beta(u(t))-\beta(\xi)\|_{L^1(D)}\,\dd t=0.
\end{equation}
\end{thm}
\begin{proof}
Fix $l>M+1$. Based on \eqref{eq:local-beta-l}, the function $H_l$ is bounded, continuous and nondecreasing. 
Note that $q_\delta$ is a smooth approximation of $\one_{(0,\infty)}$ satisfying \eqref{eq:property of q1}--\eqref{eq:property of q2}.
For $a,b\in\mathbb R$, denote the modulus of continuity of $H_l$ by
\[
\omega_l(s):=\sup_{|a-b|\leq s}|H_l(a)-H_l(b)|.
\]
Then, we have
\begin{equation*}
\int_b^a q_\delta(r-b)\,\dd H_l(r)-\int_b^a q_\delta(b-r)\,\dd H_l(r)\leq |H_l(a)-H_l(b)|,
\end{equation*}
and
\begin{equation*}
|H_l(a)-H_l(b)|\leq \int_b^a q_\delta(r-b)\,\dd H_l(r)-\int_b^a q_\delta(b-r)\,\dd H_l(r)+2\omega_l(\delta).
\end{equation*}

Based on the definition of $K_\varsigma$, we have that $x\mapsto K_\varsigma(x,y)$ is compactly supported in $D$ for all sufficiently small $\varsigma$ and each $y\in D$.
Set
\[
\Sigma_\varsigma(x):=\int_DK_\varsigma(x,y)\,\dd y.
\]
Then, we have
\begin{equation}\label{eq:kernel-properties-short}
\sup_{\varsigma>0}\|\Sigma_\varsigma\|_{L^\infty(D)}<\infty,\qquad \|1-\Sigma_\varsigma\|_{L^1(D)}\longrightarrow0.
\end{equation}

Choose the Lipschitz time cutoff
\[
\zeta_\tau(t)=\left(1-{t}/{\tau}\right)^+.
\]
It is enough to use smooth approximations of this cutoff. 
For fixed $y\in D$, apply the entropy inequality \eqref{eq:entropy formula-1} first with
\[
\eta_y^+(r)=q_\delta(r-\xi(y)),\qquad\phi(x,t)=K_\varsigma(x,y)\zeta_\tau(t),
\]
and then with
\[
\eta_y^-(r)=-q_\delta(\xi(y)-r),\qquad\phi(x,t)=K_\varsigma(x,y)\zeta_\tau(t),
\]
which are admissible since $\phi$ is compactly supported in $D$.
We now add the two entropy inequalities and integrate them with respect to $y$.
Since $\xi(y)\geq\psi(y)$ and $|x-y|\leq C\varsigma$ on the support of $K_\varsigma$, we have
\[
q_\delta(\psi(x)-\xi(y))\leq q_\delta(\psi(x)-\psi(y))\leq C\frac{\omega_\psi(C\varsigma)}{\delta}.
\]
Therefore, the positive reaction contribution is bounded by
\[
C\frac{\omega_\psi(C\varsigma)}{\delta}\nu(D\times[0,T)).
\]
The reaction contribution arising from the negative entropy is nonpositive, and therefore we may discard it.
Moreover, since $A(\nabla u)\cdot\nabla u\geq0$, both entropy-dissipation terms are nonpositive, and therefore we may discard them as well.

Since the spatial derivative of the inward kernel is of order $\varsigma^{-1}$, the remaining localization flux is bounded by
\begin{equation}\label{eq:bounded-A-average-initial}
\frac{C}{\varsigma}\int_0^\tau\int_D h_l(u)|A(\nabla u)|\,\dd x\,\dd t.
\end{equation}
Finally, the time derivative of $\zeta_\tau$ produces the averaged term over $(0,\tau)$.
Combining these observations, we obtain
\begin{align}
&\frac1\tau\int_0^\tau\int_{D\times D}K_\varsigma(x,y)|H_l(u(x,t))-H_l(\xi(y))|\,\dd x\,\dd y\,\dd t\notag\\
&\leq\int_{D\times D}K_\varsigma(x,y)|H_l(\xi(x))-H_l(\xi(y))|\,\dd x\,\dd y\notag\\
&\quad+C\omega_l(\delta)+C\frac{\omega_\psi(C\varsigma)}{\delta}\nu(D\times[0,T))\notag\\
&\quad+\frac{C}{\varsigma}\int_0^\tau\int_Dh_l(u)|A(\nabla u)|\,\dd x\,\dd t\notag\\
&\quad+C\int_0^\tau\int_D|f(x,t,u)|\,\dd x\,\dd t+C\mu_{l,1}^u(D\times[0,T)).\label{eq:truncated-doubled-trace-short}
\end{align}

By the triangle inequality,
\begin{align}
&\frac1\tau\int_0^\tau\|H_l(u(t))-H_l(\xi)\|_{L^1(D)}\,\dd t\notag\\
&\leq\frac1\tau\int_0^\tau\int_{D\times D}K_\varsigma(x,y)|H_l(u(x,t))-H_l(\xi(y))|\,\dd x\,\dd y\,\dd t\notag\\
&\quad+\int_{D\times D}K_\varsigma(x,y)|H_l(\xi(x))-H_l(\xi(y))|\,\dd x\,\dd y+2\|H_l\|_{L^\infty(\mathbb R)}\|1-\Sigma_\varsigma\|_{L^1(D)}.
\label{eq:diagonal-short}
\end{align}

Combining \eqref{eq:truncated-doubled-trace-short} and \eqref{eq:diagonal-short}, and then taking the limits in the order
\[
\tau\downarrow0,\qquad\varsigma\downarrow0,\qquad\delta\downarrow0,
\]
we obtain
\begin{equation}\label{eq:truncated-initial-trace-short}
\limsup_{\tau\downarrow0}\frac1\tau\int_0^\tau\|H_l(u(t))-H_l(\xi)\|_{L^1(D)}\,\dd t\leq C\mu_{l,1}^u(D\times[0,T)).
\end{equation}
Indeed, $h_l(u)A(\nabla u)\in L^1(D_T)^d$ by \eqref{eq:A-growth-uniqueness} and $\mathcal T_{l+1}(u)\in L^p(0,T;W_0^{1,p}(D))$.
Thus, for fixed $\varsigma>0$, we have
\[
\frac1\varsigma\int_0^\tau\int_Dh_l(u)|A(\nabla u)|\,\dd x\,\dd t\longrightarrow0,\qquad
\int_0^\tau\int_D|f(x,t,u)|\,\dd x\,\dd t\longrightarrow0
\]
as $\tau\downarrow0$. 
The remaining kernel terms vanish by \eqref{eq:kernel-properties-short} and the property of $K_\varsigma$, while ${\omega_\psi(C\varsigma)}/{\delta}\rightarrow0$ for fixed $\delta$, and $\omega_l(\delta)\to0$ as $\delta\downarrow0$.

We next prove a one-sided initial tail estimate: For every $k>M$,
\begin{equation}\label{eq:one-sided-initial-tail-short}
\limsup_{\tau\downarrow0}\frac1\tau\int_0^\tau\int_D\bigl(\beta(u(x,t))-\beta(k)\bigr)^+\,\dd x\,\dd t\leq\int_D\bigl(\beta(\xi(x))-\beta(k)\bigr)^+\,\dd x.
\end{equation}
For $n>k+1$, use the entropy inequality \eqref{eq:entropy formula-1} with $h_n$ and
\[
\eta(r)=q_\delta(r-k),\qquad \phi(x,t)=\zeta_\tau(t),
\]
by a smooth approximation of $\zeta_\tau$.
This is an admissible boundary test function for $\eta$ because $\eta(0)=0$.
Then, we have
\begin{align*}
\frac1\tau\int_0^\tau\int_D\int_k^uq_{\delta}(r-k)h_n(r)\,\dd\beta(r)\,\dd x\,\dd t&\leq\int_D\int_k^\xi q_{\delta}(r-k)h_n(r)\,\dd\beta(r)\,\dd x\\
&\quad+\int_0^\tau\int_D|f(x,t,u)|\,\dd x\,\dd t+\mu_{n,1}^u(D\times[0,T)).
\end{align*}
The obstacle term vanishes because $\psi\leq M<k$, and the diffusion term is nonpositive. 
Letting successively
\[
n\to\infty,\qquad\tau\downarrow0,\qquad\delta\downarrow0,
\]
yields \eqref{eq:one-sided-initial-tail-short}. 
Moreover, noting that $H_l(r)=\beta(r)$ on $r\in[-l,l]$, we have for $l\geq M$
\[
|\beta(r)-H_l(r)|\leq\bigl(\beta(r)-\beta(l)\bigr)^+\qquad\text{for every }r\geq-M.
\]
Then, we have
\begin{align*}
|\beta(u)-\beta(\xi)|&\leq|H_l(u)-H_l(\xi)|+\bigl(\beta(u)-\beta(l)\bigr)^++\bigl(\beta(\xi)-\beta(l)\bigr)^+.
\end{align*}
Using \eqref{eq:truncated-initial-trace-short} and \eqref{eq:one-sided-initial-tail-short} with $k=l$, we obtain
\begin{align}
&\limsup_{\tau\downarrow0}\frac1\tau\int_0^\tau\|\beta(u(t))-\beta(\xi)\|_{L^1(D)}\,\dd t\leq C\mu_{l,1}^u(D\times[0,T))+2\int_D\bigl(\beta(\xi)-\beta(l)\bigr)^+\,\dd x.\label{eq:final-trace-l-bound}
\end{align}

We now let $l\to\infty$. The fact $\beta(\xi)\in L^1(D)$, the dominated convergence theorem and \eqref{eq:defect-tail-uniqueness} imply that the right-hand side of \eqref{eq:final-trace-l-bound} tends to zero.
This proves \eqref{eq:averaged-strong-initial-trace}.
\end{proof}

\begin{rem}
The Kato inequalities use spatial localization before time diagonalization. 
The initial term produces a boundary-layer estimate \eqref{eq:bounded-A-average-initial}.
Its control requires the time interval to shrink before the spatial localization is removed. 
We therefore derive the averaged initial trace with the limit order
\[
\tau\downarrow0,\qquad\varsigma\downarrow0,\qquad\delta\downarrow0,
\]
which differs from step 8 in the proof of Lemma \ref{lem:chartwise-signed-Kato}.
On the other hand, this averaged trace supplies the initial value in the distributional Kato inequality and does not require the assumption of $L^1(D)$-continuity on $t\mapsto\beta(u(t))$.
\end{rem}

\subsection{The balance residual and the proof of comparison}
\begin{lem}[Truncated balance residual]\label{lem:truncated-balance-residual}
Let $(v,\nu_v)$ be a renormalized entropy solution with source $f_v$.
For $l>M+1$ and $0\leq\phi\in C_c^\infty(D\times[0,T))$, set
\begin{align*}
\mathscr R_l^v(\phi)&:=-\int_{D_T}H_l(v)\partial_t\phi\,\dd x\,\dd t-\int_D H_l(\xi_v)\phi(x,0)\,\dd x\\
&\quad+\int_{D_T}h_l(v)A(\nabla v)\cdot\nabla\phi\,\dd x\,\dd t-\int_{D_T}f_v(x,t,v)h_l(v)\phi\,\dd x\,\dd t -\int_{D\times[0,T)}\phi\,\dd\nu_v.
\end{align*}
Then
\begin{equation*}
|\mathscr R_l^v(\phi)|\leq\int_{D\times[0,T)}\phi\,\dd\mu_{l,1}^v.
\end{equation*}
\end{lem}
\begin{proof}
Use the constant entropy $\eta=1$ in \eqref{eq:entropy formula-1}.
This gives $\mathscr R_l^v(\phi)\leq\int_{D\times[0,T)}\phi\,\dd\mu_{l,1}^v$.
The constant entropy $\eta=-1$ gives the opposite inequality.
\end{proof}
\begin{proof}[Proof of Theorem \ref{thm:main-comparison}]
Let $L_2$ be the constant in \eqref{eq:f-beta-Lipschitz-uniqueness} for $\tilde f$.
On $\{u>\tilde u\}$,
\begin{align*}
f(x,t,u)-\tilde f(x,t,\tilde u) &=f(x,t,u)-\tilde f(x,t,u)+\tilde f(x,t,u)-\tilde f(x,t,\tilde u)\\
&\leq L_2\bigl(\beta(u)-\beta(\tilde u)\bigr)^+.
\end{align*}
Moreover, as $l\to\infty$, we have
\[
\mathcal B_l(u,\tilde u) \longrightarrow\bigl(\beta(u)-\beta(\tilde u)\bigr)^+ \quad\text{a.e. in }D_T,
\]
and the absolute value is bounded by $|\beta(u)|+|\beta(\tilde u)|$.
The source terms converge in $L^1(D_T)$ by \eqref{eq:source-integrability-uniqueness}.
Letting $l\to\infty$ in \eqref{eq:global-truncated-Kato} gives
\begin{equation}
-\int_0^T Y(t)\varphi'(t)\,\dd t \leq L_2\int_0^T Y(t)\varphi(t)\,\dd t, \label{eq:Y-distribution-inequality}
\end{equation}
where
\[
Y(t):=\int_D \bigl(\beta(u(t))-\beta(\tilde u(t))\bigr)^+\,\dd x.
\]

Let $\mathcal G$ be the set of Lebesgue points of $Y$.
Testing \eqref{eq:Y-distribution-inequality} with smooth approximations of $\one_{[s,t]}$ gives
\begin{equation}\label{eq:before gronwall +}
Y(t)\leq Y(s)+L_2\int_s^tY(r)\,\dd r
\end{equation}
for $0<s<t<T$ with $s,t\in\mathcal G$. 
For almost every $s\in(0,t)$, we have
\begin{align*}
Y(s)&\leq\|\beta(u(s))-\beta(\xi)\|_{L^1(D)}+\int_D(\beta(\xi)-\beta(\tilde\xi))^+\,\dd x\\
&\quad+\|\beta(\tilde u(s))-\beta(\tilde\xi)\|_{L^1(D)}.
\end{align*}
Applying Theorem \ref{thm:averaged-initial-trace} to both $u$ and $\tilde u$, we have
\begin{align*}
\limsup_{\tau\downarrow0}\frac1\tau\int_0^\tau Y(s)\,\dd s\leq\int_D(\beta(\xi)-\beta(\tilde\xi))^+\,\dd x.
\end{align*}
We average the inequality \eqref{eq:before gronwall +} over $s\in(0,\tau)$. 
Since $Y\in L^1(0,T)$, we may let $\tau\downarrow0$.
Therefore, we obtain
\[
Y(t)\leq Y_0+L_2\int_0^tY(r)\,\dd r,\qquad Y_0:=\int_D(\beta(\xi)-\beta(\tilde\xi))^+\,\dd x.
\]
Gr\"onwall's inequality proves \eqref{eq:main-comparison}.

If $f=\tilde{f}$ and $\xi=\tilde{\xi}$, apply the estimate in both directions.
This gives $\beta(u)=\beta(\tilde u)$ a.e.
If $\beta$ is strictly increasing, then $u=\tilde u$ a.e. and the generalized gradients agree.
For $0\leq\phi\in C_c^\infty(D\times[0,T))$, Lemma \ref{lem:truncated-balance-residual} gives
\[
\left|\int_{D\times[0,T)}\phi\,\dd(\nu-\tilde\nu)\right| \leq\int_{D\times[0,T)}\phi\,\dd\mu_{l,1}^u +\int_{D\times[0,T)}\phi\,\dd\mu_{l,1}^{\tilde u}.
\]
Letting $l\to\infty$ proves $\nu=\tilde\nu$.
\end{proof}

\subsection{Measure-free supersolutions}
We introduce a one-sided entropy supersolution class as an ambient class for the minimality result below. 
Let $(g,\tilde{\xi})$ satisfy Assumption \ref{assu:assumption for initial} and let $\tilde{\xi}\geq\psi$ a.e. in $D$. 
\begin{defn}[Measure-free entropy supersolution]\label{def:measure-free-supersolution}
A measurable function $v:D\times[0,T]\to\mathbb R$ is a measure-free entropy supersolution to \eqref{eq:ito}--\eqref{eq:lower obstacle} with data $(g,\tilde{\xi})$ if $v\geq\psi$ a.e.,
\[
\beta(v)\in L^\infty(0,T;L^1(D)),\qquad \mathcal T_k(v)\in L^p(0,T;W_0^{1,p}(D))
\]
for every $k>0$, and the following one-sided entropy condition holds:  
For every $K>0$, there exists a family $(\mu_{l,K}^v)_{l>M+1}\subset\mathcal M_+(D\times[0,T))$ such that for every $l>M+1$, $\eta\in\mathcal E$ with $-K\leq\eta\leq0$, and every test function $\phi$ admissible for $\eta$,
\begin{align}
&-\int_{D_T}\partial_t\phi\int_{\tilde{\xi}(x)}^{v(x,t)}\eta(r)h_l(r)\,\dd\beta(r)\,\dd x\,\dd t \nonumber\\
&\leq-\int_{D_T}h_l(v)A(\nabla v)\cdot\nabla\bigl(\eta(v)\phi\bigr)\,\dd x\,\dd t\nonumber\\
&\quad+\int_{D_T}g(x,t,v)\eta(v)h_l(v)\phi\,\dd x\,\dd t+\int_{D\times[0,T)}\phi\,\dd\mu_{l,K}^{v}, \label{eq:measure-free-supersolution}
\end{align}
and
\[
\lim_{l\to\infty}\mu_{l,K}^{v}(D\times[0,T))=0.
\]
\end{defn}
Here ``measure-free'' means that no reaction measure is included in the definition; the vanishing defect family $\mu_{l,K}^v$ is still retained.
Every renormalized entropy solution $(v,\nu_v)$ satisfies Definition \ref{def:measure-free-supersolution}. 
Indeed, if $\eta\leq0$, then
\[
-\int_{D\times[0,T)}\eta(\psi)\phi\,\dd\nu_v\geq0.
\]
Omitting this nonnegative term from the left-hand side of \eqref{eq:entropy formula-1} gives \eqref{eq:measure-free-supersolution}. 

\begin{cor}[Minimality in the measure-free supersolution class]\label{cor:minimality-finite-reaction}Assume the hypotheses of Theorem \ref{thm:main-comparison}.
Let $(u,\nu)$ be a renormalized entropy solution with data $(f,\xi)$, and let $v$ be a measure-free entropy supersolution with data $(g,\tilde{\xi})$ and the same obstacle.
Suppose that $\xi\geq\psi$ a.e. in $D$ and
\[
f(x,t,r)\leq g(x,t,r)\quad\text{for a.e. }(x,t)\in D_T\text{ and every }r\in\mathbb R.
\]
If $L_g$ is a Lipschitz constant for $g$ in \eqref{eq:f-beta-Lipschitz-uniqueness}, then, for almost every $t\in(0,T)$,
\begin{equation}
\int_D\bigl(\beta(u(t))-\beta(v(t))\bigr)^+\,\dd x\leq \mathrm e^{L_gt} \int_D\bigl(\beta(\xi)-\beta(\tilde{\xi})\bigr)^+\,\dd x. \label{eq:minimality-comparison}
\end{equation}
In particular, if the initial data agree, then $\beta(u)\leq\beta(v)$ a.e. in $D_T$.
If $\beta$ is strictly increasing, then $u\leq v$ a.e. in $D_T$.
Thus the renormalized entropy solution under Definition \ref{def:entropy solution with ob}, which has the finite reaction, is the smallest measure-free entropy supersolution with the same data in the $\beta$-order; when $\beta$ is strictly increasing, it is also the smallest in the pointwise order.
\end{cor}

\begin{proof}
We only show the changes in the proof of Theorem \ref{thm:main-comparison}.
Put
\[
\mathfrak e_l^*:=\mu_{l,1}^u(D\times[0,T))+\mu_{l,1}^v(D\times[0,T)).
\]
In the interior doubling argument of Lemma \ref{lem:interior-local-Kato}, we place $u$ in the first position and $v$ in the second.
Use
\[
r\mapsto q_\delta(r-V)\quad\text{for }u,\qquad r\mapsto-q_\delta(U-r) \quad\text{for }v.
\]
The second entropy belongs to $\mathcal E$ with $-1\leq-q_\delta\leq0$.
The calculation in Lemma \ref{lem:interior-local-Kato} therefore gives
\begin{align}
&-\int_{D_T}\mathcal B_l(u,v)\chi\varphi'\,\dd x\,\dd t \nonumber\\
&\quad\leq\int_{D_T}\one_{\{u>v\}}\bigl[h_l(v)A(\nabla v)-h_l(u)A(\nabla u)\bigr]\cdot\nabla\chi\,\varphi\,\dd x\,\dd t\nonumber\\
&\qquad+\int_{D_T}\one_{\{u>v\}}\bigl[f(x,t,u)h_l(u)-g(x,t,v)h_l(v)\bigr]\chi\varphi\,\dd x\,\dd t+C\mathfrak e_l^*. \label{eq:minimality-local-Kato}
\end{align}
There is no reaction term in the entropy inequality of $v$.
The term associated with $\nu$ is bounded, as before, by
\[
C\frac{\omega_\psi(C\varsigma)}{\delta}\nu(D\times[0,T))
\]
and vanishes when $\varsigma\downarrow0$ for fixed $\delta$.

The boundary localization in the proof of Lemma \ref{lem:chartwise-signed-Kato} uses only the negative entropy inequality for $v$.
For a smooth function $\zeta\in C^1_c(\overline{D}\times(0,T))$, define
\begin{align*}
\mathscr L_{l,v}^{-,0}(\zeta)&:=-\int_{D_T}\partial_t\zeta \int_{\tilde{\xi}}^v\one_{\{r<0\}}h_l(r)\,\dd\beta(r)\,\dd x\,\dd t+\int_{D_T}\one_{\{v<0\}}h_l(v)A(\nabla v)\cdot\nabla\zeta\,\dd x\,\dd t \\
&\quad-\int_{D_T}g(x,t,v)\one_{\{v<0\}}h_l(v)\zeta\,\dd x\,\dd t +\int_{D_T}\zeta\,\dd\mu_{l,1}^v.
\end{align*}
Approximating $-\one_{(-\infty,0)}$ by nonpositive entropies $-q_\delta$ in \eqref{eq:measure-free-supersolution} shows that $\mathscr L_{l,v}^{-,0}(\zeta)\geq0$ for $\zeta\geq0$.
Hence it is a nonnegative Radon measure. 
The proof of Lemma \ref{lem:L-boundary-vanishing} applies without change.
In the chartwise argument, $\mathscr L_{l,v}^{-}$ is replaced by $\mathscr L_{l,v}^{-,0}$, while $\mathscr L_{l,u}^{+}$ is unchanged.
All terms that contained a reaction for $v$ are absent.
Boundary localization and summation of the partition of unity now give
\begin{align}
&-\int_{D_T}\mathcal B_l(u,v)\varphi'(t)\,\dd x\,\dd t \nonumber\\
&\quad\leq\int_{D_T}\one_{\{u>v\}} \bigl[f(x,t,u)h_l(u)-g(x,t,v)h_l(v)\bigr]\varphi(t) \,\dd x\,\dd t+C\mathfrak e_l^*. \label{eq:minimality-global-Kato}
\end{align}

We now deal with the initial value of $v$.
Note that $H_l(r)=\int_0^rh_l(q)\,\dd\beta(q)$.
Use the entropy $-q_\delta(\tilde{\xi}(y)-r)$ in \eqref{eq:measure-free-supersolution}, test with the inward kernel as in the proof of Theorem \ref{thm:averaged-initial-trace}, and use the cutoff $(1-t/\tau)^+$.
The negative half of the proof of Theorem \ref{thm:averaged-initial-trace} gives
\[
\limsup_{\tau\downarrow0}\frac1\tau\int_0^\tau\int_D\bigl(H_l(\tilde{\xi})-H_l(v(t))\bigr)^+\,\dd x\,\dd t\leq C\mu_{l,1}^v(D\times[0,T)).
\]
Here no obstacle term occurs.
Since $v,\tilde{\xi}\geq-M$,
\[
\bigl(\beta(\tilde{\xi})-\beta(v)\bigr)^+\leq\bigl(H_l(\tilde{\xi})-H_l(v)\bigr)^++\bigl(\beta(\tilde{\xi})-\beta(l)\bigr)^+.
\]
Letting $l\to\infty$ yields the one-sided averaged trace
\begin{equation}
\lim_{\tau\downarrow0}\frac1\tau\int_0^\tau\int_D\bigl(\beta(\tilde{\xi})-\beta(v(t))\bigr)^+\,\dd x\,\dd t=0. \label{eq:measure-free-initial-trace}
\end{equation}

Let $l\to\infty$ in \eqref{eq:minimality-global-Kato}.
As in the proof of Theorem \ref{thm:main-comparison}, we obtain
\[
-\int_0^TY(t)\varphi'(t)\,\dd t\leq L_g\int_0^TY(t)\varphi(t)\,\dd t,\qquad Y(t):=\int_D\bigl(\beta(u(t))-\beta(v(t))\bigr)^+\,\dd x.
\]
For almost every $s>0$,
\begin{align*}
Y(s)&\leq\|\beta(u(s))-\beta(\xi)\|_{L^1(D)}+\int_D\bigl(\beta(\xi)-\beta(\tilde{\xi})\bigr)^+\,\dd x\\
&\quad+\int_D\bigl(\beta(\tilde{\xi})-\beta(v(s))\bigr)^+\,\dd x.
\end{align*}
Average this inequality over $s\in(0,\tau)$ and let $\tau\downarrow0$.
Theorem \ref{thm:averaged-initial-trace} and \eqref{eq:measure-free-initial-trace} give the initial value in \eqref{eq:minimality-comparison}.
Gr\"onwall's inequality completes the proof.
\end{proof}

\begin{rem}[Minimality within the measure-free class]
Minimality is often included in potential-theoretic definitions of obstacle solutions; see, e.g., \cite{korte2009obstacle,moring2024two}.
Let $\mathscr S(g,\tilde\xi)$ be a supersolution class for the same obstacle and data such that every member satisfies Definition \ref{def:measure-free-supersolution}.
Then Corollary \ref{cor:minimality-finite-reaction} gives
\[
\beta(u)\leq\beta(v)\qquad\text{a.e. in }D_T
\]
for every $v\in\mathscr S(g,\tilde\xi)$, where $(u,\nu)$ is the renormalized entropy solution, which has the finite reaction. 
If $\beta$ is strictly increasing, then $u\leq v$ a.e. in $D_T$.
In particular, the comparison applies to any variational supersolution for which the negative one-sided entropy inequality and the vanishing defect estimate in Definition \ref{def:measure-free-supersolution} have been established.
\end{rem}

\section{Existence for time-dependent obstacles}
\label{sec:existence-general}
This section treats a more general equation than the one used in Section \ref{sec:comparison-dirichlet} in two structural respects: the diffusion field may depend on the solution, and the obstacle may depend on time.
To obtain existence, however, we impose stronger coercivity, initial data, and obstacle assumptions than those used in Section \ref{sec:comparison-dirichlet}.
We consider
\begin{equation}
\partial_t\beta(u)=\operatorname{div}A(u,\nabla u)+f(x,t,u)+\nu \quad\text{in }D_T, \label{eq:general-obstacle-equation}
\end{equation}
with the initial and boundary conditions in \eqref{eq:ito} and with $u\geq\psi$ a.e. in $D_T$.
The uniqueness result of Section \ref{sec:comparison-dirichlet} is used only when $A$ is independent of its first variable and $\psi$ is independent of time.

We use the same two-parameter penalization scheme and compactness argument for both an interior localization and a localization that reaches the boundary.
The boundary-reaching case is allowed when the full obstacle is uniformly below the homogeneous Dirichlet value on the relevant boundary portion.
Thus, the same construction covers two localization regimes.

Throughout this section, we assume that $p\in[2,\infty)$.
Denote $p':={p}/{(p-1)}$.
\subsection{Assumptions, solution concept, and main existence theorem}
We first give the assumptions used in this section.
\begin{assumption}[Time nonlinearity]\label{assu:existence-beta}
There is a convex function $\mathcal J\in C^1(\mathbb R)$ such that $\mathcal J'=\beta$.
The function $\beta$ is continuous and strictly increasing, $\beta(0)=0$, and
\[
\lim_{r\to\pm\infty}\beta(r)=\pm\infty.
\]
\end{assumption}
We use the energy function
\begin{equation}
B(r):=r\beta(r)-\mathcal J(r)+\mathcal J(0)=\int_0^r s\,\dd\beta(s). \label{eq:existence-energy}
\end{equation}
The Stieltjes integral in \eqref{eq:existence-energy} is well defined because $\beta$ is continuous and nondecreasing.
It also gives
\begin{equation}
B(r)-B(s)\geq s\bigl(\beta(r)-\beta(s)\bigr), \qquad r,s\in\mathbb R.\label{eq:energy-convexity}
\end{equation}
Also, for every $a>0$ there is a constant $C_a$ such that
\begin{equation}
|\beta(r)|\leq a B(r)+C_a, \qquad r\in\mathbb R. \label{eq:beta-controlled-by-energy}
\end{equation}
Indeed, the quotient $B(r)/|\beta(r)|$ tends to infinity as $|r|\to\infty$.
This follows directly from
\[
B(r)=\int_0^r\bigl(\beta(r)-\beta(s)\bigr)\,\dd s.
\]

\begin{assumption}[Diffusion field: Alt--Luckhaus structure] \label{assu:existence-A} 
The map $A:\mathbb R\times\mathbb R^d\longrightarrow\mathbb R^d$ is continuous.
There are constants $a_0,a_1>0$ such that 
\begin{align}
\bigl(A(r,\zeta)-A(r,\widehat\zeta)\bigr)\cdot(\zeta-\widehat\zeta) &\geq a_0|\zeta-\widehat\zeta|^p, \label{eq:A-uniform-monotonicity}\\
|A(r,\zeta)| &\leq a_1\bigl(1+B(r)^{1/p'}+|\zeta|^{p-1}\bigr) \label{eq:A-growth} 
\end{align} 
for every $r\in\mathbb R$ and every $\zeta,\widehat\zeta\in\mathbb R^d$. 
\end{assumption} 
\begin{assumption}[Carrillo--Wittbold structure condition] \label{assu:existence-A-comparison} 
There are continuous functions 
\[ 
C_A:\mathbb R^2\longrightarrow[0,\infty), \qquad q_A,\widehat q_A:\mathbb R^2\longrightarrow\mathbb R^d
\] 
such that 
\begin{align} 
&\bigl(A(r,\zeta)-A(s,\widehat\zeta)\bigr)\cdot(\zeta-\widehat\zeta) +C_A(r,s) \bigl(1+|\zeta|^p+|\widehat\zeta|^p\bigr)|r-s| \nonumber\\ 
&\hspace{4cm} \geq q_A(r,s)\cdot\zeta+\widehat q_A(r,s)\cdot\widehat\zeta \label{eq:CW-structure} 
\end{align} 
for every $r,s\in\mathbb R$ and every $\zeta,\widehat\zeta\in\mathbb R^d$. 
\end{assumption} 
Assumption \ref{assu:existence-A} is an Alt--Luckhaus-type growth and coercivity condition together with uniform $p$-monotonicity in the gradient variable. 
Assumption \ref{assu:existence-A-comparison} is used only in the penalized comparison arguments in Subsection \ref{subsec:comparison penalized}.
The Alt--Luckhaus assumptions imply the coercivity estimate 
\begin{equation} 
A(r,\zeta)\cdot\zeta \geq c_{\mathrm{coer}}|\zeta|^p-C_{\mathrm{coer}}\bigl(1+B(r)\bigr) \label{eq:A-coercivity} 
\end{equation} 
for some constants $c_{\mathrm{coer}},C_{\mathrm{coer}}>0$. 

\begin{rem}[The range of $p$ in the existence section]\label{rem:restriction-uniform-p-monotonicity}
The Alt--Luckhaus existence framework \cite[Theorem 1.7]{alt1983quasilinear} is formulated for $1<p<\infty$.
In the present paper we retain the quantitative monotonicity condition
\[
[A(r,\zeta)-A(r,\widehat\zeta)]\cdot(\zeta-\widehat\zeta)\geq a_0|\zeta-\widehat\zeta|^p,
\]
which is also part of the quantitative structure considered in the Alt--Luckhaus setting.
The restriction $p\geq2$ records the range in which this particular coefficient class is nonempty.
Indeed, if $1<p<2$, fix $r\in\mathbb R$ and a unit vector $e$.
Applying the displayed inequality to $\zeta_i=(i/n)e$ and $\zeta_{i-1}=((i-1)/n)e$, and summing over $i=1,\ldots,n$, gives
\[
[A(r,e)-A(r,0)]\cdot e\geq a_0 n^{2-p},
\]
which is impossible for a finite-valued map as $n\to\infty$.
For $1<p<2$, the singular $p$-Laplace field instead satisfies a weighted monotonicity inequality.
Treating that structure would require a different strong-convergence argument.
We therefore state the present existence result for $p\geq2$.
\end{rem}

\begin{assumption}[Source and initial value] \label{assu:existence-data} 
The map $ f:D\times[0,T]\times\mathbb R\longrightarrow\mathbb R $ is Carath\'eodory. 
There are constants $c_f>0$ and a nonnegative function $f_0\in L^{p'}(D_T)$ such that 
\begin{equation} 
|f(x,t,r)| \leq f_0(x,t)+c_f B(r)^{1/p'} \label{eq:f-energy-growth} 
\end{equation} 
for almost every $(x,t)\in D\times[0,T]$ and every $r\in\mathbb R$. 
For the comparison arguments, define $\widehat f(x,t,z):=f(x,t,\beta^{-1}(z))$ for $(x,t,z)\in D\times[0,T]\times\mathbb{R}$. We assume that, for some $L_f>0$, 
\begin{equation} 
|\widehat f(x,t,z)-\widehat f(x,t,w)| \leq L_f|z-w| \label{eq:fhat-Lipschitz}
\end{equation} 
for almost every $(x,t)\in D\times[0,T]$ and all $z,w\in\mathbb R$. 
The initial value $\xi$ is measurable, 
\[ 
B(\xi)\in L^1(D), \qquad \xi\geq\psi(\cdot,0) \quad\text{a.e. in }D. 
\] 
\end{assumption} 
\begin{rem}[Different roles of the source assumptions] 
The Carath\'eodory property and \eqref{eq:f-energy-growth} are sufficient for weak existence. 
The Lipschitz condition \eqref{eq:fhat-Lipschitz} is used only in the comparison and monotonicity arguments for the penalized equations. 
\end{rem}

\begin{assumption}[Decomposed obstacle]\label{assu:existence-obstacle}
The full obstacle $\psi\in C(\overline D\times[0,T])$.
There exists $\psi_1\in L^p(0,T;W^{1,p}(D))\cap L^\infty(D_T)\cap C([0,T];L^1(D))$ such that
\[
\operatorname{Tr}\psi_1\leq0\quad\text{for a.e. }t\in(0,T),
\]
such that
\begin{equation}
r_{\psi_1}:=\partial_t\beta(\psi_1)-\operatorname{div}A(\psi_1,\nabla\psi_1)-f(x,t,\psi_1)\in L^\infty(D_T)\label{eq:obstacle-residual}
\end{equation}
in the sense of distributions.
We also assume that
\[
\psi_{1}(0)\leq\xi\quad\text{a.e. in }D.
\]
Set $\psi_2:=\psi-\psi_1$ and assume that $\psi_2\geq0$ a.e. in $D_T$.
There exists a compact set $K_2\subset\overline D$ such that
\begin{equation}
\psi_2=0\quad\text{a.e. on }(D\setminus K_2)\times(0,T).\label{eq:psi2-localization}
\end{equation}
Set
\[
M_{\mathrm{obs}}:=\max\left\{\|\psi\|_{L^\infty(\overline D\times[0,T])},\|\psi_1\|_{L^\infty(D_T)}\right\}.
\]
If $K_2\cap\partial D\neq\varnothing$, assume that there is $\sigma_\partial>0$ such that
\begin{equation}
\psi(x,t)\leq-\sigma_\partial\quad\text{for every }(x,t)\in(K_2\cap\partial D)\times[0,T].\label{eq:K2-boundary-gap}
\end{equation}
\end{assumption}
\begin{rem}[Role of the decomposition]
The component $\psi_1$ supplies the boundary trace and residual control used in the first penalization. 
The correction $\psi_2$ is localized in $K_2$.
If $K_2\Subset D$, then $\psi=\psi_1$ near the boundary.
If $K_2$ reaches the boundary, the strict inequality \eqref{eq:K2-boundary-gap} replaces this local agreement.
\end{rem}

The regularity and the boundedness of $\psi_1$ and the continuity of $\beta$ imply $\beta(\psi_1)\in C([0,T];L^1(D))$.
By \eqref{eq:obstacle-residual}, the growth of $A$, and
$\psi_1\in L^p(0,T;W^{1,p}(D))\cap L^\infty(D_T)$, we have
\[
 \partial_t\beta(\psi_1)
 \in L^{p'}(0,T;W^{-1,p'}(D))+L^{p'}(D_T).
\]
The integration-by-parts argument yields the initial-value identity with initial datum $\psi_1(0)$ (see \eqref{eq:AL-initial-value}). 
Thus $\psi_1$ is admissible as the lower barrier in Lemma \ref{lem:comparison-lower-barrier}.

The dependence of $A$ on $u$ produces one extra term in the renormalized identity.
We include it in the definition.

\begin{defn}[Renormalized entropy solution for $A=A(u,\nabla u)$]\label{def:general-renormalized-obstacle-solution}
A pair $(u,\nu)$ is a renormalized entropy solution of \eqref{eq:general-obstacle-equation} if the following properties hold.
\begin{itemize}
\item[(i)] $u\geq\psi$ a.e. in $D_T$, $B(u)\in L^\infty(0,T;L^1(D))$, and $\mathcal T_k(u)\in L^p(0,T;W_0^{1,p}(D))$ for every $k>0$.
\item[(ii)] $\nu$ is a nonnegative finite Radon measure on $D\times[0,T)$.
\item[(iii)] For every $K>0$, there exists a family $(\mu_{l,K}^u)_{l>M_{\mathrm{obs}}+1}\subset\mathcal M_+(D\times[0,T))$ such that for every $l>M_{\mathrm{obs}}+1$, $\eta\in\mathcal E$ with $\|\eta\|_{L^\infty(\mathbb R)}\leq K$, and every test function $\phi$ admissible for $\eta$ from Definition \ref{def:entropy solution with ob},
\begin{align}
&-\int_{D_T}\partial_t\phi\int_\xi^u\eta(r)h_l(r)\,\dd\beta(r)\,\dd x\,\dd t-\int_{D\times[0,T)}\eta(\psi)\phi\,\dd\nu\nonumber\\
&\leq-\int_{D_T}A(u,\nabla u)\cdot\nabla\bigl(\eta(u)\phi\bigr)h_l(u)\,\dd x\,\dd t+\int_{D_T}f(x,t,u)\eta(u)h_l(u)\phi\,\dd x\,\dd t \nonumber\\
&\qquad+\int_{D_T}\nabla\phi\cdot\int_0^u A(r,0)\eta(r)h_l'(r)\,\dd r\,\dd x\,\dd t+\int_{D\times[0,T)}\phi\,\dd\mu_{l,K}^u, \label{eq:general-renormalized-inequality}
\end{align}
and
\begin{equation}
\lim_{l\to\infty}\mu_{l,K}^u(D\times[0,T))=0. \label{eq:general-tail-condition}
\end{equation}
\end{itemize}
\end{defn}

This definition is an energy-strengthened extension of Definition \ref{def:entropy solution with ob} to diffusion fields depending on the solution.
If $A$ is independent of its first variable and $A(0)=0$, then $A(r,0)=0$, so the additional term in \eqref{eq:general-renormalized-inequality} vanishes.
Moreover, \eqref{eq:beta-controlled-by-energy} gives
\[
B(u)\in L^\infty(0,T;L^1(D))\quad\Longrightarrow\quad
\beta(u)\in L^\infty(0,T;L^1(D)).
\]
Hence, every solution in the sense of Definition \ref{def:general-renormalized-obstacle-solution} belongs to the class of Definition \ref{def:entropy solution with ob}.

\begin{thm}[Existence for a time-dependent obstacle]\label{thm:existence-general-time-obstacle}
Let $p\geq2$, and let Assumptions \ref{assu:existence-beta}--\ref{assu:existence-obstacle} hold.
Then there exists a renormalized entropy solution $(u,\nu)$ in the sense of Definition \ref{def:general-renormalized-obstacle-solution}.
Moreover,
\[
\nu=\nu_1\,\dd x\,\dd t+\nu_2\in\mathcal M_+(D\times[0,T)),
\]
where $\nu_1\in L^\infty(D_T)$ and $\nu_2\in\mathcal M_+(D\times[0,T))$, and
\begin{align*}
0\leq\nu_1&\leq r_{\psi_1}^+\quad\text{a.e. in }D_T,\\
\nu_1[\beta(u)-\beta(\psi_1)]&=0\quad\text{a.e. in }D_T,\\
\nu_1&=0\quad\text{a.e. on }\{\psi_1<\psi\},
\end{align*}
and $\nu_2$ is concentrated on $(K_2\cap D)\times[0,T)$.
\end{thm}
The theorem does not assert a separate support identity $\nu_2(\{u>\psi\})=0$ or a classical complementarity relation for the possibly singular
measure $\nu_2$. 
Its interaction with the obstacle is encoded by the term $\eta(\psi)\,\dd\nu_2$ in \eqref{eq:general-renormalized-inequality}. 

The proof uses two monotone penalization limits in \eqref{eq:two-parameter-penalty}.
For fixed $\varepsilon$, the limit $\delta\downarrow0$ enforces the lower barrier $\psi_1$ and produces the bounded reaction density $\nu_{1,\varepsilon}$.
The uniform estimate in Lemma \ref{lem:uniform-energy-general-obstacle} gives an $\varepsilon$-independent bound for the total mass of $P_\varepsilon(\cdot,\cdot,u_\varepsilon)\,\dd x\,\dd t$.
Thus, as $\varepsilon\downarrow0$, a subsequence converges weakly-* to a finite Radon measure.

This measure-valued limit prevents a direct time pairing with the limit solution in flux identification.
We instead use the comparison function $z_{k,\lambda}$, the one-sided time regularization of the truncated solution $\mathcal{T}_k(u)$.
The relative-energy identity controls the time term.
Then, on the support of the penalty term, we have
\[
z_{k,\lambda}\geq\psi-c_{\lambda}.
\]
Therefore the penalty contribution is bounded by $c_\lambda$ times the uniform penalty mass.
This error vanishes as $\lambda\downarrow0$ and uniform monotonicity gives strong convergence of the gradients.
The final subsection then passes the limit to the renormalized entropy inequality.

\subsection{Alt-Luckhaus weak solutions and time calculus} 
We record the Alt--Luckhaus weak formulation used for the penalized equations. 
Set 
\[ 
\mathbb V:=W^{1,p}_0(D), \qquad \mathbb V^*:=W^{-1,p'}(D). 
\] 
Consider the auxiliary initial-boundary value problem 
\begin{align} 
\partial_t\beta(v) &= \operatorname{div}A(v,\nabla v)+g(x,t,v) &&\text{in }D_T, \label{eq:AL-auxiliary-problem}\\ 
v&=0 &&\text{on }\partial D\times(0,T),\nonumber\\ 
v(0)&=\xi &&\text{in }D.\nonumber 
\end{align}
\begin{defn}[Alt--Luckhaus weak solution] \label{def:AL-weak-solution} 
A function $v$ is a weak solution of \eqref{eq:AL-auxiliary-problem} if the following properties hold. 
\begin{enumerate} 
\item 
\[ 
v\in L^p(0,T;\mathbb V), \qquad \beta(v)\in L^\infty(0,T;L^1(D)), \qquad \partial_t\beta(v)\in L^{p'}(0,T;\mathbb V^*). 
\] 
\item The initial value is attained in the Alt--Luckhaus sense: 
\begin{equation} 
\int_0^T \left\langle\partial_t\beta(v),\zeta\right\rangle_{\mathbb V^*,\mathbb V} \,\dd t +\int_{D_T} \bigl(\beta(v)-\beta(\xi)\bigr)\partial_t\zeta \,\dd x\,\dd t=0\label{eq:AL-initial-value}
\end{equation} 
for every
\[ 
\zeta\in L^p(0,T;\mathbb V) \cap W^{1,1}(0,T;L^\infty(D)), \qquad \zeta(T)=0. 
\] 
\item The functions 
\[ 
A(v,\nabla v)\in L^{p'}(D_T)^d, \qquad g(\cdot,\cdot,v)\in L^{p'}(D_T), 
\] 
and for every $\varphi\in L^p(0,T;\mathbb V)$
\begin{equation} 
\int_0^T \left\langle\partial_t\beta(v),\varphi\right\rangle_{\mathbb V^*,\mathbb V} \,\dd t+\int_{D_T}A(v,\nabla v)\cdot\nabla\varphi\,\dd x\,\dd t=\int_{D_T}g(x,t,v)\varphi\,\dd x\,\dd t. \label{eq:AL-weak-equation} 
\end{equation} 
\end{enumerate} 
\end{defn} 

\begin{lem}[Homogeneous Alt--Luckhaus chain rule]\label{lem:AL-time-chain-rule}
Let $v$ satisfy items {\rm(1)} and {\rm(2)} of Definition \ref{def:AL-weak-solution} with initial datum $\xi$, and assume that $B(\xi)\in L^1(D)$.
Then
\[
B(v)\in L^\infty(0,T;L^1(D)).
\]
Moreover, the function
\begin{equation}
\mathscr E_v(t):=\int_D B(\xi)\,\dd x+\int_0^t\left\langle\partial_\tau\beta(v),v\right\rangle_{\mathbb V^*,\mathbb V}\,\dd\tau\label{eq:AL-energy-representative-definition}
\end{equation}
belongs to the absolutely continuous function space $AC([0,T])$ and satisfies
\begin{equation}
\mathscr E_v(t)=\int_D B(v(t))\,\dd x\quad\text{for almost every }t\in(0,T).\label{eq:AL-energy-representative-identification}
\end{equation}
Consequently,
\begin{equation*}
\mathscr E_v(t)-\mathscr E_v(s)=\int_s^t\left\langle\partial_\tau\beta(v),v\right\rangle_{\mathbb V^*,\mathbb V}\,\dd\tau \quad (0\leq s\leq t\leq T). 
\end{equation*}
\end{lem}

\begin{proof}
This is the homogeneous-boundary case of \cite[Lemma 1.5, pp. 315--317]{alt1983quasilinear}.
In the notation of that paper, take $b=\beta$ and $u^D=0$. 
The terms involving $u^D$ and $\partial_t u^D$ therefore vanish. 
Therefore, formula \eqref{eq:AL-energy-representative-identification} follows directly from the cited lemma, and \eqref{eq:AL-energy-representative-definition} defines the corresponding absolutely continuous representative on $[0,T]$.
\end{proof}
\begin{thm}[Alt--Luckhaus weak existence] \label{thm:AL-weak-existence} 
Let Assumptions \ref{assu:existence-beta}--\ref{assu:existence-A} hold, and let $B(\xi)\in L^1(D)$. 
Suppose that $g:D\times[0,T]\times\mathbb R\rightarrow\mathbb R$ is Carath\'eodory and that there are $g_0\in L^{p'}(D_T)$, $g_0\geq0$, and $c_g>0$ such that 
\begin{equation} 
|g(x,t,r)| \leq g_0(x,t)+c_g\bigl(1+B(r)^{1/p'}\bigr) \label{eq:AL-source-growth} 
\end{equation} 
for almost every $(x,t)\in D\times[0,T]$ and every $r\in\mathbb R$. 
Then \eqref{eq:AL-auxiliary-problem} has a weak solution in the sense of Definition \ref{def:AL-weak-solution}. 
Moreover, 
\begin{equation*} 
B(v)\in L^\infty(0,T;L^1(D)) 
\end{equation*} 
and 
\begin{equation*} 
\operatorname*{ess\,sup}_{0<t<T} \int_D B(v(t))\,\dd x+\int_{D_T}|\nabla v|^p\,\dd x\,\dd t\leq C, 
\end{equation*} 
where $C$ depends only on the structural constants, $T$, $D$, $\Vert B(\xi)\Vert_{L^{1}(D)}$, and $\|g_0\|_{L^{p'}(D_T)}$. 
\end{thm} 
\begin{proof} 
In the notation of \cite[Section 1.1]{alt1983quasilinear}, set $z=\beta(v)$ and
\[
\widetilde A(z,\zeta):=A(\beta^{-1}(z),\zeta),\qquad\widetilde g(x,t,z):=g(x,t,\beta^{-1}(z)).
\]
Assumption \ref{assu:existence-beta} makes $\beta$ a homeomorphism of $\mathbb R$.
The function $B$ in \eqref{eq:existence-energy} is the energy associated with the time nonlinearity in the cited construction.
Conditions \eqref{eq:A-uniform-monotonicity}, \eqref{eq:A-growth}, and \eqref{eq:A-coercivity} give, respectively, the strict spatial monotonicity, growth, and coercivity estimates used in the elliptic time steps.
Condition \eqref{eq:AL-source-growth} gives the $L^{p'}$ lower-order bound.
For the $x,t$ dependence, use the piecewise constant time averages from \cite[Section 1.10]{alt1983quasilinear}; Lemma \ref{lem:AL-space-time-source} verifies their growth, local-uniform convergence, and weak composition convergence.
Finally, $B(\xi)\in L^1(D)$ is exactly the initial-energy hypothesis.
Thus the scalar homogeneous-Dirichlet specialization of \cite[Theorem 1.7]{alt1983quasilinear} yields the existence of weak solution and estimate.
\end{proof}

\begin{lem}[Evaluation against a fixed spatial function]\label{lem:AL-fixed-spatial-test}
Let $v$ satisfy items {\rm(1)} and {\rm(2)} of Definition \ref{def:AL-weak-solution}.
For every $H\in\mathbb V\cap L^\infty(D)$, the scalar function
\[
y_H(t):=\int_D[\beta(v(t))-\beta(\xi)]H\,\dd x
\]
has an absolutely continuous representative satisfying
\[
y_H(0)=0,\qquad y_H'(t)=\langle\partial_t\beta(v(t)),H\rangle_{\mathbb V^*,\mathbb V}\quad\text{for a.e. }t.
\]
Consequently, for every $t\in[0,T]$ in terms of this
representative,
\begin{equation}
\int_0^t\langle\partial_s\beta(v),H\rangle_{\mathbb V^*,\mathbb V}\,\dd s=\int_D[\beta(v(t))-\beta(\xi)]H\,\dd x.\label{eq:AL-fixed-spatial-test}
\end{equation}
\end{lem}

\begin{proof}
For $\eta\in W^{1,1}(0,T)$ with $\eta(T)=0$, the function $\zeta(x,t):=\eta(t)H(x)$ is admissible in \eqref{eq:AL-initial-value}.
Hence
\[
\int_0^T\eta(t)\langle\partial_t\beta(v(t)),H\rangle_{\mathbb V^*,\mathbb V}\,\dd t+\int_0^T\eta'(t)y_H(t)\,\dd t=0.
\]
Taking first $\eta\in C_c^\infty(0,T)$ shows in distributions that $y_H'=\langle\partial_t\beta(v),H\rangle_{\mathbb V^*,\mathbb V}\in L^{p'}(0,T)$.
Thus $y_H$ has a representative in $W^{1,p'}(0,T)\subset AC([0,T])$.
Integrating by parts for a general admissible $\eta$ gives
\[
0=-\eta(0)y_H(0).
\]
Since $\eta(0)$ is arbitrary, we have $y_H(0)=0$.
The fundamental theorem of calculus for absolutely continuous functions now gives \eqref{eq:AL-fixed-spatial-test}.
\end{proof}
For $a,b\in\mathbb R$, define the relative energy
\begin{equation}\label{eq:def-relative-energy}
\mathcal B_\beta(a\mid b):=B(a)-B(b)-b[\beta(a)-\beta(b)]=\int_b^a(r-b)\,\dd\beta(r)\geq0.
\end{equation}

\begin{lem}[Relative-energy chain rule]\label{lem:relative-energy-chain-rule}
Let $v$ satisfy items {\rm(1)} and {\rm(2)} of Definition \ref{def:AL-weak-solution} with initial datum $v_0$, and assume that $B(v_0)\in L^1(D)$.
Let
\[
z\in L^p(0,T;\mathbb V)\cap W^{1,1}(0,T;L^\infty(D)),\qquad z_0:=z(0).
\]
Then
\[
t\longmapsto\int_D\mathcal B_\beta(v(t)\mid z(t))\,\dd x
\]
has a nonnegative representative in $AC([0,T])$.
We denote this representative by the same expression.
For every $t\in[0,T]$,
\begin{align}
&\int_D\mathcal B_\beta(v(t)\mid z(t))\,\dd x-\int_D\mathcal B_\beta(v_0\mid z_0)\,\dd x\nonumber\\
&=\int_0^t\langle\partial_\tau\beta(v),v-z\rangle_{\mathbb V^*,\mathbb V}\,\dd\tau-\int_0^t\!\int_D[\beta(v)-\beta(z)]\partial_\tau z\,\dd x\,\dd\tau.\label{eq:relative-energy-chain-rule}
\end{align}
\end{lem}
\begin{proof}
Lemma \ref{lem:AL-time-chain-rule} gives an absolutely continuous representative of the homogeneous energy and for every $t\in(0,T)$,
\begin{equation}
\int_D B(v(t))\,\dd x=\int_D B(v_0)\,\dd x+\int_0^t\langle\partial_\tau\beta(v),v\rangle_{\mathbb V^*,\mathbb V}\,\dd\tau.\label{eq:relative-proof-v-energy}
\end{equation}
It also gives $B(v)\in L^\infty(0,T;L^1(D))$.
Hence \eqref{eq:beta-controlled-by-energy} yields
\[
\beta(v)\in L^\infty(0,T;L^1(D)),\qquad\beta(v_0)\in L^1(D).
\]
We next derive the product rule for $\int_D\beta(v(t))z(t)\,\dd x$.
Let $\vartheta\in C^1([0,T])$ satisfy $\vartheta(T)=0$ and use the test function $\zeta(x,t)=\vartheta(t)z(x,t)$ in the Alt--Luckhaus initial-value identity \eqref{eq:AL-initial-value}.
Then we have
\begin{align}
0&=\int_0^T\vartheta\langle\partial_t\beta(v),z\rangle_{\mathbb V^*,\mathbb V}\,\dd t+\int_0^T\vartheta'(t)\int_D[\beta(v)-\beta(v_0)]z\,\dd x\,\dd t\nonumber\\
&\quad+\int_0^T\vartheta(t)\int_D[\beta(v)-\beta(v_0)]\partial_tz\,\dd x\,\dd t.\label{eq:relative-proof-cross-test}
\end{align}
Taking first $\vartheta\in C_c^\infty(0,T)$ shows that
\[
t\longmapsto\int_D[\beta(v(t))-\beta(v_0)]z(t)\,\dd x
\]
has an absolutely continuous representative whose derivative is
\[
\langle\partial_t\beta(v),z\rangle_{\mathbb V^*,\mathbb V}+\int_D[\beta(v)-\beta(v_0)]\partial_tz\,\dd x.
\]
Using a general $\vartheta$ in \eqref{eq:relative-proof-cross-test} shows that this representative vanishes at $t=0$.
Since
\[
t\longmapsto\int_D\beta(v_0)z(t)\,\dd x
\]
is absolutely continuous, we obtain
\begin{align}
&\int_D\beta(v(t))z(t)\,\dd x-\int_D\beta(v_0)z_0\,\dd x=\int_0^t\langle\partial_\tau\beta(v),z\rangle_{\mathbb V^*,\mathbb V}\,\dd\tau+\int_0^t\!\int_D\beta(v)\partial_\tau z\,\dd x\,\dd\tau.\label{eq:relative-proof-cross-product}
\end{align}

It remains to compute the part depending only on $z$.
Note that the range of $z$ is bounded because $z\in W^{1,1}(0,T;L^\infty(D))$.
Since $\mathcal J\in C^1(\mathbb R)$ and $\mathcal J'=\beta$, using the definition of $B$, the scalar chain rule gives, for almost every $(x,t)$,
\[
\partial_t\!\left[z\beta(z)-B(z)\right]=\partial_t\!\left[\mathcal J(z)-\mathcal J(0)\right]=\beta(z)\partial_tz.
\]
After integration in space and time,
\begin{align}
&\int_D[z(t)\beta(z(t))-B(z(t))]\,\dd x-\int_D[z_0\beta(z_0)-B(z_0)]\,\dd x=\int_0^t\!\int_D\beta(z)\partial_\tau z\,\dd x\,\dd\tau.\label{eq:relative-proof-z-energy}
\end{align}

Adding \eqref{eq:relative-proof-v-energy} and \eqref{eq:relative-proof-z-energy}, and subtracting \eqref{eq:relative-proof-cross-product}, using the definition of $\mathcal B_\beta$, we have \eqref{eq:relative-energy-chain-rule}.
The preceding three identities provide absolutely continuous representatives of
\[
\int_D B(v(t))\,\dd x,\qquad\int_D\beta(v(t))z(t)\,\dd x,\qquad\int_D[z(t)\beta(z(t))-B(z(t))]\,\dd x.
\]
Their combination is therefore an absolutely continuous representative of
\[
t\longmapsto\int_D\mathcal B_\beta(v(t)\mid z(t))\,\dd x.
\]
Moreover, note that for almost every $t$, this integral is nonnegative.
If its continuous representative were negative at some time, it would be negative on an interval of positive measure, a contradiction.
Thus the representative is nonnegative on $[0,T]$.
\end{proof}

\subsection{Comparison and the two-parameter penalized problem}\label{subsec:comparison penalized}
We record the comparison result for two equations with ordered sources and zero boundary values.
\begin{lem}[Penalized comparison]\label{lem:penalized-comparison-general-A}
Assume Assumptions \ref{assu:existence-beta}--\ref{assu:existence-A-comparison}.
For $i=1,2$, let $v_i$ be a weak solution
\begin{equation*}
\partial_t\beta(v_i)=\operatorname{div}A(v_i,\nabla v_i)+g_i(x,t,\beta(v_i)) 
\end{equation*}
with zero boundary values and initial datum $\xi_i$.
Suppose that
\begin{align*}
g_1(x,t,z)&\leq g_2(x,t,z),\\
|g_2(x,t,z)-g_2(x,t,w)|&\leq L|z-w|
\end{align*}
for almost every $(x,t)\in D\times[0,T]$ and all $z,w\in\mathbb R$.
Then, we have
\begin{align}
&\int_D\bigl(\beta(v_1(t))-\beta(v_2(t))\bigr)^+\,\dd x\leq \mathrm e^{Lt} \int_D\bigl(\beta(\xi_1)-\beta(\xi_2)\bigr)^+\,\dd x \label{eq:penalized-comparison-estimate}
\end{align}
for almost every $t\in(0,T)$.
Moreover, if $\beta(\xi_1)\leq\beta(\xi_2)$, then
\[
\beta(v_1)\leq\beta(v_2)\quad\text{a.e. in }D_T.
\]
\end{lem}

\begin{proof}
The chain-rule argument in \cite[Proposition 1.3]{carrillo1999uniqueness} gives the corresponding renormalized inequalities for the weak solution.
The Kato inequality in \cite[Theorem 2.3]{carrillo1999uniqueness}, whose diffusion hypothesis is \eqref{eq:CW-structure}, gives
\begin{align*}
Y(t)&\leq Y(0)+\int_0^t\!\int_{\{v_1>v_2\}}\bigl[g_1(x,s,\beta(v_1))-g_2(x,s,\beta(v_2))\bigr]\,\dd x\,\dd s,
\end{align*}
where
\[
Y(t):=\int_D\bigl(\beta(v_1(t))-\beta(v_2(t))\bigr)^+\,\dd x.
\]
On $\{v_1>v_2\}$, strict monotonicity of $\beta$ gives
\begin{align*}
&g_1(x,t,\beta(v_1))-g_2(x,t,\beta(v_2))\\
&\quad\leq g_2(x,t,\beta(v_1))-g_2(x,t,\beta(v_2))\leq L\bigl(\beta(v_1)-\beta(v_2)\bigr)^+.
\end{align*}
Gr\"onwall's inequality proves \eqref{eq:penalized-comparison-estimate}.
\end{proof}
Fix $K_P>0$.
Let $\kappa\in C^\infty(\mathbb R;[0,1])$ be nonincreasing, with $\kappa=1$ on $(-\infty,0]$ and $\kappa=0$ on $[1,\infty)$.
For $\delta>0$, set
\[
\kappa_\delta(z):=\kappa(z/\delta),\qquad\text{for }z\in\mathbb{R}.
\]
For $\varepsilon>0$ and $(x,t,r)\in D\times[0,T]\times\mathbb{R}$, define
\begin{align*}
P_\varepsilon(x,t,r)&:=\frac{1}{\varepsilon} \min\bigl\{(\beta(\psi(x,t))-\beta(r))^+,K_P\bigr\}, \\
\alpha_\varepsilon(x,t) &:=\bigl(r_{\psi_1}(x,t)-P_\varepsilon(x,t,\psi_1(x,t))\bigr)^+. 
\end{align*}
Then for a.e. $(x,t)\in D\times[0,T]$ and every $r\in\mathbb R$,
\begin{equation}
0\leq P_\varepsilon(x,t,r)\leq\frac{K_P}{\varepsilon},\qquad 0\leq\alpha_\varepsilon(x,t)\leq r_{\psi_1}^+(x,t).\label{eq:alpha-uniform-bound}
\end{equation}
For $\delta,\varepsilon\in(0,1)$, consider
\begin{align}
\partial_t\beta(u_{\delta,\varepsilon})&=\operatorname{div}A(u_{\delta,\varepsilon},\nabla u_{\delta,\varepsilon})+f(x,t,u_{\delta,\varepsilon})\nonumber\\
&\quad+\alpha_\varepsilon(x,t)\kappa_\delta\bigl(\beta(u_{\delta,\varepsilon})-\beta(\psi_1)\bigr)+P_\varepsilon(x,t,u_{\delta,\varepsilon})&&\text{in }D_T,\label{eq:two-parameter-penalty}\\
u_{\delta,\varepsilon}(0)&=\xi&&\text{in }D,\nonumber\\
\qquad u_{\delta,\varepsilon}&=0\quad&&\text{on }\partial D\times(0,T). \nonumber
\end{align}

For fixed $\delta,\varepsilon\in(0,1)$, define 
\begin{align*} 
g_{\delta,\varepsilon}(x,t,r) &:= f(x,t,r) +\alpha_\varepsilon(x,t) \kappa_\delta\bigl(\beta(r)-\beta(\psi_1(x,t))\bigr)+P_\varepsilon(x,t,r). 
\end{align*} 
This is a Carath\'eodory function. 
Note that $0\leq\kappa_\delta\leq1$. 
Then  
\begin{equation*} 
|g_{\delta,\varepsilon}(x,t,r)| \leq f_0(x,t)  +\|r_{\psi_1}^+\|_{L^\infty(D_T)}+\frac{K_P}{\varepsilon} +c_f B(r)^{1/p'}. 
\end{equation*} 
For fixed $\delta$ and $\varepsilon$, the right-hand side has the form required in \eqref{eq:AL-source-growth}. 
Theorem \ref{thm:AL-weak-existence} therefore gives a weak solution $u_{\delta,\varepsilon}$ of \eqref{eq:two-parameter-penalty} in the sense of Definition \ref{def:AL-weak-solution}. 
The representation of the source as a function of $\beta(r)$ and the Lipschitz properties of $\widehat{f}$,  $\kappa_\delta$, and $P_\varepsilon$ allow Lemma \ref{lem:penalized-comparison-general-A} to be applied.
Hence the weak solution is unique, and $u_{\delta,\varepsilon}$ is well defined.

The next lemma extends the comparison result to a lower barrier with nonpositive boundary trace.

\begin{lem}[Comparison with the lower barrier]\label{lem:comparison-lower-barrier}
Assume Assumptions \ref{assu:existence-beta}--\ref{assu:existence-data}.
Let 
\[
z\in L^p(0,T;W^{1,p}(D))\cap L^\infty(D_T)
\] 
satisfy $\operatorname{Tr}z\leq0$ for a.e. $t\in(0,T)$, and
\begin{equation}
\partial_t\beta(z) = \operatorname{div}A(z,\nabla z) +f(x,t,z)+r_z \quad\text{in }\mathcal D'(D_T), \label{eq:barrier-residual-equation}
\end{equation}
where $r_z\in L^\infty(D_T)$.
Assume that $z$ has initial datum $\xi_z$ in the Alt--Luckhaus sense in Definition \ref{def:AL-weak-solution}.
For $\delta,\varepsilon>0$, let $v$ be an Alt--Luckhaus weak solution with zero boundary values and initial datum $\xi_v$ of
\begin{align*}
\partial_t\beta(v) &= \operatorname{div}A(v,\nabla v) +f(x,t,v)+ \alpha_\varepsilon\kappa_\delta \bigl(\beta(v)-\beta(z)\bigr)+P_\varepsilon(x,t,v),
\end{align*}
where $P_\varepsilon$ is the penalty defined above and
\[
\alpha_\varepsilon := \bigl(r_z-P_\varepsilon(\cdot,\cdot,z)\bigr)^+.
\]
If $\xi_z\leq\xi_v$ a.e. in $D$, then $z\leq v$ a.e. in $D_T$.
\end{lem}

\begin{proof}
The proof uses the interior Kato inequality and the sign boundary splitting from Lemmas \ref{lem:interior-local-Kato} and \ref{lem:chartwise-signed-Kato}.
The only additional point is that $z$ has a nonpositive boundary trace, while the estimate for diffusion $A(v,\nabla v)$ can be done as in \cite[proof of Proposition 3.2]{carrillo1999uniqueness}. 
First, the growth assumptions and the boundedness of $z$ imply
\[
\partial_t\beta(z)\in L^{p'}(0,T;W^{-1,p'}(D)).
\]
Thus the integration-by-parts argument of \cite[Proposition 1.3]{carrillo1999uniqueness} applies to $z$ and $v$ in the interior.

For $\gamma>0$, let $q_\gamma$ be the one-sided cutoff fixed in \eqref{eq:property of q1}--\eqref{eq:property of q2}, with the parameter there renamed from $\delta$ to $\gamma$.
It is a $C^1$ version of the function $H_\varepsilon$ introduced in \cite[Lemma 3.1]{carrillo1999uniqueness} and used in the proof of \cite[Proposition 3.2]{carrillo1999uniqueness}.
It approximates $\operatorname{sign}_0^+$ and localizes the comparison to $(0,\gamma)$.
For $l>0$, use the cutoff $h_l$ fixed in \eqref{eq:admissible-cutoff-family}.

Use the time mollifier and the inward spatial kernels from Lemma \ref{lem:inward-kernel-properties}.
In the cited Kato argument, take $z$ as the first function, with source $f(x,t,z)+r_z$, and take $v$ as the second function, with source
\[
f(x,t,v)+\alpha_\varepsilon\kappa_\delta(\beta(v)-\beta(z))+P_\varepsilon(x,t,v).
\]

For the positive branch, use the function
\[
h_l(r)q_\gamma(r-v^+)
\]
in the equation for $z$.
Since $\operatorname{Tr}z\leq0$, we have $\operatorname{Tr}z^+=0$.
Hence this test has zero trace.
The companion test in the equation for $v$ is generated by $h_l(r)q_\gamma(z^+-r^+)$.

For the negative branch, interchange the two functions as in the negative-part calculation in the proof of Lemma \ref{lem:chartwise-signed-Kato}.
The test in the equation for $v$ is generated by
\[
 h_l(r)q_\gamma(r^--z^-).
\]
It has zero trace because $v$ has homogeneous boundary values.
The companion test in the equation for $z$ is compactly supported in the other spatial variable.
Thus no zero-trace assumption on $z^-$ is needed.

Adding the two branches uses
\[
(z-v)^+=(z^+-v^+)^++(v^--z^-)^+.
\]
With $l$ fixed, we follow the limit order from \cite[Proposition 3.2]{carrillo1999uniqueness}: first let the sign-cutoff parameter $\gamma\downarrow0$, then the time mollifier parameter $\theta\downarrow0$, and finally the inward spatial mollifier parameter $\varsigma\downarrow0$. 
We then let $l\to\infty$. 
Note that there is no reflection measure in the inequality. 
Then, using the estimates in \cite[proof of Theorem 2.3]{carrillo1999uniqueness}, we have for almost every $\tau\in(0,T)$,
\begin{align}
Y(\tau)&\leq Y(0)+\int_0^\tau\!\int_{\{z>v\}}[f(x,t,z)-f(x,t,v)]\,\dd x\,\dd t\nonumber\\
&\quad+\int_0^\tau\!\int_{\{z>v\}}\left[r_z-\alpha_\varepsilon\kappa_\delta(\beta(v)-\beta(z))-P_\varepsilon(x,t,v)\right]\,\dd x\,\dd t,\label{eq:lower-barrier-Kato}
\end{align}
where
\[
Y(t):=\int_D[\beta(z(t))-\beta(v(t))]^+\,\dd x.
\]
The initial term is $Y(0)$ by the initial-trace identities for $z$ and $v$.

On $\{z>v\}$, strict monotonicity of $\beta$ gives $\kappa_\delta(\beta(v)-\beta(z))=1$.
Since $P_\varepsilon(x,t,\cdot)$ is nonincreasing,
\begin{align*}
r_z-\alpha_\varepsilon-P_\varepsilon(x,t,v)&\leq r_z-(r_z-P_\varepsilon(x,t,z))^+-P_\varepsilon(x,t,z)\leq0.
\end{align*}
Moreover, \eqref{eq:fhat-Lipschitz} gives
\[
\one_{\{z>v\}}[f(x,t,z)-f(x,t,v)]\leq L_f[\beta(z)-\beta(v)]^+.
\]
Therefore
\[
Y(\tau)\leq Y(0)+L_f\int_0^\tau Y(t)\,\dd t.
\]
Since $\xi_z\leq\xi_v$, the initial-trace identities give $Y(0)=0$.
Gronwall's inequality yields $Y=0$, and hence $z\leq v$ a.e. in $D_T$.
\end{proof}
\begin{rem}[Order of the Kato limits]
At the penalized level, the terms $\alpha_\varepsilon\kappa_\delta$ and $P_\varepsilon$ are nonlinear source terms.  
The passage above therefore contains no pairing with a
reaction measure, and the sign--time--space order in \cite{carrillo1999uniqueness} applies unchanged.
\end{rem}

\begin{lem}[Localization of the $P_\varepsilon$-penalty]\label{lem:principal-penalty-localization}
Let $u_{\delta,\varepsilon}$ be a solution of \eqref{eq:two-parameter-penalty}.
Then we have
\[
u_{\delta,\varepsilon}\geq\psi_1\quad\text{a.e. in }D_T.
\]
Moreover, we have
\begin{equation}
\{P_\varepsilon(\cdot,\cdot,u_{\delta,\varepsilon})>0\} \subset\{\psi_2>0\}\subset K_2\times[0,T]\quad\text{up to a null set}.\label{eq:principal-penalty-localization}
\end{equation}
\end{lem}

\begin{proof}
Since $\psi_1$ satisfies the lower-barrier equation with residual $r_{\psi_1}$, Lemma \ref{lem:comparison-lower-barrier} gives
\[
u_{\delta,\varepsilon}\geq\psi_1\quad\text{a.e. in }D_T.
\]
If $P_\varepsilon(x,t,u_{\delta,\varepsilon}(x,t))>0$, then the definition of $P_\varepsilon$ gives
\[
u_{\delta,\varepsilon}(x,t)<\psi(x,t).
\]
Since we also have $u_{\delta,\varepsilon}(x,t)\geq\psi_1(x,t)$, we obtain
\[
\psi_2(x,t)=\psi(x,t)-\psi_1(x,t)>0.
\]
Therefore, we have $(x,t)\in K_2\times[0,T]$, and \eqref{eq:principal-penalty-localization} follows.
\end{proof}
\begin{lem}[A zero-boundary majorant]\label{lem:penalty-footprint-majorant}
Under Assumption \ref{assu:existence-obstacle}, there exists
\[
H_*\in W_0^{1,p}(D)\cap L^\infty(D)\quad\text{and}\quad\sigma_*>0
\]
such that
\begin{equation}
H_*(x)\geq\psi(x,t)+\sigma_*\quad\text{for every }(x,t)\in K_2\times[0,T].\label{eq:K2-majorant}
\end{equation}
\end{lem}

\begin{proof}
Since $D$ is Lipschitz, let $d_{\partial D}(x):=\operatorname{dist}(x,\partial D)\in W^{1,\infty}(D)\cap W_0^{1,p}(D)$.
If $K_2=\varnothing$, then $\psi_2\equiv0$ and the  penalty term involving $P_\varepsilon$ vanishes because $u_{\delta,\varepsilon}\geq\psi_1=\psi$.
We can take $H_*\equiv0$ and $\sigma_*=1$.
Hence we only consider the case $K_2\neq\varnothing$.

If $K_2\cap\partial D=\varnothing$, with the compactness of $K_2$, set   
\[
\sigma_*:=1,\qquad d_0:=\operatorname{dist}(K_2,\partial D)>0.
\]
Then,
\[
H_*(x):=(M_{\mathrm{obs}}+\sigma_*)\frac{d_{\partial D}(x)}{d_0}
\]
belongs to $W_0^{1,p}(D)\cap L^\infty(D)$ and satisfies $H_*\geq M_{\mathrm{obs}}+\sigma_*\geq\psi+\sigma_*$ on $K_2\times[0,T]$.

Suppose now that $K_2\cap\partial D\neq\varnothing$.
Choose
\[
\sigma_*=\min\{1,\frac{\sigma_\partial}{4}\}.
\]
We claim that there exists $\rho>0$ such that
\begin{equation}
\psi(x,t)\leq-2\sigma_*\quad\text{whenever }x\in K_2, \ d_{\partial D}(x)<\rho,\ t\in[0,T].\label{eq:K2-boundary-strip-gap}
\end{equation}
Otherwise there would be sequences $x_n\in K_2$ and $t_n\in[0,T]$ with $d_{\partial D}(x_n)\to0$ and $\psi(x_n,t_n)>-2\sigma_*$.
By compactness, after extraction, $x_n\to x_*\in K_2\cap\partial D$ and $t_n\to t_*$.
Continuity gives
\[
\psi(x_*,t_*)\geq-2\sigma_*>-\sigma_\partial,
\]
which is contrary to \eqref{eq:K2-boundary-gap}.
Thus \eqref{eq:K2-boundary-strip-gap} holds. 
Define
\[
H_*(x):=(M_{\mathrm{obs}}+\sigma_*)\frac{d_{\partial D}(x)}{\rho}.
\]
If $x\in K_2$ and $d_{\partial D}(x)<\rho$, then $H_*(x)\geq0\geq\psi(x,t)+\sigma_*$.
If $d_{\partial D}(x)\geq\rho$, then $H_*(x)\geq M_{\mathrm{obs}}+\sigma_*\geq\psi(x,t)+\sigma_*$.
Hence \eqref{eq:K2-majorant} holds on $K_2\times[0,T]$.
\end{proof}
\begin{rem}[Why the strict boundary gap is needed]
Lemma \ref{lem:penalty-footprint-majorant} is the only point at which the strict boundary gap is used.
The function $H_*$ has zero trace and satisfies $H_*\geq\psi+\sigma_*$ on $K_2$.
Testing with $u_{\delta,\varepsilon}-H_*$ therefore yields the coercive contribution
\[
\sigma_*\int_{D_T}P_\varepsilon(x,t,u_{\delta,\varepsilon})\,\dd x\,\dd t,
\]
which gives the uniform mass bound for the localized penalty.
\end{rem}
The next lemma gives estimates independent of both parameters.

\begin{lem}[Uniform estimate]\label{lem:uniform-energy-general-obstacle}
There is a constant $C$, independent of $\delta,\varepsilon\in(0,1)$, such that
\begin{align}
&\operatorname*{ess\,sup}_{0<t<T}\int_D B(u_{\delta,\varepsilon}(t))\,\dd x +\int_{D_T}|\nabla u_{\delta,\varepsilon}|^p\,\dd x\,\dd t \nonumber\\
&\quad +\int_{D_T}P_\varepsilon(x,t,u_{\delta,\varepsilon})\,\dd x\,\dd t+\int_{D_T}P_\varepsilon(x,t,u_{\delta,\varepsilon})(\psi-u_{\delta,\varepsilon})^+\,\dd x\,\dd t\nonumber\\
&\leq C\left(1+\|B(\xi)\|_{L^1(D)}+\|r_{\psi_1}^+\|_{L^1(D_T)}+\|f_0\|_{L^{p'}(D_T)}^{p'}\right).\label{eq:uniform-energy-general-obstacle}
\end{align}
\end{lem}

\begin{proof}
Let $H_*$ and $\sigma_*$ be given by Lemma \ref{lem:penalty-footprint-majorant}.
Since $H_*\in W_0^{1,p}(D)\cap L^\infty(D)$, use $u_{\delta,\varepsilon}-H_*$ as a test function in \eqref{eq:two-parameter-penalty}. For almost every $t\in(0,T)$, we have 
\begin{align*}
\int_0^t\langle\partial_s\beta(u_{\delta,\varepsilon}),u_{\delta,\varepsilon}-H_*\rangle_{\mathbb V^*,\mathbb V}\,\dd s&=-\int_0^t\!\int_DA(u_{\delta,\varepsilon},\nabla u_{\delta,\varepsilon})\cdot\nabla(u_{\delta,\varepsilon}-H_*)\,\dd x\,\dd s\\
&\quad+\int_0^t\!\int_Df(x,s,u_{\delta,\varepsilon})(u_{\delta,\varepsilon}-H_*)\,\dd x\,\dd s\\
&\quad+\int_0^t\!\int_D\alpha_\varepsilon(x,s)\kappa_\delta(\beta(u_{\delta,\varepsilon})-\beta(\psi_1))(u_{\delta,\varepsilon}-H_*)\,\dd x\,\dd s\\
&\quad+\int_0^t\!\int_DP_\varepsilon(x,s,u_{\delta,\varepsilon})(u_{\delta,\varepsilon}-H_*)\,\dd x\,\dd s.
\end{align*}
Now, we estimate each term. 
On the set where $P_\varepsilon(x,t,u_{\delta,\varepsilon})>0$, we have $u_{\delta,\varepsilon}<\psi$. 
Lemma \ref{lem:principal-penalty-localization} shows that the penalty term involving $P_\varepsilon$ is active only on $K_2\times[0,T]$.
Since $H_*\geq\psi+\sigma_*$ on $K_2\times[0,T]$, we have
\begin{align*}
&P_\varepsilon(x,t,u_{\delta,\varepsilon}) (u_{\delta,\varepsilon}-H_*)\leq-P_\varepsilon(x,t,u_{\delta,\varepsilon})(\psi-u_{\delta,\varepsilon})^+-\sigma_*P_\varepsilon(x,t,u_{\delta,\varepsilon}).
\end{align*}

For the time term, by Lemma \ref{lem:AL-time-chain-rule}, we have, for almost every $t\in(0,T)$,
\[
\int_0^t\left\langle\partial_s\beta(u_{\delta,\varepsilon}),u_{\delta,\varepsilon}\right\rangle_{\mathbb V^*,\mathbb V}\,\dd s=\int_D B(u_{\delta,\varepsilon}(t))\,\dd x-\int_D B(\xi)\,\dd x.
\]
Moreover, Lemma \ref{lem:AL-fixed-spatial-test} gives, for almost every $t\in(0,T)$,
\[
\int_0^t\left\langle\partial_s\beta(u_{\delta,\varepsilon}),H_*\right\rangle_{\mathbb V^*,\mathbb V}\,\dd s=\int_D[\beta(u_{\delta,\varepsilon}(t))-\beta(\xi)]H_*\,\dd x.
\]
Since \eqref{eq:beta-controlled-by-energy} holds and $H_*\in L^\infty(D)$, we have
\begin{align*}
&\int_0^t\left\langle\partial_s\beta(u_{\delta,\varepsilon}),u_{\delta,\varepsilon}-H_*\right\rangle_{\mathbb V^*,\mathbb V}\,\dd s\geq\frac12\int_D B(u_{\delta,\varepsilon}(t))\,\dd x-C\left(1+\int_D B(\xi)\,\dd x\right).
\end{align*}
By \eqref{eq:A-coercivity}, \eqref{eq:A-growth}, and Young's inequality,
\begin{align*}
&-\int_0^t\!\int_D A(u_{\delta,\varepsilon},\nabla u_{\delta,\varepsilon})\cdot\nabla(u_{\delta,\varepsilon}-H_*)\,\dd x\,\dd s\\
&\qquad\leq-\frac{c_{\mathrm{coer}}}{2}\int_0^t\!\int_D|\nabla u_{\delta,\varepsilon}|^p\,\dd x\,\dd s+C\int_0^t\!\int_D(1+B(u_{\delta,\varepsilon}))\,\dd x\,\dd s.
\end{align*}
The source term is estimated by \eqref{eq:f-energy-growth}, Young's inequality, and Poincar\'e's inequality.
\begin{align*}
&\int_0^t\!\int_D|f(x,s,u_{\delta,\varepsilon})||u_{\delta,\varepsilon}-H_*|\,\dd x\,\dd s\\
&\quad\leq\frac{c_{\mathrm{coer}}}{4}\int_0^t\!\int_D|\nabla u_{\delta,\varepsilon}|^p\,\dd x\,\dd s+C\bigg(1+\int_0^t\!\int_D\bigl(1+B(u_{\delta,\varepsilon})+|f_0|^{p'}\bigr)\,\dd x\,\dd s\bigg).
\end{align*}
For the term involving $\alpha_\varepsilon$, using Lemma \ref{lem:principal-penalty-localization}, on the support of $\kappa_\delta(\beta(u_{\delta,\varepsilon})-\beta(\psi_1))$, one has
\[
0\leq\beta(u_{\delta,\varepsilon})-\beta(\psi_1)<\delta.
\]
Since $|\psi_1|\leq M_{\mathrm{obs}}$, $\delta<1$, and $\beta$ is strictly increasing, $u_{\delta,\varepsilon}$ is bounded there independently of the parameters $\delta$ and $\varepsilon$.
Using \eqref{eq:alpha-uniform-bound}, the penalty term involving $\alpha_\varepsilon$ tested against $u_{\delta,\varepsilon}-H_*$ is bounded above by $C\|r_{\psi_1}^+\|_{L^1(D_T)}$.

Combining all the preceding estimates, and using Gr\"onwall's inequality, we have \eqref{eq:uniform-energy-general-obstacle}, which completes the proof.
\end{proof}

\subsection{The first penalization limit \texorpdfstring{$\delta\downarrow0$}{as delta tends to zero}}\label{subsec:first-penalty-limit}
For fixed $\varepsilon>0$ and $z\in\mathbb{R}$, the map $\alpha_\varepsilon\kappa_\delta(z-\beta(\psi_1))$ is nondecreasing with $\delta$.
Lemma \ref{lem:penalized-comparison-general-A} therefore gives almost everywhere
\begin{equation}
u_{\delta_2,\varepsilon}\leq u_{\delta_1,\varepsilon},\quad\text{if }0<\delta_2<\delta_1<1. \label{eq:delta-monotonicity-revised}
\end{equation}
Combining with Lemma \ref{lem:principal-penalty-localization} and Lemma \ref{lem:uniform-energy-general-obstacle}, we obtain the first limit.

\begin{prop}[First penalization limit in the Alt--Luckhaus class]\label{prop:first-penalty-limit}
Let Assumptions \ref{assu:existence-beta}--\ref{assu:existence-obstacle} hold. 
For every fixed $\varepsilon\in(0,1)$, there exist
\[
u_\varepsilon\in L^p(0,T;\mathbb V), \qquad \nu_{1,\varepsilon}\in L^\infty(D_T),
\]
such that, as $\delta\downarrow0$,
\begin{align*}
u_{\delta,\varepsilon} &\longrightarrow u_\varepsilon &&\text{a.e. in $D_T$ and strongly in $L^p(D_T)$},\\
u_{\delta,\varepsilon} &\rightharpoonup u_\varepsilon &&\text{weakly in $L^p(0,T;\mathbb V)$},\\
B(u_{\delta,\varepsilon}) &\longrightarrow B(u_\varepsilon) &&\text{strongly in $L^1(D_T)$},\\
\beta(u_{\delta,\varepsilon}) &\longrightarrow\beta(u_\varepsilon) &&\text{strongly in $L^1(D_T)$},\\
\alpha_\varepsilon\kappa_\delta \bigl(\beta(u_{\delta,\varepsilon})-\beta(\psi_1)\bigr) &\overset{*}{\rightharpoonup}\nu_{1,\varepsilon} &&\text{weakly-* in }L^\infty(D_T).
\end{align*}
Moreover,
\begin{align}
u_\varepsilon&\geq\psi_1 &&\text{a.e. in }D_T, \label{eq:first-limit-obstacle}\\
0\leq\nu_{1,\varepsilon}&\leq r_{\psi_1}^+ &&\text{a.e. in }D_T, \label{eq:nu1-epsilon-bound}\\
\nu_{1,\varepsilon}&=0 &&\text{a.e. on }\{u_\varepsilon>\psi_1\}. \label{eq:nu1-contact}
\end{align}
Furthermore,
\[
B(u_\varepsilon)\in L^\infty(0,T;L^1(D)),
\]
and $u_\varepsilon$ is an Alt--Luckhaus weak solution, in the sense of Definition \ref{def:AL-weak-solution}, of
\begin{equation}
\partial_t\beta(u_\varepsilon) = \operatorname{div}A(u_\varepsilon,\nabla u_\varepsilon) + f(x,t,u_\varepsilon) +\nu_{1,\varepsilon}+ P_\varepsilon(x,t,u_\varepsilon)  \quad\text{in }D_T, \label{eq:epsilon-equation}
\end{equation}
with homogeneous boundary data and initial datum $\xi$.
\end{prop}

\begin{proof}
By \eqref{eq:delta-monotonicity-revised}, the pointwise limit $u_\varepsilon:=\lim_{\delta\downarrow0}u_{\delta,\varepsilon}$ exists almost everywhere in $D_T$.
Moreover, note that
\[
 \psi_1\leq u_{\delta,\varepsilon} \leq u_{\delta_0,\varepsilon} \quad\text{a.e. in }D_T
\]
for every fixed $\delta_0\in(0,1)$ and every $0<\delta\leq\delta_0$.
Hence the dominated convergence theorem and Lemma \ref{lem:uniform-energy-general-obstacle} give
\[
u_{\delta,\varepsilon}\longrightarrow u_\varepsilon \quad\text{strongly in }L^p(D_T),\qquad u_{\delta,\varepsilon}\rightharpoonup u_\varepsilon \quad\text{weakly in }L^p(0,T;\mathbb V).
\]
Since the nonnegative function $B$ is nonincreasing on $(-\infty,0]$ and nondecreasing on $[0,\infty)$, it follows that
\[
0\leq B(u_{\delta,\varepsilon}) \leq B(\psi_1)+B(u_{\delta_0,\varepsilon}) \quad\text{a.e. in }D_T.
\]
The right-hand side belongs to $L^1(D_T)$. 
Then the pointwise convergence and the dominated convergence theorem give
\[
B(u_{\delta,\varepsilon}) \longrightarrow B(u_\varepsilon) \quad\text{strongly in }L^1(D_T).
\]
For the convergence of $\beta$, using \eqref{eq:beta-controlled-by-energy} and Vitali's theorem, we have
\[
\beta(u_{\delta,\varepsilon}) \longrightarrow\beta(u_\varepsilon)\quad\text{strongly in }L^1(D_T).
\]
Moreover, Fatou's lemma and \eqref{eq:uniform-energy-general-obstacle} give
\[
B(u_\varepsilon)\in L^\infty(0,T;L^1(D)).
\]
Set
\[
r^{(1)}_{\delta,\varepsilon} := \alpha_\varepsilon\kappa_\delta \bigl(\beta(u_{\delta,\varepsilon})-\beta(\psi_1)\bigr)\geq0.
\]
By \eqref{eq:alpha-uniform-bound} and definition of $\kappa_\delta$, we have $r^{(1)}_{\delta,\varepsilon}\leq r_{\psi_1}^+$  almost everywhere in $D_T$.
Then, after passing to a subsequence, we have
\[
r^{(1)}_{\delta,\varepsilon} \overset{*}{\rightharpoonup}\nu_{1,\varepsilon} \quad\text{in }L^\infty(D_T),
\]
and \eqref{eq:nu1-epsilon-bound} follows.
Since $u_{\delta,\varepsilon}\geq\psi_1$, passage to the pointwise limit gives \eqref{eq:first-limit-obstacle}.
Moreover, on the support of $r^{(1)}_{\delta,\varepsilon}$, we have
\[
0\leq \beta(u_{\delta,\varepsilon})-\beta(\psi_1) \leq\delta,\qquad0\leq r^{(1)}_{\delta,\varepsilon}\leq r^+_{\psi_1}.
\]
Then, we have
\[
\int_{D_T} r^{(1)}_{\delta,\varepsilon} \bigl[ \beta(u_{\delta,\varepsilon})-\beta(\psi_1) \bigr] \,\dd x\,\dd t \longrightarrow0,\quad\text{as }\delta\to0.
\]
Using the weak-* convergence of $r^{(1)}_{\delta,\varepsilon}$ in $L^\infty(D_T)$ and the strong $L^1(D_T)$-convergence of $\beta(u_{\delta,\varepsilon})$, we obtain
\[
\nu_{1,\varepsilon} \bigl[ \beta(u_\varepsilon)-\beta(\psi_1) \bigr] =0 \quad\text{a.e. in }D_T.
\]
Since $\beta$ is strictly increasing, this implies \eqref{eq:nu1-contact}.

It remains to identify the flux and to verify that the limit retains the Alt--Luckhaus formulation.
Fix an arbitrary sequence $\delta_j\downarrow0$.
By the uniform estimates \eqref{eq:uniform-energy-general-obstacle} and \eqref{eq:A-growth}, after passing to a subsequence,
\[
A(u_{\delta_j,\varepsilon},\nabla u_{\delta_j,\varepsilon})\rightharpoonup w_\varepsilon \quad\text{weakly in }L^{p'}(D_T)^d.
\]
For fixed $\varepsilon$, the continuity and boundedness of the penalty imply
\[
P_\varepsilon(\cdot,\cdot,u_{\delta_j,\varepsilon})\longrightarrow P_\varepsilon(\cdot,\cdot,u_\varepsilon) \quad\text{strongly in }L^q(D_T),
\]
for every finite $q$.
For the convergence of $f$, the almost everywhere convergence of $u_{\delta_j,\varepsilon}$ and the continuity of $f(x,t,\cdot)$ give the almost everywhere convergence of $f(\cdot,\cdot,u_{\delta_j,\varepsilon})$ to $f(\cdot,\cdot,u_\varepsilon)$ on $D_T$.
Moreover, note that the order bounds imply $B(u_{\delta_j,\varepsilon}) \leq B(\psi_1)+B(u_{\delta_1,\varepsilon})$ almost everywhere on $D_T$, while the right-hand side is integrable. 
Combining \eqref{eq:f-energy-growth},  the dominated convergence theorem yields
\[
f(\cdot,\cdot,u_{\delta_j,\varepsilon}) \longrightarrow f(\cdot,\cdot,u_\varepsilon) \quad\text{strongly in }L^{p'}(D_T).
\]
Combining all these convergences and passing to the limit, with smooth test functions, gives
\begin{equation}
\partial_t\beta(u_\varepsilon)=\operatorname{div}w_\varepsilon + f(x,t,u_\varepsilon)+\nu_{1,\varepsilon}+ P_\varepsilon(x,t,u_\varepsilon)\quad\text{in }\mathcal D'(D_T).\label{eq:first-limit-unidentified-flux}
\end{equation}
The growth assumptions imply that the right-hand side belongs to $L^{p'}(0,T;\mathbb V^*)$.
Hence
\[
\partial_t\beta(u_\varepsilon) \in L^{p'}(0,T;\mathbb V^*),
\]
and, by density, \eqref{eq:first-limit-unidentified-flux} holds in the dual formulation for every $\varphi\in L^p(0,T;\mathbb V)$.

We next pass to the limit $j\to\infty$ in the Alt--Luckhaus initial-value identity.
For every $\zeta\in L^p(0,T;\mathbb V) \cap W^{1,1}(0,T;L^\infty(D))$ satisfying $\zeta(T)=0$, we have
\[
\int_0^T \left\langle \partial_t\beta(u_{\delta_j,\varepsilon}),\zeta \right\rangle_{\mathbb V^*,\mathbb V}\,\dd t + \int_{D_T} [\beta(u_{\delta_j,\varepsilon})-\beta(\xi)] \partial_t\zeta\,\dd x\,\dd t=0.
\]
The equations and the preceding weak convergence of each term give the weak convergence of $\partial_t\beta(u_{\delta_j,\varepsilon})$ to $\partial_t\beta(u_\varepsilon)$ in $L^{p'}(0,T;\mathbb V^*)$.
Furthermore, we have $\beta(u_{\delta_j,\varepsilon})\downarrow\beta(u_\varepsilon)$ almost everywhere on $D_T$,
and
\[
0\leq \beta(u_{\delta_j,\varepsilon}(t))-\beta(u_\varepsilon(t)) \leq \beta(u_{\delta_1,\varepsilon}(t))-\beta(\psi_1(t)).
\]
The $L^1(D)$-norm of the right-hand side is essentially bounded in time.
Therefore, for every $\partial_t\zeta\in L^1(0,T;L^\infty(D))$, the dominated convergence theorem gives
\[
\int_{D_T} [\beta(u_{\delta_j,\varepsilon})-\beta(u_\varepsilon)] \partial_t\zeta\,\dd x\,\dd t \longrightarrow0,\qquad\text{as }j\to\infty.
\]
It follows that
\[
\int_0^T \left\langle \partial_t\beta(u_\varepsilon),\zeta \right\rangle_{\mathbb V^*,\mathbb V}\,\dd t + \int_{D_T} [\beta(u_\varepsilon)-\beta(\xi)] \partial_t\zeta\,\dd x\,\dd t =0.
\]
Thus $u_\varepsilon$ satisfies the Alt--Luckhaus initial-value condition.

We now identify $w_\varepsilon$ in \eqref{eq:first-limit-unidentified-flux} with a Minty's argument.
With the weak-* convergence of $r^{(1)}_{\delta_j,\varepsilon}$ in $L^\infty(D_T)$ and the strong convergence of $u_{\delta_j,\varepsilon}$ in $L^1(D_T)$, we have for every $\tau\in[0,T]$
\begin{equation}\label{eq:first-reaction-product}
\int_0
^\tau\int_{D}r^{(1)}_{\delta_j,\varepsilon}u_{\delta_j,\varepsilon}\, \dd x\,\dd t \longrightarrow \int_0
^\tau\int_{D}\nu_{1,\varepsilon}u_\varepsilon\, \dd x\,\dd t\quad\text{as }j\to\infty.
\end{equation}
Applying Lemma \ref{lem:AL-time-chain-rule} to $u_{\delta_j,\varepsilon}$ and to $u_\varepsilon$, we have for almost every $\tau\in(0,T)$,
\begin{align*}
\int_0^\tau\int_D A(u_{\delta_j,\varepsilon},\nabla u_{\delta_j,\varepsilon})\cdot\nabla u_{\delta_j,\varepsilon}\, \dd x\,\dd t&= \int_0^\tau\int_D \bigl[f(x,t,u_{\delta_j,\varepsilon})+ r^{(1)}_{\delta_j,\varepsilon}+ P_\varepsilon(x,t,u_{\delta_j,\varepsilon})  \bigr]u_{\delta_j,\varepsilon} \, \dd x\,\dd t\\
&\quad- \int_D B(u_{\delta_j,\varepsilon}(\tau)) \, \dd x+ \int_D B(\xi)\, \dd x,
\end{align*}
whereas
\begin{align*}
\int_0^\tau\int_D w_\varepsilon\cdot\nabla u_\varepsilon \, \dd x\,\dd t&=\int_0^\tau\int_D \bigl[ f(x,t,u_\varepsilon)+\nu_{1,\varepsilon}+ P_\varepsilon(x,t,u_\varepsilon)\bigr]u_\varepsilon \, \dd x\,\dd t\\
&\quad- \int_D B(u_\varepsilon(\tau))\, \dd x+ \int_D B(\xi)\, \dd x.
\end{align*}
After passing to a further subsequence, the strong $L^1(D_T)$ convergence of $B(u_{\delta_j,\varepsilon})$ gives
\[
B(u_{\delta_j,\varepsilon}(\tau))\longrightarrow B(u_\varepsilon(\tau))\quad\text{strongly in }L^1(D)
\]
for almost every $\tau\in(0,T)$.  
Then convergence of the energy, the source term and the reaction product \eqref{eq:first-reaction-product} give for almost every $\tau\in(0,T)$,
\begin{equation}
\lim_{j\to\infty} \int_0^\tau\int_DA(u_{\delta_j,\varepsilon},\nabla u_{\delta_j,\varepsilon})\cdot\nabla u_{\delta_j,\varepsilon} \, \dd x\,\dd t= \int_0^\tau\int_D w_\varepsilon\cdot\nabla u_\varepsilon \, \dd x\,\dd t. \label{eq:first-limit-energy-convergence}
\end{equation}
For every $\Theta\in L^p(D\times(0,\tau))^d$, pointwise monotonicity gives
\[
\int_0^\tau\int_{D}[A(u_{\delta_j,\varepsilon},\nabla u_{\delta_j,\varepsilon})-A(u_{\delta_j,\varepsilon},\Theta)]\cdot(\nabla u_{\delta_j,\varepsilon}-\Theta)\,\dd x\,\dd t\geq0.
\]
For fixed $\Theta$, continuity of $A$ and the a.e. convergence of $u_{\delta_j,\varepsilon}$ give
\[
A(u_{\delta_j,\varepsilon},\Theta)\longrightarrow A(u_\varepsilon,\Theta)\quad\text{a.e. in }D\times(0,\tau).
\]
Moreover,
\[
|A(u_{\delta_j,\varepsilon},\Theta)|^{p'}\leq C\bigl(1+B(u_{\delta_j,\varepsilon})+|\Theta|^p\bigr).
\]
The strong $L^1$ convergence of $B(u_{\delta_j,\varepsilon})$ implies uniform integrability of the right-hand side.
Vitali's theorem therefore yields strong $L^{p'}(D\times(0,\tau))^d$ convergence of these comparison fluxes.
Using \eqref{eq:first-limit-energy-convergence} and the remaining weak convergences, we obtain
\[
\int_0^\tau\int_{D}[w_\varepsilon-A(u_\varepsilon,\Theta)]\cdot(\nabla u_\varepsilon-\Theta)\,\dd x\,\dd t\geq0\qquad\text{for every }\Theta\in L^p(D\times(0,\tau))^d.
\]
This is the classical Minty--Browder identification; see
\cite{minty1962monotone,leray1965quelques}.
Indeed, fix an arbitrary $Z\in L^p(D\times(0,\tau))^d$ and take $\Theta=\nabla u_\varepsilon-\theta Z$, $\theta>0$.
After division by $\theta$ and passage to the limit $\theta\downarrow0$, the continuity and growth of $A$ give
\[
\int_0^\tau\int_{D}[w_\varepsilon-A(u_\varepsilon,\nabla u_\varepsilon)]\cdot Z\,\dd x\,\dd t\geq0.
\]
Replacing $Z$ by $-Z$ gives equality for every $Z\in L^p(D\times(0,\tau))^d$.
Hence
\[
w_\varepsilon=A(u_\varepsilon,\nabla u_\varepsilon)\quad\text{a.e. in }D\times(0,\tau).
\]
Since $\tau$ can be chosen arbitrarily close to $T$, the identity holds almost everywhere in $D_T$.
Consequently, \eqref{eq:first-limit-unidentified-flux} becomes \eqref{eq:epsilon-equation}, and all three requirements of Definition \ref{def:AL-weak-solution} are satisfied.

The sequence $\delta_j\downarrow0$ was arbitrary.
The monotone pointwise limit is unique, and the limiting reaction is uniquely determined by \eqref{eq:epsilon-equation}.
Hence the stated convergences hold for the whole family as $\delta\downarrow0$.
\end{proof}

\subsection{The second penalization limit \texorpdfstring{$\varepsilon\downarrow0$}{as epsilon tends to zero}}

We next compare the family $(u_\varepsilon)$. 
The contact property of the first reaction supplies the required sign.

\begin{lem}[Monotonicity in $\varepsilon$]\label{lem:epsilon-monotonicity-revised}
Let Assumptions \ref{assu:existence-beta}--\ref{assu:existence-obstacle} hold. 
If $0<\varepsilon_2<\varepsilon_1<1$, then
\[
u_{\varepsilon_1}\leq u_{\varepsilon_2}\quad\text{a.e. in }D_T.
\]
\end{lem}

\begin{proof}
Apply the Kato inequality from the proof of Lemma \ref{lem:penalized-comparison-general-A} to \eqref{eq:epsilon-equation} with $u_{\varepsilon_1}$ in the first position.
On $\{u_{\varepsilon_1}>u_{\varepsilon_2}\}$, relation \eqref{eq:nu1-contact} gives $\nu_{1,\varepsilon_1}=0$, while $-\nu_{1,\varepsilon_2}\leq0$.
Also,
\[
P_{\varepsilon_1}(x,t,r)\leq P_{\varepsilon_2}(x,t,r)
\]
and $r\mapsto P_{\varepsilon_2}(x,t,r)$ is nonincreasing.
Thus, on this set,
\begin{align*}
&P_{\varepsilon_1}(x,t,u_{\varepsilon_1})-P_{\varepsilon_2}(x,t,u_{\varepsilon_2})\leq P_{\varepsilon_2}(x,t,u_{\varepsilon_1})-P_{\varepsilon_2}(x,t,u_{\varepsilon_2})\leq0.
\end{align*}
The term involving $f$ is bounded by $L_f$ times the positive part of $\beta(u_{\varepsilon_1})-\beta(u_{\varepsilon_2})$.
Gr\"onwall's inequality completes the proof.
\end{proof}

The second limit contains a reaction that converges only as a measure.
We use the relative energy defined in \eqref{eq:def-relative-energy}.
We first prepare the initial value.
This step is needed because the initial datum is not assumed to belong to the energy space.

\begin{lem}[Initial approximation above the obstacle]\label{lem:initial-energy-approximation}
Let Assumptions \ref{assu:existence-beta}--\ref{assu:existence-obstacle} hold.
For every integer $k>M_{\mathrm{obs}}+2$, there exists
\[
\xi_k\in W_0^{1,p}(D)\cap L^\infty(D)
\]
such that
\begin{equation}
|\xi_k|\leq k\quad\text{a.e. in }D,\qquad \xi_k\geq\psi(\cdot,0)\quad\text{a.e. on }K_2\cap D,\label{eq:initial-approximation-properties}
\end{equation}
and
\begin{equation}
E_k:=\int_D\mathcal B_\beta(\xi\mid\xi_k)\,\dd x \longrightarrow0\quad\text{as }k\to\infty. \label{eq:initial-relative-energy-approximation}
\end{equation}
\end{lem}
\begin{proof}
Fix an integer $k>M_{\mathrm{obs}}+2$, and set
\[
\overline\xi_k:=\mathcal T_k(\xi).
\]
Since $\xi\geq\psi(\cdot,0)$ and $\psi(\cdot,0)\geq-M_{\mathrm{obs}}>-k$, we have
\[
\overline\xi_k=\min\{\xi,k\}\geq\psi(\cdot,0)\quad\text{a.e. in }D.
\]
Moreover, since $\overline\xi_k=\xi$ on $\{\xi\leq k\}$, we have
\[
0\leq\mathcal B_\beta(\xi\mid\overline\xi_k)\leq B(\xi)\one_{\{\xi>k\}}.
\]
Since $B(\xi)\in L^1(D)$, we therefore have
\begin{equation}
\int_D\mathcal B_\beta(\xi\mid\overline\xi_k)\,\dd x\longrightarrow0,\qquad\text{as }k\to\infty.\label{eq:initial-truncation-relative-energy-unified}
\end{equation}

We next construct an approximation that remains above $\psi(\cdot,0)$ on $K_2\cap D$ and vanishes on $\partial D$.
Choose $c_j\in C^\infty(\mathbb R^d)$ such that as $j\to\infty$
\[
c_j\longrightarrow\psi(\cdot,0)\quad\text{uniformly on }\overline D,
\]
and set
\[
e_j:=\|c_j-\psi(\cdot,0)\|_{L^\infty(D)}+j^{-1}.
\]
Then, we have
\[
c_j+e_j\geq\psi(\cdot,0)\quad\text{on }\overline D.
\]

If $K_2\cap\partial D=\varnothing$, then we choose $\eta_j\in C_c^\infty(D)$ such that $0\leq\eta_j\leq1$, $\eta_j=1$ on a neighborhood of $K_2$, and $\eta_j(x)\to1$ for every $x\in D$.
Suppose that $K_2\cap\partial D\neq\varnothing$. 
Since \eqref{eq:K2-boundary-gap} holds at $t=0$ and $\psi(\cdot,0)$ is uniformly continuous, there is $r_0>0$ such that
\[
\psi(x,0)\leq-\frac{\sigma_\partial}{2}\quad\text{for every }x\in K_2\text{ satisfying }d_{\partial D}(x)<r_0.
\]
Choose $r_j\downarrow0$ with $r_j<r_0$, and choose $\eta_j\in C_c^\infty(D)$ such that
\[
0\leq\eta_j\leq1,\qquad\{\eta_j<1\}\subset\{x\in D:d_{\partial D}(x)<r_j\},\qquad\eta_j(x)\to1
\]
for every $x\in D$.
Since $c_j+e_j\to\psi(\cdot,0)$ uniformly on $\overline{D}$ as $j\to\infty$, we have, for all sufficiently large $j$,
\[
c_j+e_j<0\quad\text{on }K_2\cap\{\eta_j<1\}.
\]
Therefore, on this set $K_2\cap\{\eta_j<1\}$, we have
\[
\eta_j(c_j+e_j)\geq c_j+e_j\geq\psi(\cdot,0).
\]
Where $\eta_j=1$, we also have $\eta_j(c_j+e_j)=c_j+e_j\geq\psi(\cdot,0)$.
Thus, we have
\[
\eta_j(c_j+e_j)\geq\psi(\cdot,0)\quad\text{on }K_2\cap D
\]
for all sufficiently large $j$.

Since $c_j+e_j\to\psi(\cdot,0)$ uniformly and $|\psi(\cdot,0)|\leq M_{\mathrm{obs}}$, we have $|c_j+e_j|\leq M_{\mathrm{obs}}+1<k$ for all sufficiently large $j$.
Define
\[
\theta_{k,j}:=\mathcal T_k(\eta_j(c_j+e_j)).
\]
Then we have $\theta_{k,j}\in W_0^{1,p}(D)\cap L^\infty(D)$, $|\theta_{k,j}|\leq k$, and
\[
\theta_{k,j}\geq\psi(\cdot,0)\quad\text{on }K_2\cap D.
\]
Moreover, since $\eta_j\to1$ pointwise and $c_j+e_j\to\psi(\cdot,0)$ uniformly on $\overline{D}$, we have as $j\to\infty$
\[
\theta_{k,j}\longrightarrow\psi(\cdot,0) \quad\text{a.e. in }D.
\]

Choose $a_{k,j}\in C_c^\infty(D)$ such that, after passing to a subsequence if necessary, it converges to $\overline\xi_k$ strongly in $L^p(D)$ and almost everywhere in $D$ as $j\to\infty$.
Set
\[
 b_{k,j}:=\max\{\mathcal T_k(a_{k,j}),\theta_{k,j}\}.
\]
Then, from the properties of $\theta_{k,j}$, we have $b_{k,j}\in W_0^{1,p}(D)$, $|b_{k,j}|\leq k$, and
\[
b_{k,j}\geq\psi(\cdot,0)\quad\text{on }K_2\cap D.
\]
Furthermore, since $\overline\xi_k\geq\psi(\cdot,0)$ a.e. in $D$, we have as $j\to\infty$
\[
b_{k,j}\longrightarrow\max\{\overline\xi_k,\psi(\cdot,0)\}=\overline\xi_k\quad\text{a.e. in }D.
\]

To prove \eqref{eq:initial-relative-energy-approximation}, for fixed $k$ and all $j$, we have
\[
0\leq\mathcal B_\beta(\xi\mid b_{k,j})\leq B(\xi)+k|\beta(\xi)|+C_k.
\]
Based on \eqref{eq:beta-controlled-by-energy}, the pointwise convergence of $b_{k,j}$, and the dominated convergence theorem, we have
\[
\int_D\mathcal B_\beta(\xi\mid b_{k,j})\,\dd x\longrightarrow\int_D\mathcal B_\beta(\xi\mid\overline\xi_k)\,\dd x.
\]
Choose $j(k)$ so large that the absolute difference between these two integrals is at most $1/k$, and set $\xi_k:=b_{k,j(k)}$.
Then \eqref{eq:initial-approximation-properties} holds.
Since \eqref{eq:initial-truncation-relative-energy-unified} also holds, we obtain \eqref{eq:initial-relative-energy-approximation}.
\end{proof}

For a Banach space $X$, $z\in L^p(0,T;X)$, and $z_0\in X$, define the one-sided time regularization \cite{landes1981existence}
\begin{equation}
R_\lambda[z,z_0](t):=\mathrm e^{-t/\lambda}z_0+\frac1\lambda\int_0^t\mathrm e^{-(t-s)/\lambda}z(s)\,\dd s.\label{eq:one-sided-time-regularization}
\end{equation}
Then $R_\lambda[z,z_0]\in W^{1,p}(0,T;X)$, and it satisfies
\begin{equation}
\lambda\partial_tR_\lambda[z,z_0]+R_\lambda[z,z_0]=z, \qquad R_\lambda[z,z_0](0)=z_0, \label{eq:one-sided-time-regularization-equation}
\end{equation}
and as $\lambda\downarrow0$,
\begin{equation*}
R_\lambda[z,z_0]\longrightarrow z \quad\text{strongly in }L^p(0,T;X). 
\end{equation*}
These facts follow directly from \eqref{eq:one-sided-time-regularization} and the continuity of translations in $L^p(0,T;X)$.
If $X$ is a Banach lattice of functions, then the nonnegative weights in \eqref{eq:one-sided-time-regularization} also preserve pointwise order and pointwise bounds.

We now give the compactness used in the second penalization limit.
By Lemma \ref{lem:epsilon-monotonicity-revised}, for a sequence $\varepsilon_n\downarrow0$, we have
\[
\psi_1\leq u_{\varepsilon_n}\leq u_{\varepsilon_{n+1}}\quad\text{a.e. in }D_T.
\]
Let $u$ be the pointwise limit. 
On the other hand, by \eqref{eq:uniform-energy-general-obstacle}, after passing to a subsequence, there exists $v\in L^p(0,T;W_0^{1,p}(D))$ such that $u_{\varepsilon_n}$ weakly converges to $v$ in $L^p(0,T;W_0^{1,p}(D))$.
Since $u_{\varepsilon_n}\to u$ a.e. in $D_T$, the weak limit is necessarily $v=u$.  
Hence
\[
u\in L^p(0,T;W_0^{1,p}(D)),\qquad\nabla u_{\varepsilon_n}\rightharpoonup\nabla u\quad\text{weakly in }L^p(D_T)^d.
\]
Fatou's lemma applied to the energy term gives $B(u)\in L^1(D_T)$.
Since $\psi_1\leq u_{\varepsilon_n}\leq u$, we have
\[
|u_{\varepsilon_n}|\leq|\psi_1|+|u|,\qquad B(u_{\varepsilon_n})\leq B(\psi_1)+B(u).
\]
The dominated convergence theorem therefore gives
\[
u_{\varepsilon_n}\to u\quad\text{strongly in }L^p(D_T),\qquad B(u_{\varepsilon_n})\to B(u)\quad\text{strongly in }L^1(D_T).
\]
The estimate \eqref{eq:beta-controlled-by-energy} implies uniform integrability of $\{\beta(u_{\varepsilon_n})\}_n$.  
Since
$\beta(u_{\varepsilon_n})\to\beta(u)$ a.e., Vitali's theorem yields
\[
\beta(u_{\varepsilon_n})\to\beta(u)\quad\text{strongly in }L^1(D_T).
\]
The convergence of the source term in $L^{p'}(D_T)$ follows from \eqref{eq:f-energy-growth} and Vitali's theorem.

Regard the nonnegative measures $P_{\varepsilon_n}(\cdot,\cdot,u_{\varepsilon_n})\,\dd x\,\dd t$ as measures on the compact space $K_2\times[0,T]$.
The uniform estimate \eqref{eq:uniform-energy-general-obstacle} gives, after extraction, a measure $\overline\nu_2\in\mathcal M_+(K_2\times[0,T])$.
Together with the remaining uniform estimates in Lemma \ref{lem:uniform-energy-general-obstacle}, we obtain
\begin{equation}\label{eq:second-limit-convergences}
\begin{aligned}
u_{\varepsilon_n}&\to u &&\text{a.e. in $D_T$ and strongly in }L^p(D_T),\\
B(u_{\varepsilon_n})&\to B(u)&&\text{strongly in }L^1(D_T),\\
\beta(u_{\varepsilon_n})&\to\beta(u)&&\text{strongly in }L^1(D_T),\\
\nabla u_{\varepsilon_n}&\rightharpoonup\nabla u&&\text{weakly in }L^p(D_T)^d,\\
A(u_{\varepsilon_n},\nabla u_{\varepsilon_n})&\rightharpoonup w&&\text{weakly in }L^{p'}(D_T)^d,\\
f(\cdot,\cdot,u_{\varepsilon_n})&\to f(\cdot,\cdot,u)&&\text{strongly in }L^{p'}(D_T),\\
\nu_{1,\varepsilon_n}&\overset{*}{\rightharpoonup}\nu_1&&\text{weakly-* in }L^\infty(D_T),\\
P_{\varepsilon_n}(\cdot,\cdot,u_{\varepsilon_n})\,\dd x\,\dd t&\overset{*}{\rightharpoonup}\overline\nu_2&&\text{weakly-* in }\mathcal M(K_2\times[0,T]).
\end{aligned}
\end{equation}
Moreover, Fatou's lemma and \eqref{eq:uniform-energy-general-obstacle} give
\[
\int_D B(u(t))\,\dd x\leq\liminf_{n\to\infty}\int_D B(u_{\varepsilon_n}(t))\,\dd x\leq C
\]
for a.e. $t\in(0,T)$. 
Therefore we have $B(u)\in L^\infty(0,T;L^1(D))$.
Define
\begin{equation}
\nu_2:=\overline\nu_2|_{(K_2\cap D)\times[0,T)},\label{eq:def for nu2}
\end{equation}
The extended limit $\overline\nu_2$ may also contain lateral-boundary or terminal-time mass. 
These parts are not included in the reaction measure $\nu_2$ of the solution: the terminal part is invisible to admissible tests, while the boundary part either does not meet the test or can be discarded in the boundary-touching entropy inequality.
This is verified explicitly in Subsection \ref{subsec:entropy-passage}.
Here and below, $\nu_2$ is extended by zero from $(K_2\cap D)\times[0,T)$ to $D\times[0,T)$.

Note that \eqref{eq:uniform-energy-general-obstacle} indicates
\[
\int_{D_T}\min\{(\beta(\psi)-\beta(u_{\varepsilon_n}))^+,K_P\}\,\dd x\,\dd t\leq C\varepsilon_n.
\]
Since $u_{\varepsilon_n}\to u$ almost everywhere, Fatou's lemma implies
\[
\min\{(\beta(\psi)-\beta(u))^+,K_P\}=0\quad\text{a.e. in }D_T.
\]
The strict monotonicity of $\beta$ therefore gives
\begin{equation}
u\geq\psi\quad\text{a.e. in }D_T.\label{eq:second-limit-obstacle-order}
\end{equation}

Furthermore, with \eqref{eq:nu1-contact}, we have $\nu_{1,\varepsilon_n}(u_{\varepsilon_n}-\psi_1)=0$ a.e. Passing to the limit $n\to\infty$ and using the weak-* convergence of $\nu_{1,\varepsilon_{n}}$ in $L^\infty(D_T)$ and strong convergence of $u_{\varepsilon_n}$, we have
\begin{equation}
\nu_1(u-\psi_1)=0 \quad\text{a.e. in }D_T. \label{eq:second-limit-contact}
\end{equation}

We now identify the weak flux limit. 
\begin{lem}[Flux identification in the second penalization limit]\label{lem:minty-obstacle-reactions}
Let Assumptions \ref{assu:existence-beta}--\ref{assu:existence-obstacle} hold.
For the sequence fixed in \eqref{eq:second-limit-convergences},
\begin{equation}
\nabla u_{\varepsilon_n}\longrightarrow\nabla u \quad\text{strongly in }L^p(D_T)^d,\label{eq:minty-gradient-strong}
\end{equation}
and
\begin{equation}
A(u_{\varepsilon_n},\nabla u_{\varepsilon_n}) \longrightarrow A(u,\nabla u) \quad\text{strongly in }L^{p'}(D_T)^d.\label{eq:minty-flux-strong}
\end{equation}
Consequently,
\begin{equation}
w=A(u,\nabla u) \quad\text{a.e. in }D_T. \label{eq:minty-flux-identification}
\end{equation}
\end{lem}
\begin{proof}
Fix $k>M_{\mathrm{obs}}+2$ and let $\xi_k$ be given by Lemma \ref{lem:initial-energy-approximation}.
Since $u\geq\psi\geq-M_{\mathrm{obs}}>-k$,
\begin{equation}
\mathcal T_k(u)\geq\psi\quad\text{a.e. in }D_T.\label{eq:truncation-above-obstacle}
\end{equation}
If $K_2=\varnothing$, then the penalty term involving $P_\varepsilon$ vanishes, and no obstacle error is present.
Hence assume below that $K_2\neq\varnothing$. 
Define
\[
\omega_{\psi,K_2}(\rho):= \sup_{\substack{x\in K_2,\ t,s\in[0,T]\\|t-s|\leq\rho}}|\psi(x,t)-\psi(x,s)|
\]
and, for $\lambda>0$,
\begin{align}
z_{k,\lambda}&:=R_\lambda[\mathcal T_k(u),\xi_k], \label{eq:regularized-truncation-definition}\\
c_\lambda&:=\sup_{0\leq t\leq T}\left[e^{-t/\lambda}\omega_{\psi,K_2}(t)+\frac1\lambda\int_0^t e^{-(t-s)/\lambda}\omega_{\psi,K_2}(t-s)\,\dd s\right].\label{eq:obstacle-time-deficit}
\end{align}
The weights in \eqref{eq:one-sided-time-regularization} are nonnegative and have total mass one.
Hence we have $|z_{k,\lambda}|\leq k$. 
Moreover,
\begin{equation}
z_{k,\lambda}\to\mathcal T_k(u)\quad\text{strongly in }L^p(0,T;\mathbb V),\qquad z_{k,\lambda}(0)=\xi_k,\label{eq:zklambda-convergence}
\end{equation}
and
\begin{equation}
\partial_tz_{k,\lambda}=\frac{\mathcal T_k(u)-z_{k,\lambda}}{\lambda},\qquad|\partial_tz_{k,\lambda}|\leq\frac{2k}{\lambda}\quad \text{on }D_T.\label{eq:zklambda-time-derivative}
\end{equation}
Using $\xi_k\geq\psi(\cdot,0)$ on $K_2\cap D$ and \eqref{eq:truncation-above-obstacle}, we obtain
\begin{equation}
z_{k,\lambda}\geq\psi-c_\lambda\quad\text{a.e. on }(K_2\cap D)\times(0,T).\label{eq:regularization-above-obstacle-up-to-error}
\end{equation}
The uniform continuity of $\psi$ and the property of $R_\lambda$ give
\begin{equation}
c_\lambda\longrightarrow0\quad\text{as }\lambda\downarrow0.\label{eq:regularization-error-zero}
\end{equation}

Testing \eqref{eq:epsilon-equation} with $u_{\varepsilon_n}-z_{k,\lambda}$, we have
\begin{align}
\int_0^T\langle\partial_t\beta(u_{\varepsilon_n}),u_{\varepsilon_n}-z_{k,\lambda}\rangle_{\mathbb V^*,\mathbb V}\dd t&=-\int_0^T\!\int_DA(u_{\varepsilon_n},\nabla u_{\varepsilon_n})\cdot\nabla(u_{\varepsilon_n}-z_{k,\lambda})\,\dd x\,\dd t\nonumber\\
&\quad+\int_0^T\!\int_Df(x,t,u_{\varepsilon_n})(u_{\varepsilon_n}-z_{k,\lambda})\,\dd x\,\dd t\nonumber\\
&\quad+\int_0^T\!\int_D(u_{\varepsilon_n}-z_{k,\lambda})\nu_{1,\varepsilon_n}\,\dd x\,\dd t\nonumber\\
&\quad+\int_0^T\!\int_DP_{\varepsilon_n}(x,t,u_{\varepsilon_n})(u_{\varepsilon_n}-z_{k,\lambda})\,\dd x\,\dd t.\label{eq:testu-z}
\end{align}
For fixed $n,k,\lambda$, set
\[
I_{n,k,\lambda}:=\int_0^T\left\langle\partial_t\beta(u_{\varepsilon_n}),u_{\varepsilon_n}-z_{k,\lambda}\right\rangle_{\mathbb V^*,\mathbb V}\,\dd t.
\]
Applying Lemma \ref{lem:relative-energy-chain-rule} with $v=u_{\varepsilon_n}$ and $z=z_{k,\lambda}$, and using the nonnegativity at $t=T$ of the absolutely continuous representative of the relative energy, we have
\[
I_{n,k,\lambda}\geq-E_k+\int_{D_T}[\beta(u_{\varepsilon_n})-\beta(z_{k,\lambda})]\partial_tz_{k,\lambda}\,\dd x\,\dd t.
\]

Since $\beta(u_{\varepsilon_n})\to\beta(u)$ strongly in $L^1(D_T)$, with $\partial_tz_{k,\lambda}\in L^\infty(D_T)$, we may pass to the limit $n\to\infty$ to obtain
\begin{equation}\label{eq:time-term-after-n-limit}
\liminf_{n\to\infty}I_{n,k,\lambda} \geq-E_k+\int_{D_T}[\beta(u)-\beta(z_{k,\lambda})]\partial_tz_{k,\lambda}\,\dd x\,\dd t.
\end{equation}
By \eqref{eq:zklambda-time-derivative}, we have
\begin{equation}
\int_{D_T}[\beta(u)-\beta(z_{k,\lambda})]\partial_tz_{k,\lambda}\,\dd x\,\dd t=\frac1\lambda\int_{D_T}[\beta(u)-\beta(z_{k,\lambda})] [\mathcal T_k(u)-z_{k,\lambda}]\,\dd x\,\dd t\geq0. \label{eq:positive-regularized-time-term}
\end{equation}
Indeed, \eqref{eq:positive-regularized-time-term} follows from monotonicity of $\beta$ when $u\leq k$; when $u>k$, both factors are nonnegative, and $u<-k$ is impossible because $u\geq-M_{\mathrm{obs}}>-k$.

On the set where $P_{\varepsilon_n}(\cdot,\cdot,u_{\varepsilon_n})>0$, we have $u_{\varepsilon_n}<\psi$, and this set is contained in $(K_2\cap D)\times[0,T]$ up to a null set.
Hence \eqref{eq:regularization-above-obstacle-up-to-error} and \eqref{eq:uniform-energy-general-obstacle} give
\begin{align}
&\int_{D_T}P_{\varepsilon_n}(\cdot,\cdot,u_{\varepsilon_n}) (u_{\varepsilon_n}-z_{k,\lambda})\,\dd x\,\dd t \leq c_\lambda\int_{D_T} P_{\varepsilon_n}(\cdot,\cdot,u_{\varepsilon_n})\,\dd x\,\dd t \leq Cc_\lambda. \label{eq:penalty-comparison-error}
\end{align}
Subtracting $\int A(u_{\varepsilon_n},\nabla z_{k,\lambda})\cdot\nabla(u_{\varepsilon_n}-z_{k,\lambda})$ from \eqref{eq:testu-z} and using \eqref{eq:A-uniform-monotonicity}, we obtain
\begin{align}
a_0\int_{D_T}|\nabla u_{\varepsilon_n}-\nabla z_{k,\lambda}|^p\,\dd x\,\dd t&\leq\int_{D_T}(f(\cdot,\cdot,u_{\varepsilon_n})+\nu_{1,\varepsilon_n})(u_{\varepsilon_n}-z_{k,\lambda})\,\dd x\,\dd t-I_{n,k,\lambda}\nonumber\\
&\quad-\int_{D_T}A(u_{\varepsilon_n},\nabla z_{k,\lambda})\cdot(\nabla u_{\varepsilon_n}-\nabla z_{k,\lambda})\,\dd x\,\dd t+Cc_\lambda. \label{eq:coercive-comparison-estimate}
\end{align}
For fixed $k$ and $\lambda$, using continuity and growth of $A$, and the strong $L^1$ convergence of $B(u_{\varepsilon_n})$, we have
\[
A(u_{\varepsilon_n},\nabla z_{k,\lambda})\to A(u,\nabla z_{k,\lambda})\quad\text{strongly in }L^{p'}(D_T)^d.
\]
Thus \eqref{eq:second-limit-convergences}, \eqref{eq:time-term-after-n-limit}, and \eqref{eq:positive-regularized-time-term} allow us to pass to the limit $n\to\infty$ in \eqref{eq:coercive-comparison-estimate}.

Moreover, if $k>M_{\mathrm{obs}}$, then $\mathcal T_k(u)=u$ on $\{u=\psi_1\}$.
Then using \eqref{eq:second-limit-contact}, we have
\[
\int_{D_T}\nu_1(u-\mathcal T_k(u))\,\dd x\,\dd t=0.
\]
Then let $\lambda\downarrow0$ in the preceding estimate and use \eqref{eq:zklambda-convergence} and \eqref{eq:regularization-error-zero}.
We obtain 
\begin{align*} a_0\liminf_{\lambda\downarrow0}\limsup_{n\to\infty} \int_{D_T}|\nabla u_{\varepsilon_n}-\nabla z_{k,\lambda}|^p\,\dd x\,\dd t&\leq E_k+\int_{D_T}(f(\cdot,\cdot,u)+\nu_1)(u-\mathcal T_k(u))\,\dd x\,\dd t\\
&\quad-\int_{D_T}A(u,\nabla\mathcal T_k(u)) \cdot\nabla(u-\mathcal T_k(u))\,\dd x\,\dd t\\
&=:\mathfrak R_k.
\end{align*}
We now verify that the nonnegative $\mathfrak R_k$ converges to 0 as $k\to\infty$.
First, by Lemma \ref{lem:initial-energy-approximation}, we have $E_k\to0$.
Moreover, since $(f(\cdot,\cdot,u)+\nu_1)\in L^{p'}(D_T)$ and $\mathcal T_k(u)\to u$ strongly in $L^p(D_T)$, we have
\[
\int_{D_T}(f(x,t,u)+\nu_1)(u-\mathcal T_k(u))\,\dd x\,\dd t \longrightarrow0.
\]
Finally, since $\mathcal T_k(u)\to u$ strongly in $L^p(0,T;W_0^{1,p}(D))$ and the growth condition for $A$ holds, we have
\[
\int_{D_T}A(u,\nabla\mathcal T_k(u))\cdot\nabla(u-\mathcal T_k(u))\,\dd x\,\dd t\longrightarrow0.
\]
Consequently, we have $\mathfrak R_k\to0$.

For each $k$, choose $\lambda_k>0$ so small that
\begin{align*}
\limsup_{n\to\infty}\|\nabla u_{\varepsilon_n}-\nabla z_{k,\lambda_k}\|_{L^p(D_T)}^p\leq\frac{2}{a_0}\mathfrak R_k+\frac1k
\end{align*}
and
\[
\|\nabla z_{k,\lambda_k}-\nabla\mathcal T_k(u)\|_{L^p(D_T)}\leq\frac1k.
\]
Since
\begin{align*}
\|\nabla u_{\varepsilon_n}-\nabla u\|_{L^p(D_T)}&\leq\|\nabla u_{\varepsilon_n}-\nabla z_{k,\lambda_k}\|_{L^p(D_T)}+\|\nabla z_{k,\lambda_k}-\nabla\mathcal T_k(u)\|_{L^p(D_T)}\\
&\quad+\|\nabla\mathcal T_k(u)-\nabla u\|_{L^p(D_T)},
\end{align*}
we may first take the upper limit as $n\to\infty$ and then let $k\to\infty$.
Therefore, we obtain
\[
\nabla u_{\varepsilon_n}\longrightarrow\nabla u \quad\text{strongly in }L^p(D_T)^d.
\]
This proves \eqref{eq:minty-gradient-strong}.

Finally, note that
\[
|A(u_{\varepsilon_n},\nabla u_{\varepsilon_n})|^{p'} \leq C\bigl(1+B(u_{\varepsilon_n})+|\nabla u_{\varepsilon_n}|^p\bigr).
\]
The family on the right-hand side is uniformly integrable based on the strong convergence of both $B(u_{\varepsilon_n})$ and $|\nabla u_{\varepsilon_n}|^p$ in $L^1(D_T)$. 
Continuity of $A$ and Vitali's theorem yield
\[
A(u_{\varepsilon_n},\nabla u_{\varepsilon_n})\to A(u,\nabla u)\quad\text{strongly in }L^{p'}(D_T)^d.
\]
Thus $w=A(u,\nabla u)$, and the proof is complete.
\end{proof}
\begin{rem}[Time regularization in the flux identification]\label{rem:time regularization}
This regularization is essential in the second penalization limit.
After passage to the limit, the right-hand side of the equation for $\partial_t\beta(u_{\varepsilon_n})$ contains the weak-* limit of the penalty measures.  
Thus its limit need not belong to the Alt--Luckhaus space for the time-derivative $L^{p'}(0,T;\mathbb{V}^*)$, and a direct comparison with $u$ does not identify the time term.

The truncation $\mathcal T_k(u)$ already has the required spatial regularity and zero boundary trace.  
The time regularization $z_{k,\lambda}$ supplies the missing time derivative, so that the relative-energy identity can be applied before taking the measure limit.
Note that \eqref{eq:regularization-above-obstacle-up-to-error} gives $z_{k,\lambda}\geq\psi-c_\lambda$ on $(K_2\cap D)\times(0,T)$.  
Hence \eqref{eq:penalty-comparison-error} bounds the resulting term by
\[
c_\lambda\|P_{\varepsilon_n}(\cdot,\cdot,u_{\varepsilon_n})\|_{L^1(D_T)}\leq Cc_\lambda.
\]
The uniform mass estimate makes this error uniform in $n$.  
We can therefore take the limits in the order $n\to\infty$, then $\lambda\downarrow0$, and finally $k\to\infty$.
\end{rem}

\begin{prop}[Limit as $\varepsilon\downarrow0$]\label{prop:second-penalty-limit}
Let Assumptions \ref{assu:existence-beta}--\ref{assu:existence-obstacle} hold. 
Let $\varepsilon_n$, $u$, $\nu_1$, and $\nu_2$ be fixed by \eqref{eq:second-limit-convergences} and \eqref{eq:def for nu2}.
Then $u\geq\psi$ a.e. in $D_T$ and
\begin{equation}
\partial_t\beta(u) = \operatorname{div}A(u,\nabla u)+f(x,t,u)+\nu_1+\nu_2 \quad\text{in }\mathcal D'(D_T). \label{eq:limit-distribution-equation}
\end{equation}
In addition,
\begin{equation}
0\leq\nu_1\leq r_{\psi_1}^+,\qquad \nu_1\bigl(\beta(u)-\beta(\psi_1)\bigr)=0,\qquad \nu_1=0\quad\text{a.e. on }\{\psi_1<\psi\}. \label{eq:nu1-final-properties}
\end{equation}
\end{prop}

\begin{proof}
The obstacle inequality is \eqref{eq:second-limit-obstacle-order}.
The bound $0\leq\nu_1\leq r_{\psi_1}^+$ follows from \eqref{eq:nu1-epsilon-bound} and weak-* convergence.
Furthermore, we have
\[
\alpha_\varepsilon=\bigl(r_{\psi_1}-P_\varepsilon(\cdot,\cdot,\psi_1)\bigr)^+\longrightarrow0\quad\text{a.e. on }\{\psi_1<\psi\}.
\]
Since $0\leq\alpha_\varepsilon\leq r_{\psi_1}^{+}\in L^\infty(D_T)$, the dominated convergence theorem gives
\[
\alpha_{\varepsilon_n}\one_{\{\psi_1<\psi\}}\longrightarrow0\quad\text{strongly in }L^1(D_T).
\]
Using $0\leq\nu_{1,\varepsilon_n}\leq\alpha_{\varepsilon_n}$ and testing with $\one_{\{\psi_1<\psi\}}$, we obtain
\[
0\leq\int_{\{\psi_1<\psi\}}\nu_1\,\dd x\,\dd t=\lim_{n\to\infty}\int_{\{\psi_1<\psi\}}\nu_{1,\varepsilon_n}\,\dd x\,\dd t\leq\lim_{n\to\infty}\int_{\{\psi_1<\psi\}}\alpha_{\varepsilon_n}\,\dd x\,\dd t=0.
\]
Thus
\[
\nu_1=0\quad\text{a.e. on }\{\psi_1<\psi\}.
\]
The contact identity follows from \eqref{eq:second-limit-contact}, or equivalently from its formulation in terms of $\beta$.

Lemma \ref{lem:minty-obstacle-reactions} gives $w=A(u,\nabla u)$.
Taking $\varepsilon=\varepsilon_n$ in \eqref{eq:epsilon-equation}, passing to the limit $n\to\infty$ and using \eqref{eq:second-limit-convergences}, we have \eqref{eq:limit-distribution-equation}.
\end{proof}

\subsection{Passage to the entropy inequality and proof of existence}\label{subsec:entropy-passage}
\begin{lem}[Convergence of the truncation defect measures]\label{lem:defect-measure-convergence}
Let Assumptions \ref{assu:existence-beta}--\ref{assu:existence-obstacle} hold.
For every $K>0$, $\varepsilon\in(0,1)$, and $l>M_{\mathrm{obs}}+1$, define
\begin{align}
\dd\mu_{l,K}^{u_\varepsilon}&:=K\|h_l'\|_{L^\infty(\mathbb R)}\bigl(A(u_\varepsilon,\nabla u_\varepsilon)-A(u_\varepsilon,0)\bigr)\cdot\nabla u_\varepsilon\one_{\{l<|u_\varepsilon|<l+1\}}\,\dd x\,\dd t,\nonumber\\
\dd\mu_{l,K}^{u}&:=K\|h_l'\|_{L^\infty(\mathbb R)}\bigl(A(u,\nabla u)-A(u,0)\bigr)\cdot\nabla u\one_{\{l<|u|<l+1\}}\,\dd x\,\dd t.\label{eq:defect-measures-in-lemma}
\end{align}
Then these are nonnegative finite Radon measures and
\[
\|\mu_{l,K}^{u_{\varepsilon_n}}-\mu_{l,K}^u\|_{\mathcal M(D_T)}\longrightarrow0.
\]
\end{lem}
\begin{proof}
Set
\[
\mathfrak A_n:=A(u_{\varepsilon_n},\nabla u_{\varepsilon_n})-A(u_{\varepsilon_n},0),\qquad\mathfrak A:=A(u,\nabla u)-A(u,0).
\]
The continuity and growth of $A$, together with Lemma \ref{lem:minty-obstacle-reactions} and the strong $L^1(D_T)$ convergence of $B(u_{\varepsilon_n})$, give
\[
\mathfrak A_n\to\mathfrak A\quad\text{strongly in }L^{p'}(D_T)^d.
\]
Consequently,
\[
\mathfrak A_n\cdot\nabla u_{\varepsilon_n}\longrightarrow\mathfrak A\cdot\nabla u\quad\text{strongly in }L^1(D_T).
\]
Outside $\{|u|=l\}\cup\{|u|=l+1\}$, the indicators of the two truncations converge almost everywhere.
Stampacchia's lemma (see, for example, \cite[Chapter 5, Exercises 17,18]{evans1998partial}) gives $\mathfrak A\cdot\nabla u=0$ almost everywhere on these level sets.
Therefore, we have
\begin{align*}
&\|(\mathfrak A_n\cdot\nabla u_{\varepsilon_n})\one_{\{l<|u_{\varepsilon_n}|<l+1\}}-(\mathfrak A\cdot\nabla u)\one_{\{l<|u|<l+1\}}\|_{L^1(D_T)}\\
&\leq\|\mathfrak A_n\cdot\nabla u_{\varepsilon_n}-\mathfrak A\cdot\nabla u\|_{L^1(D_T)}\\
&\quad+\|(\mathfrak A\cdot\nabla u)[\one_{\{l<|u_{\varepsilon_n}|<l+1\}}-\one_{\{l<|u|<l+1\}}]\|_{L^1(D_T)}\longrightarrow0,
\end{align*}
which completes the proof.
\end{proof}

Since these measures are absolutely continuous in $D_T$, their extensions by zero to $\overline D\times[0,T]$ do not create any lateral-boundary or terminal-time mass.

\begin{proof}[Proof of Theorem \ref{thm:existence-general-time-obstacle}]
Fix $K>0$, $l>M_{\mathrm{obs}}+1$, an entropy $\eta\in\mathcal E$ with $\|\eta\|_{L^\infty(\mathbb R)}\leq K$, and a test function $\phi$ admissible for $\eta$.
Apply the time-chain rule of \cite[Lemma 1.4]{carrillo1999uniqueness} to the Alt--Luckhaus weak solution $u_\varepsilon$ of \eqref{eq:epsilon-equation}. 
We obtain
\begin{align}
-\int_{D_T} \partial_t\phi \int_\xi^{u_\varepsilon} \eta(r)h_l(r)\,\dd\beta(r)\,\dd x\,\dd t &= -\int_{D_T} A(u_\varepsilon,\nabla u_\varepsilon) \cdot \nabla \bigl(\eta(u_\varepsilon)h_l(u_\varepsilon)\phi \bigr)\,\dd x\,\dd t \nonumber\\
&\quad+ \int_{D_T} f(x,t,u_\varepsilon) \eta(u_\varepsilon)h_l(u_\varepsilon)\phi\,\dd x\,\dd t\nonumber\\
&\quad+ \int_{D_T}\bigl(      \nu_{1,\varepsilon}+P_\varepsilon(x,t,u_\varepsilon) \bigr)\eta(u_\varepsilon)h_l(u_\varepsilon)\phi\,\dd x\,\dd t. \label{eq:epsilon-renormalized-identity}
\end{align}

For the diffusion term, write
\begin{align*}
-A(u_\varepsilon,\nabla u_\varepsilon) \cdot \nabla \bigl(  \eta(u_\varepsilon)h_l(u_\varepsilon)\phi \bigr) &= -h_l(u_\varepsilon) A(u_\varepsilon,\nabla u_\varepsilon) \cdot \nabla \bigl(  \eta(u_\varepsilon)\phi \bigr)\\
&\quad-\eta(u_\varepsilon)h_l'(u_\varepsilon)\phi \bigl(  A(u_\varepsilon,\nabla u_\varepsilon)  -  A(u_\varepsilon,0) \bigr) \cdot\nabla u_\varepsilon\\
&\quad-\eta(u_\varepsilon)h_l'(u_\varepsilon)\phi A(u_\varepsilon,0)\cdot\nabla u_\varepsilon.
\end{align*}
The second term on the right-hand side is bounded above by $\phi\,\mu_{l,K}^{u_\varepsilon}$.
For the last term, the spatial chain rule gives
\[
-\int_{D_T} \eta(u_\varepsilon)h_l'(u_\varepsilon)\phi A(u_\varepsilon,0)\cdot\nabla u_\varepsilon\,\dd x\,\dd t = \int_{D_T} \nabla\phi\cdot \int_0^{u_\varepsilon} A(r,0)\eta(r)h_l'(r)\,\dd r\,\dd x\,\dd t.
\]
For the first term and fixed $l$, the factors involving $\eta$ and $\eta'$ are evaluated only on the bounded state range limited by $h_l$ and $h_l'$. 
The strong convergence of $u_{\varepsilon_n}$ in $L^p(0,T;\mathbb V)$ and the strong convergence of the flux in $L^{p'}(D_T)^d$ therefore give the convergence of the first diffusion term.

On the support of $P_\varepsilon(x,t,u_\varepsilon)$, we have $\psi_1\leq u_\varepsilon<\psi$.
Since $l>M_{\mathrm{obs}}+1$, one has $h_l(u_\varepsilon)=1$ on this set, and the monotonicity of $\eta$ gives
\[
P_\varepsilon(x,t,u_\varepsilon) \eta(u_\varepsilon)h_l(u_\varepsilon) \leq P_\varepsilon(x,t,u_\varepsilon)\eta(\psi).
\]
Furthermore, note that $\nu_{1,\varepsilon}=0$ almost everywhere on $\{u_\varepsilon>\psi_1\}$, and $u_\varepsilon\geq\psi_1$. 
Hence
\[
\nu_{1,\varepsilon} \eta(u_\varepsilon)h_l(u_\varepsilon) = \nu_{1,\varepsilon}\eta(\psi_1).
\]

We now let $\varepsilon=\varepsilon_n\downarrow0$.
Lemma \ref{lem:minty-obstacle-reactions} gives
\[
u_{\varepsilon_n}\to u \quad\text{strongly in }L^p(0,T;\mathbb V),
\]
and
\[
A(u_{\varepsilon_n},\nabla u_{\varepsilon_n}) \to A(u,\nabla u) \quad\text{strongly in }L^{p'}(D_T)^d.
\]
Consequently, all diffusion terms in \eqref{eq:epsilon-renormalized-identity} converge to their corresponding limits.
The time term converges by the strong $L^1(D_T)$ convergence of $\beta(u_{\varepsilon_n})$, and the source term converges by the strong $L^{p'}(D_T)$ convergence of $f(\cdot,\cdot,u_{\varepsilon_n})$. 

It remains to pass to the $P_\varepsilon$-penalty term in the entropy inequality.
Since $\psi$, $\eta$, and $\phi$ are continuous, we have $\eta(\psi)\phi\in C(K_2\times[0,T])$.
Since $\phi$ vanishes in a neighborhood of $t=T$, the terminal part of $\overline\nu_2$ does not contribute.
Therefore, we have
\begin{align*}
&\int_{D_T}P_{\varepsilon_n}(x,t,u_{\varepsilon_n})\eta(\psi)\phi\,\dd x\,\dd t\longrightarrow\int_{(K_2\cap D)\times[0,T)}\eta(\psi)\phi\,\dd\overline\nu_2+\int_{(K_2\cap\partial D)\times[0,T)}\eta(\psi)\phi\,\dd\overline\nu_2.
\end{align*}

If $\phi$ is compactly supported in $D\times[0,T)$, then the second integral is zero.
Suppose that $\phi$ reaches the lateral boundary.
Since $\phi$ is admissible, we have $\eta(0)=0$.
Moreover, since \eqref{eq:K2-boundary-gap} holds, we have $\psi<0$ on the relevant lateral boundary.
Since $\eta$ is nondecreasing, we therefore have
\[
\eta(\psi)\leq\eta(0)=0.
\]
Since $\phi\geq0$ and $\overline\nu_2\geq0$, we obtain
\[
\int_{(K_2\cap\partial D)\times[0,T)}\eta(\psi)\phi\,\dd\overline\nu_2\leq0.
\]
Consequently, in both admissible cases, we have
\begin{align*}
\limsup_{n\to\infty}\int_{D_T}P_{\varepsilon_n}(x,t,u_{\varepsilon_n})\eta(\psi)\phi\,\dd x\,\dd t\leq\int_{D\times[0,T)}\eta(\psi)\phi\,\dd\nu_2.
\end{align*}
This is exactly the upper bound required in the passage to the renormalized entropy inequality.

Furthermore, since $\phi$ vanishes in a neighborhood of $t=T$, the reaction convergence also gives
\[
\int_{D_T}\nu_{1,\varepsilon_n}\eta(\psi_1)\phi\,\dd x\,\dd t\longrightarrow\int_{D_T}\nu_1\eta(\psi_1)\phi\,\dd x\,\dd t.
\]
By the contact property in \eqref{eq:second-limit-contact},  we have $\nu_1=0$ a.e. on $\{u>\psi_1\}$.
Since $u\geq\psi$ and therefore $\nu_1=0$ on $\{\psi_1<\psi\}$, we have
\[
\nu_1\eta(\psi_1)=\nu_1\eta(\psi)\quad\text{a.e. in }D_T.
\]
For fixed $l$ and $K$, Lemma \ref{lem:defect-measure-convergence} gives the convergence of the defect measures in total variation.
Combining these convergences and passing to the limit $n\to\infty$ in \eqref{eq:epsilon-renormalized-identity} gives \eqref{eq:general-renormalized-inequality}.

It remains to verify the tail condition \eqref{eq:general-tail-condition}.
By the growth assumption on $A$, the integrability of $B(u)$, \eqref{eq:A-uniform-monotonicity}, and $\nabla u\in L^p(D_T)^d$, one has
\[
0\leq\bigl(A(u,\nabla u)-A(u,0)\bigr)\cdot\nabla u \in L^1(D_T).
\]
Hence, using \eqref{eq:admissible-cutoff-family},
\begin{align*}
\mu_{l,K}^u(D\times[0,T))&= K\|h_l'\|_{L^\infty(\mathbb R)} \int_{\{l<|u|<l+1\}}\bigl(A(u,\nabla u)-A(u,0)\bigr)\cdot\nabla u \,\dd x\,\dd t\\
&\leq K\sup_{j>0}\|h_j'\|_{L^\infty(\mathbb R)} \int_{\{|u|>l\}}\bigl(A(u,\nabla u)-A(u,0)\bigr)\cdot\nabla u \,\dd x\,\dd t.
\end{align*}
Since $\one_{\{|u|>l\}}\to0$ a.e. in $D_T$, the dominated convergence theorem gives
\[
\mu_{l,K}^u(D\times[0,T)) \longrightarrow0 \qquad\text{as }l\to\infty.
\]
Thus \eqref{eq:general-tail-condition} holds.
Finally, the compactness argument gives
\[
u\in L^p(0,T;W_0^{1,p}(D)),\qquad B(u)\in L^\infty(0,T;L^1(D)).
\]
Moreover, $\nu=\nu_1\,\dd x\,\dd t+\nu_2$ is a finite nonnegative Radon measure on $D\times[0,T)$.
Together with \eqref{eq:second-limit-obstacle-order}, this verifies items (i)--(iii) of
Definition \ref{def:general-renormalized-obstacle-solution}.
This completes the proof.
\end{proof}

\begin{cor}[Well-posedness in the common class]\label{cor:wellposedness-common-class}
Assume the hypotheses of Theorem \ref{thm:existence-general-time-obstacle}.
Assume in addition that $A$ is independent of its first variable, $A(0)=0$, and that $\psi$ is independent of time.
Then the solution constructed by Theorem \ref{thm:existence-general-time-obstacle} is the unique renormalized entropy solution with the prescribed data.
\end{cor}

\begin{proof}
Under these additional assumptions, the solution constructed in Theorem \ref{thm:existence-general-time-obstacle} belongs to the class of Definition \ref{def:entropy solution with ob}.
Indeed, \eqref{eq:beta-controlled-by-energy} and $B(u)\in L^\infty(0,T;L^1(D))$ give $\beta(u)\in L^\infty(0,T;L^1(D))$, while $A(r,0)=0$ makes the additional term in \eqref{eq:general-renormalized-inequality} vanish.
It remains to verify the boundary condition for the static obstacle.
On $K_2\cap\partial D$, \eqref{eq:K2-boundary-gap} gives
$\psi<0$.
If $x_0\in\partial D\setminus K_2$, then, since $K_2$ is closed, $\psi_2=0$ a.e. in a neighborhood of $x_0$ inside $D$.
Hence $\psi=\psi_1$ there in the trace sense, and $\operatorname{Tr}\psi_1\leq0$ gives $\psi(x_0)\leq0$ by the continuity of $\psi$.
Therefore $\psi\leq0$ on $\partial D$, so Assumption \ref{assu:assumption for barrier} is satisfied.
Condition \eqref{eq:fhat-Lipschitz} is exactly the Lipschitz condition with respect to $\beta(u)$ used in Theorem \ref{thm:main-comparison}.
Since $\beta$ is strictly increasing, Theorem \ref{thm:main-comparison} gives the uniqueness of $u$ and of $\nu$.
\end{proof}

\medskip
\noindent\textbf{Acknowledgments.}
The author would like to thank Professor Kai Du, Shanghai Center for Mathematical Sciences at Fudan University, for helpful discussions.
\par\medskip

\appendix

\section{The space--time dependent source in the Alt--Luckhaus scheme}

Alt and Luckhaus already observe in \cite[Section 1.10]{alt1983quasilinear} that the lower-order term may depend on $x$ and $t$.
In the time-discrete construction, the time-dependent coefficient is replaced by a piecewise constant time average.
We record the approximation property needed to apply that observation under the Carath\'eodory and energy-growth assumptions used in this paper.

Let $N\in\mathbb N$, let $h=T/N$, and set
\[
    I_k:=((k-1)h,kh],\qquad k=1,\ldots,N.
\]
For $t\in I_k$, define
\begin{equation*}
    f_h(x,t,r):=\frac1h\int_{I_k}f(x,s,r)\,\dd s,\qquad (f_0)_h(x,t):=\frac1h\int_{I_k}f_0(x,s)\,\dd s. 
\end{equation*}

\begin{lem}[Space--time source approximation]\label{lem:AL-space-time-source}
Let Assumption \ref{assu:existence-beta} hold. 
Suppose that $f:D\times[0,T]\times\mathbb R\longrightarrow\mathbb R$ is Carath\'eodory and that
\begin{equation}
    |f(x,t,r)|\leq f_0(x,t)+c_fB(r)^{1/p'}, \qquad f_0\in L^{p'}(D_T),\quad f_0\geq0. \label{eq:appendix-source-growth}
\end{equation}
Then the following statements hold.

\begin{enumerate}
\item The function $f_h$ is Carath\'eodory and
\begin{equation}
    |f_h(x,t,r)|\leq(f_0)_h(x,t)+c_fB(r)^{1/p'}.\label{eq:appendix-averaged-growth}
\end{equation}
Moreover,
\begin{equation}
\|(f_0)_h\|_{L^{p'}(D_T)}\leq\|f_0\|_{L^{p'}(D_T)}. \label{eq:appendix-average-contraction}
\end{equation}

\item For every $R>0$,
\begin{equation}
    \Big\|\sup_{|r|\leq R}|f_h(\cdot,\cdot,r)-f(\cdot,\cdot,r)|\Big\|_{L^{p'}(D_T)}\longrightarrow0\qquad\text{as }h\downarrow0. \label{eq:appendix-local-uniform-convergence}
\end{equation}

\item
Let $v_h\to v$ in measure in $D\times[0,T]$ and suppose that
\begin{equation}
    \sup_h\int_{D_T}B(v_h)\,\dd x\,\dd t<\infty. \label{eq:appendix-energy-bound}
\end{equation}
Then
\begin{equation}
    f_h(\cdot,\cdot,v_h)\longrightarrow f(\cdot,\cdot,v)\quad\text{in measure in }D\times[0,T], \label{eq:appendix-composition-measure}
\end{equation}
and
\begin{equation}
    \sup_h\|f_h(\cdot,\cdot,v_h)\|_{L^{p'}(D_T)}<\infty.\label{eq:appendix-composition-bound}
\end{equation}
Consequently,
\begin{equation}
    f_h(\cdot,\cdot,v_h)\rightharpoonup f(\cdot,\cdot,v)\quad\text{weakly in }L^{p'}(D_T). \label{eq:appendix-composition-weak}
\end{equation}
\end{enumerate}
\end{lem}

\begin{proof}
For fixed $r$, the function $(x,t)\mapsto f_h(x,t,r)$ is measurable.
Fix $R>0$.
For almost every $x$, the function $s\mapsto f_0(x,s)$ is integrable on $(0,T)$.
If $r_j\to r$ and $|r_j|,|r|\leq R$, then
\[
f(x,s,r_j)\longrightarrow f(x,s,r)
\]
for almost every $s$, while
\[
|f(x,s,r_j)|\leq f_0(x,s)+c_f\sup_{|q|\leq R}B(q)^{1/p'}.
\]
Dominated convergence theorem therefore shows that $r\mapsto f_h(x,t,r)$ is continuous.
Hence $f_h$ is Carath\'eodory.

Averaging \eqref{eq:appendix-source-growth} gives \eqref{eq:appendix-averaged-growth}.
Jensen's inequality gives
\[
    |(f_0)_h(x,t)|^{p'}\leq\frac1h\int_{I_k}|f_0(x,s)|^{p'}\,\dd s\qquad (t\in I_k).
\]
After integration over $D_T$, this proves \eqref{eq:appendix-average-contraction}.

We next prove \eqref{eq:appendix-local-uniform-convergence}.
For fixed $R>0$, define
\[
\mathcal F_R(x,t):=f(x,t,\cdot)
\]
as a function with values in the separable Banach space $C([-R,R])$.
The Carath\'eodory property implies that $\mathcal F_R$ is strongly measurable.
Moreover,
\[
\|\mathcal F_R(x,t)\|_{C([-R,R])}\leq f_0(x,t)+c_f\sup_{|q|\leq R}B(q)^{1/p'}.
\]
Thus
\[
\mathcal F_R\in L^{p'}\bigl(D_T;C([-R,R])\bigr).
\]
Piecewise constant time averages converge strongly to the original function in this Bochner space.
Hence
\[
\|(\mathcal F_R)_h-\mathcal F_R\|_{L^{p'}(D_T;C([-R,R]))}\longrightarrow0,
\]
which is precisely \eqref{eq:appendix-local-uniform-convergence}.

We now prove the composition statement.
Note that $B(r)/|\beta(r)|$ tends to infinity as $|r|\to\infty$. Then, we have
\[
\inf_{|q|\geq R}B(q)\longrightarrow\infty\qquad\text{as }R\to\infty.
\]
Then, we have
\[
\bigl|\{|v_h|\geq R\}\bigr|\leq\frac{1}{\inf_{|q|\geq R}B(q)}\int_{D_T}B(v_h)\,\dd x\,\dd t,
\]
uniformly in $h$.
Fatou's lemma gives $B(v)\in L^1(D_T)$, and therefore the same estimate holds for $v$.

On
\[
E_{h,R}:=\{|v_h|\leq R,\ |v|\leq R\},
\]
we have
\begin{align*}
|f_h(x,t,v_h)-f(x,t,v)|&\leq\sup_{|q|\leq R}|f_h(x,t,q)-f(x,t,q)|\nonumber\\
&+|f(x,t,v_h)-f(x,t,v)|.
\end{align*}
The first term converges to zero in measure by \eqref{eq:appendix-local-uniform-convergence}.
For the second term, let $(h_j)$ be an arbitrary subsequence. 
Since $v_{h_j}\to v$ in measure, a further subsequence converges to $v$ almost everywhere in $D_T$. 
The Carath\'eodory property then gives $f(x,t,v_{h_j})\to f(x,t,v)$ a.e. along that further subsequence. 
Hence $f(\cdot,\cdot,v_h)\to f(\cdot,\cdot,v)$ in measure on every bounded state region.

Since the complement of $E_{h,R}$ has uniformly small measure for large $R$, this proves \eqref{eq:appendix-composition-measure}.

Finally, \eqref{eq:appendix-averaged-growth} gives
\begin{equation*}
|f_h(x,t,v_h)|^{p'}\leq C\bigl((f_0)_h^{p'}+B(v_h)\bigr).
\end{equation*}
Equations \eqref{eq:appendix-average-contraction} and \eqref{eq:appendix-energy-bound} therefore imply \eqref{eq:appendix-composition-bound}.

Since $p'>1$, the sequence $f_h(\cdot,\cdot,v_h)$ is weakly relatively compact in $L^{p'}(D_T)$.
Every weakly convergent subsequence has, by \eqref{eq:appendix-composition-measure}, the same limit $f(\cdot,\cdot,v)$.
This proves \eqref{eq:appendix-composition-weak}.
\end{proof}

\bibliographystyle{plain}
\bibliography{bi_PDE}

\end{document}